\documentclass[a4paper,11pt]{article}

\usepackage{float}

\usepackage{amsmath}
\usepackage{amssymb}
\usepackage{theorem}
\usepackage{pstricks}
\usepackage{euscript}
\usepackage{epic,eepic}
\usepackage{graphicx}
\PassOptionsToPackage{normalem}{ulem}
\usepackage{hyperref}
\hypersetup{colorlinks, citecolor=blue, filecolor=black, linkcolor=blue, urlcolor=blue}

\usepackage{enumitem}
\usepackage[misc]{ifsym}

\PassOptionsToPackage{normalem}{ulem}
\newcommand{\N}{\mathbb{N}}

\newcommand{\R}{\mathbb{R}}
\newcommand{\Rinf}{\mathbb{R}\cup \{+\infty \}}

\newcommand{\eps}{\varepsilon}
\newcommand{\KL}{\mathrm{KL}}
\newcommand{\kl}{\mathrm{kl}}

\DeclareMathOperator{\argmin}{argmin}
\newcommand{\dom}{\mathop\mathrm{\rm dom}}

\newcommand{\inte}{\mbox{\rm int}}
\DeclareMathOperator{\im}{Im}

\newcommand{\prox}{\mbox{\rm prox}}

\newcommand{\Id}{\operatorname{Id}}

\newcommand{\minimize}[2]{\ensuremath{\underset{\substack{{#1}}}%
{\text{\rm minimize}}\;\;#2 }}

\renewcommand{\liminf}[1][]{ \underset{#1}{\mbox{\rm liminf}} \ }

\newcommand{\nin}{{n\in\N}}

\renewcommand{\d}{{\rm d}}
\newcommand{\dt}{{\rm d}t}

\DeclareFontEncoding{FMS}{}{}
\DeclareFontSubstitution{FMS}{futm}{m}{n}
\DeclareFontEncoding{FMX}{}{}
\DeclareFontSubstitution{FMX}{futm}{m}{n}
\DeclareSymbolFont{fouriersymbols}{FMS}{futm}{m}{n}
\DeclareSymbolFont{fourierlargesymbols}{FMX}{futm}{m}{n}
\DeclareMathDelimiter{\VERT}{\mathord}{fouriersymbols}{152}{fourierlargesymbols}{147}

\newcommand{\longlongrightarrow}{} 
\DeclareRobustCommand{\longlongrightarrow}{\relbar\joinrel\longrightarrow}

\newlist{myitemize}{itemize}{3}
\setlist[myitemize,1]{leftmargin=1.15cm}
\setlist[itemize]{label=$\bullet$,leftmargin=0.35cm}

\newtheorem{theorem}{Theorem}[section]
\newtheorem{lemma}[theorem]{Lemma}

\newtheorem{proposition}[theorem]{Proposition}
\newtheorem{definition}[theorem]{Definition}
\theoremstyle{plain}{\theorembodyfont{\rmfamily}
}
\theoremstyle{plain}{\theorembodyfont{\rmfamily}
}
\theoremstyle{plain}{\theorembodyfont{\rmfamily}
}
\theoremstyle{plain}{\theorembodyfont{\rmfamily}
\newtheorem{example}[theorem]{Example}}
\theoremstyle{plain}{\theorembodyfont{\rmfamily}
}
\theoremstyle{plain}{\theorembodyfont{\rmfamily}
\newtheorem{remark}[theorem]{Remark}}
\theoremstyle{plain}{\theorembodyfont{\rmfamily}
}

\definecolor{labelkey}{rgb}{0,0.08,0.45}
\definecolor{refkey}{rgb}{0,0.6,0.0}
\definecolor{Brown}{rgb}{0.45,0.0,0.05}
\definecolor{dgreen}{rgb}{0.00,0.49,0.00}
\definecolor{dblue}{rgb}{0,0.08,0.75}
\numberwithin{equation}{section}

\begin{document}
\title{\sffamily 
Iterative regularization via  dual diagonal descent
\thanks{This material is based upon work supported by the Center for Brains, Minds and Machines (CBMM), funded by NSF STC award CCF-1231216.
 The research of Guillaume Garrigos was partially supported by the Air Force Office of Scientific Research, 
 Air Force Material Command, USAF, under grant number F49550-1 5-1-0500 .L. Rosasco acknowledges the financial support of the Italian Ministry of Education, University and Research FIRB project RBFR12M3AC.
S. Villa is member of the Gruppo Nazionale per
l'Analisi Matematica, la Probabilit\`a e le loro Applicazioni (GNAMPA)
of the Istituto Nazionale di Alta Matematica (INdAM). }}
\author{Guillaume Garrigos$^{1}$, Lorenzo Rosasco$^{1,2}$, and Silvia Villa$^3$
\\[5mm]
\small
\small $\!^1$ LCSL, Istituto Italiano di Tecnologia and Massachusetts Institute of Technology,\\
\small Bldg. 46-5155, 77 Massachusetts Avenue, Cambridge, MA 02139, USA\\
\small \ttfamily{guillaume.garrigos@iit.it}
\\[5mm]
\small
\small $\!^2$ DIBRIS, Universit\`a degli Studi di Genova\\
\small  Via Dodecaneso 35, 16146, Genova, Italy\\
\small \ttfamily{lrosasco@mit.edu}
\\[5mm]
\small
\small $\!^3$ Dipartimento di Matematica, Politecnico di Milano\\
\small Via Bonardi 9, 20133 Milano, Italy\\
\small \ttfamily{silvia.villa@polimi.it}
}
\date{~}

\maketitle

\begin{abstract}
In the context of linear inverse problems, we propose and study a general iterative regularization method 
allowing to consider large classes of regularizers and data-fit terms. 
The algorithm we propose is based on a primal-dual diagonal {descent} method. 
Our analysis establishes convergence  as well as stability results. 
Theoretical findings are complemented with  numerical experiments showing state of the art performances. 
\end{abstract}

\begin{small}

\end{small}

\section{Introduction}

Many applied problems in science and engineering can be modeled as noisy inverse problems. 
{This is  true in particular for many problems  in image processing, such as image denoising, image deblurring, image segmentation, or inpainting.}
Tackling these problems requires to deal with their possible ill-posedeness \cite{EngHanNeu96} and to devise efficient numerical procedures to quickly and accurately compute a solution.

Tikhonov regularization is a classical approach to restore well-posedness \cite{DonZol93}. 
A stable solution is defined by the minimization of an objective function being the sum of two terms: a data-fit term 
and a regularizer  ensuring stability. 
From a numerical perspective, first order methods
have recently become popular to solve the corresponding optimization problem~\cite{ComPes11}.
Indeed, simplicity and low iteration cost make these methods especially suitable in large scale applications.

In practice, finding the best Tikhonov regularized solution  requires specifying a regularization parameter  determining the trade-off between data-fit and stability.
Discrepancy principles \cite{EngHanNeu96}, SURE \cite{Ste81,DelVaiFad14}, and cross-validation \cite{SteChr08} are some  of the methods used to this purpose.
An observation important for our work is that, from a numerical perspective, choosing the regularization parameter for Tikhonov regularization typically requires solving not one, but several optimization problems, i.e. one for each regularization parameter to be tried.  Clearly, this can dramatically increase the computational costs to find a good solution, and the question of how to keep accuracy while ensuring better numerical complexity is a main motivation for our study.

 In this paper, we depart from Tikhonov regularization and consider iterative regularization approaches \cite{BakKok04}. The latter are classical  regularization techniques  based on the observation that stopping an iterative procedure corresponding to the minimization of an empirical objective has a self-regularizing property \cite{EngHanNeu96}. 
 Crucially, the number of iterations becomes  the regularization parameter, and hence  controls at the same time the stability of the solution as well as the computational complexity of the method. 
This property makes parameter tuning numerically efficient and iterative regularization an alternative to Tikhonov regularization, which potentially alleviates the  aforementioned drawbacks.
{Indeed, an advantage of iterative regularization strategies is that they are developed in conjunction with the optimization algorithm, which is tailored to the structure of 
the problem of interest.}

 Iterative regularization methods are classical  both in linear \cite{EngHanNeu96} and non linear inverse problems \cite{BakKok04,KalNeuSch08}, for quadratic data-fit term and quadratic   regularizers. Extensions to more general regularizers have been considered in recent works \cite{BacBur09,BurOsh13,BurResHe07,BotHei12}.
 However,  
 we are not aware of iterative regularization methods that allow  considering more general data-fit terms. 
 Indeed,  while this is easily done  in Tikhonov regularization,  how to do the same in iterative regularization is less clear and our study provides an answer.  

Our starting point is viewing  the inverse problem as a hierarchical optimization problem  defined by the  regularizer and the data-fit term. 
The latter can belong to  wide classes of  convex,  but possibly non-smooth functionals.   
To solve such an optimization problem 
we combine duality techniques \cite{ComDunVu10} with a diagonal approach \cite{BahLem94}. As a result, we obtain a primal-dual method,   
given by a diagonal forward-backward algorithm on the dual problem. The algorithm thus obtained is simple and easy to implement.  Our main result proves  convergence in the noiseless case,  and is an optimization result interesting in its own right. Combining this result with a stability analysis  allows to derive  iterative regularization properties of the method.  Our theoretical analysis is complemented with numerical results comparing the proposed method with Tikhonov regularization on various imaging problems. 
The obtained results show that our approach is  competitive in terms of accuracy and often outperforming 
Tikhonov regularization from a numerical perspective. To the best of our knowledge, our analysis is the first  
study on iterative regularization methods for  general data-fit terms and hence it is a step towards broadening the applicability and practical impact of these techniques.

The rest of the paper is organized as follows. In Section~\ref{S:Background and notation} we collect some technical definitions and results needed in the rest 
of the paper, whereas in Section~\ref{S:diagonal regularization} we recall the basic ideas in inverse problems and regularization theory.
In Section~\ref{S:the 3D method} we introduce the algorithm we propose in this paper and present in Section \ref{S:main results} its regularization properties, which constitutes our main results.
The theoretical analysis of the (3-D) method is made in Sections~\ref{S:Convergence} and \ref{S:regularization}, while Section~\ref{S:numerical} contains its numerical study. The Appendix
contains the proof of some auxiliary and technical results.

\section{Background and notation}\label{S:Background and notation}

We give here some mathematical background  needed in the paper. 
We refer to \cite{BauCom,Pey} for an account of the main results in convex analysis.

Since our algorithm will essentially rely on duality arguments, we first introduce 
the notion of (Fenchel) conjugate. Let $H$ be  a Hilbert space, $2^{H}$ its power set, and  
$f : H \longrightarrow [-\infty,+\infty]$. Its {\em Fenchel conjugate} $f^* : H \longrightarrow [-\infty,+\infty]$ is 
\[
(\forall x\in H)\quad f^*(x):=  \sup_{x'\in H}  \left\{\langle x',x\rangle -f(x') \right\}.
\]
We say that $f$ is {\em coercive} if $\lim_{\|x\|\to+\infty} f(x)=+\infty$.
We  denote by $\Gamma_0(H)$ the set of proper, convex and lower semi-continuous functions from $H$ to $]-\infty,+\infty]$.
Let  $\sigma \in\left]0,+\infty\right[$. We say that  $f \in \Gamma_0(H)$ is $\sigma$-{\em strongly convex}  if 
$f-{\sigma}\Vert \cdot \Vert^2/2 \in \Gamma_0(H)$.
The Fenchel conjugate of a $\sigma$-strongly convex function is differentiable, with a $\sigma^{-1}$-Lipschitz 
continuous gradient \cite[Theorem 18.15]{BauCom}. 
We recall that the {\em subdifferential} of  $f \in \Gamma_0(H)$ is the operator 
$\partial f : H \rightarrow 2^{H}$  defined by, for every $x\in H$,
\[
x^* \in \partial f(x)\Leftrightarrow (\forall x' \in H)\ \  \ f(x') - f(x) - \langle x^* , x' - x \rangle \geq 0.
\]
If $f$ is Fr\'echet differentiable at $x \in H$, then $\partial f(x)=\{\nabla f(x) \}$.
The subdifferential also enjoys a symmetry property with respect to the Fenchel 
conjugation~\cite[Theorem 16.23]{BauCom}: 
\[
(\forall (x,x^*)\in H^2)\quad x^* \in \partial f(x) \Leftrightarrow x \in \partial f^*(x^*).
\] 

Given two functions $f,g \in \Gamma_0(H)$, their {\em infimal convolution} 
(or inf-convolution) is the function $f \# g$ in  $\Gamma_0(H)$ defined by
\begin{equation*}
 (\forall x\in H)\quad(f \# g) (x):=\inf\limits_{x' \in H}  \left\{ f(x') + g(x-x') \right\}.
\end{equation*}
The Fenchel conjugate of the infimal convolution of two functions can be 
simply computed by the following rule \cite[Proposition 13.21(i)]{BauCom}
\begin{equation}
\label{e:fenchelinfc}
 (f\# g)^*=f^*+g^*.
\end{equation}

We also recall the notion of proximity operator, which is a key tool to define the algorithm we study.
The {\em proximity operator} of $f \in \Gamma_0(H)$ is the operator $\prox_f : H \longrightarrow H$ defined by,
for every $x\in H$,
\begin{equation}\label{D:proximal operator}
\prox_{f}(x)=\underset{x' \in H}{\argmin} \left\{f(x')+\frac{1}{2}\|x'-x\|^2 \right\} .
\end{equation}
The proximity operator is particularly relevant when computing the gradient of the conjugate of 
a strongly convex function.
\begin{lemma}
\label{L:gradient dual prox formula}
Let $H_1,H_2$ be Hilbert spaces. Let $J \in \Gamma_0(H_2)$, $\sigma\in\left]0,+\infty\right[$, $x' \in H_1$, and $W\colon H_1\to H_2$
be a linear orthogonal operator. Let $f \in \Gamma_0(H_1)$ be the function defined by $f(x) = J(W x) + {\sigma} \Vert x  - x'\Vert^2/2.$ 
Then, 
$$(\forall x\in H_1) \quad\nabla f^*(x) = W^*\prox_{\sigma^{-1} J}(Wx' + \sigma^{-1}Wx).$$
\end{lemma}
The proof is postponed to  Appendix~\ref{S:Annex1}.

We end this section by introducing the notion of conditioning~\cite{Vai70,Zol78,Zal},  
which is a common tool in the optimization and regularization literature \cite{Lem98,BolNguPeySut15},
and will be required later for the data-fit function. 
\begin{definition}\label{D:conditioning}
Let $f\in\Gamma_0(H)$ having a unique minimizer $x_0\in H$. 
The function $f$ is said to be {\em well conditioned} if there exists a positive even function $m\in\Gamma_0(\R)$
such that, for every $t\in\R$, $m(t)=0\implies t=0$, and 
\begin{equation}
\label{E:growth condition}
(\forall x \in H)\quad  m\left( \Vert x -x_0\Vert \right) \leq f(x)- f(x_0).
\end{equation}
In that case, $m$ is called a {\em conditioning modulus} (or growth modulus) for $f$. 
Let $p\in\left[1,+\infty\right[$. We  say that $f$ is $p$-{\em well conditioned},
if there exists $(\varepsilon,\gamma)\in\left(\left]0,+\infty\right[\right)^2$
such that
\begin{equation}
\label{E:p_well} 
\left(\forall t\in\left]-\varepsilon,\varepsilon\right[\right)\quad m(t)\geq \frac{\gamma}{p}|t|^p
\end{equation}
\end{definition}
\paragraph{Notation.}
We adopt the following standard notation: $\mathbb{R}_+=\left[0,+\infty\right[$, 
$\mathbb{R}_{++}=\left]0,+\infty\right[$,  $\mathbb{R}^d_{++}=\left]0,+\infty\right[^d$,
  $\mathbb{N}^*=\N\setminus\{0\}$. The identity operator from a set to itself is denoted by $\Id$.
Given $C$ a subset of a topological space, $\mathrm{int} C$ denotes the interior of  $C$. 
The set of minimizers of a function $f$ is denoted by $\argmin f$, and its domain is noted $\dom f$.
The range of a linear operator $A : H_1 \rightarrow H_2$ will be denoted by  $\im A$, and its norm 
is denoted by $\|A\|$. The norms in the considered Hilbert spaces are always denoted by $\|\cdot\|$.
 
\section{Background: inverse problems and regularization}
\label{S:diagonal regularization}
In this section, we recall  basic notions in inverse problems  theory and   introduce  Tikhonov and iterative regularization. 

\subsection{Linear inverse problems}

Let $X$ and $Y$ be Hilbert spaces. 
Given $\bar {y}\in Y$ and a bounded linear operator  $A : X \longrightarrow Y$ 
the corresponding inverse problem is to find ${\bar x}\in X$ satisfying
\begin{equation}
\label{e:ideal}
 {\bar{y}}=A {\bar{x}}.
\end{equation}
{For example, in denoising  $A=Id$ and in deblurring $A$ is an integral operator for suitable kernel}.
In general, the above problem is {\em ill-posed}, in the sense that a solution might not exist, might not be unique, or might not depend continuously to the 
data $\bar{y}$ \cite{EngHanNeu96}. The first step towards restoring well-posedness is then to introduce a notion of generalized solution. 
This latter definition  hinges on the choice of a regularizer, that is a functional $R \in \Gamma_0(X)$ and a data-fit function 
$D: Y^2 \longrightarrow \Rinf$. A generalized solution $x^\dagger\in X$  is then defined 
as a solution of the problem
\begin{equation}
\label{e:P}
\tag{$P$}  \text{minimize} \  \left\{ R(x)  \  \left| \ x\in\argmin_{x'\in X} D(Ax',\bar{y}) \right\} \right..
\end{equation}
Consider the following classical example to illustrate the above definition.

\begin{example}[Moore-Penrose solution]\label{ex:MP}
Let  $R(x)=\|x\|$ and  $D(Ax,y)=\|Ax-y\|^2$ for all $x\in X$. 
Then, under mild assumptions, there exists a unique generalized solution to \eqref{e:P} which is the Moore-Penrose solution $x^\dagger=A^\dagger \bar y$, where $A^\dagger$ 
is the pseudo-inverse of $A$ \cite{EngHanNeu96}.
\end{example}

We add two comments. 
First, note that there might be more than one generalized solution, however in the following we will restore uniqueness of $x^\dagger$ by assuming $R$ to be strongly convex.
Second, as we discuss next, in general $x^\dagger$ might not depend continuously on $\bar y$,  as it is clear from the  Example~\ref{ex:MP}. This  last observation is crucial, since in practice  only a noisy  datum  is available. 
Ensuring continuity, hence stability, to noisy data is the main motivation of regularization techniques  described in the next section. We first add a few further examples of regularizers and data-fit functions, and one remark.

 \begin{example}[Regularizers]\label{E:regularizers list}
 A choice of regularizer popular in image processing  is the 
 $\ell^1$-norm of the coefficients of $x\in X$ with respect to an orthonormal basis, or a more general dictionary, e.g. a frame. Indeed, this regularizer can be shown to correspond to a sparsity prior assumption on the solution \cite{Mal09}.
Another popular choice of regularizer is the total variation 
\cite{RudOshFat92}, due to its ability to preserve edges.
Other possibilities are total generalized variation \cite{BreKunPoc10}, or infimal convolutions
between total variation and higher order derivatives \cite{ChaLio97}. Yet another possibility is 
to consider a Huber norm of the gradient, instead of the $L^1$ norm~\cite{ChaPoc11}. We refer to  \cite{Lan16} for additional references.
\end{example}

We next discuss several examples of data-fit functions. 
Flexibility in the  choice of the latter  is  a key aspect 
for our study. 

\begin{example}[Data-fit function]
\label{R:data-fit function list}
As mentioned above,  a classical choice for the data-fit function is 
the $\ell^2$ norm:
\[(\forall (u,y)\in Y^2)\quad D(u;y)=\|u-y\|^2/2.\]
We list a few further examples.
\begin{itemize}
\item  the $\ell^1$-norm in $\mathbb{R}^d$,
\[
(\forall (u,y)\in \R^d\times \R^d)\quad D(u;y)=\|u-y\|_1;
\]
\item  the Kullback-Leibler divergence, defined , for every $ y \in\R^d$, as
$D(u;y):=\KL(y,u)=\sum_{i=1}^d \kl(y_i,u_i)$, where
\[
\kl(y_i,u_i)=\begin{cases} \displaystyle\sum_{{i=1}}^d y_{i}\log\frac{y_{i}}{u_{i}}-y_{i}+u_{i} &\text{if }  (y_i,u_i) \in\left ]0,+\infty\right[^{2} \\[2ex] 
+\infty &\text{otherwise;}\end{cases}
\]
\item the weighted sum of $L^1$ and $L^2$ norms in $\R^d$  \cite{HinLan13},
\[
(\forall (u,y)\in \R^d\times \R^d) \quad D(u;y)=\|u-y\|_1+\frac{\sigma}{2}\|u-y\|^2, 
\]
for some $\sigma \in\left]0,+\infty\right[$;
\item 
the Huber data-fit function \cite{CalReySch16} in $\R^d$.
Let $\sigma\in\R_{++}$ and    the Huber function be $h_\sigma \colon \R\to\R_{+}$,
\begin{equation}
\label{e:huber}
(\forall t\in \R)\quad h_\sigma(t)=\begin{cases}\frac{1}{2 \sigma }t^2  & \text{ if } \vert t \vert \leq \sigma \\
 \vert t \vert -\frac{\sigma}{2}& \text{ otherwise.}
\end{cases}
\end{equation}
 Then the corresponding data-fit function can be formulated as, 
$(\forall (u,y)\in \R^d\times \R^d)$,
$$
D(u;y)=H_{\sigma}(u-y):=\sum\limits_{i = 1}^d h_{\sigma}(u_i-y_i).
$$
\end{itemize}
\end{example}
Both the choice of the regularizer and the data-fit function reflect some prior information about the problem at hand. This latter observation can be further developed taking a probabilistic (Bayesian) perspective, as we recall in the next remark.

\begin{remark}[Bayesian interpretation]\label{rem:bay}
In a Bayesian framework,  the choice of regularizers and data-fit functions can be related to the choice of a prior distribution on the solution and a noise model with a corresponding likelihood. In particular, for the data fit functions it can be seen that the quadratic norm is related to Gaussian noise and moreover, 
\begin{itemize}
\item  the $L^1$-norm is related to impulse noise, e.g.  salt and pepper or random-valued impulse noise \cite{Nik02},
\item  the Kullback-Leibler divergence is related to Poisson noise \cite{LeChaAsa07},
\item the weighted sum of $L^1$ and $L^2$ norms is related to mixed Gaussian and impulse noise \cite{HinLan13},
\item  and  the Huber data-fit function \cite{CalReySch16} is related to
 the mixed Gaussian and impulse noise.
\end{itemize}
\end{remark}

\subsection{Tikhonov and iterative regularization}
The basic idea of regularization is to approximate a generalized solution $x^\dagger$ of \eqref{e:P} with a family of solutions having  better  stability properties.  More precisely, given a pair $(A, \bar y)$, a regularization method defines a sequence $(x_\lambda)_{\lambda\in\Lambda}\in X$, where 
$\Lambda = ]0, +\infty[$, or 
$\Lambda= \N$.  
The idea is  that the so called regularization parameter $\lambda$ controls the  accuracy with which $x_\lambda$ approximates $x^\dagger$. 
Indeed, the first basic regularization property  is to require 
 \begin{equation}\label{e:reg_prop}
 x_\lambda \to x^\dagger, \text{ as } \lambda \to 0,
 \end{equation}
 (or as $\lambda\to +\infty$ when $\Lambda=\N$).
 The second basic property of a regularization method is stability. 
{ Given  $\hat y $ a noisy version  of the exact datum $\bar{y}$}, this latter property can be seen as the requirement  for the 
sequence  $(\hat x_\lambda)_{\lambda\in\Lambda}\in X$, corresponding to the regularization method applied to
 $(A,\hat y)$, to be sufficiently close to  $(x_\lambda)_{\lambda\in\Lambda}\in X$.
 This latter property,  together with the regularization property, allows to show that $\hat x_\lambda$ is  a good approximation  to $x^\dagger$ -- at least provided a suitable regularization parameter choice $\lambda\in \Lambda$. 
 We refer to \cite{EngHanNeu96} for further details and illustrate the above definitions with two specific examples of regularization operators.

{\paragraph{Tikhonov regularization.}
In the setting of Example~\ref{ex:MP}, Tikhonov regularization is defined by the following minimization problem
$$
x_\lambda= \minimize{x\in X} \|x\|^2 + \frac{1}{\lambda} \|Ax-y\|^2.
$$ 
The above approach easily  extends to more general regularizers/data-fit terms considering
\begin{equation*}
\label{e:Tik}
\tag{$P_\lambda$} \quad x_\lambda = \minimize{x\in X}  R(x) + \frac{1}{\lambda} D(Ax,y). 
\end{equation*}
From the above definition it is clear that Tikhonov regularization requires to solve an optimization problem, for each value of the regularization parameter $\lambda$. 
{For large scale applications, or if the problem is non-linear, solving $(P_\lambda)$ exactly is not possible, so only an approximation of $x_\lambda$ can be considered.
While a variety of techniques can be used to this purpose, iterative methods, and in particular those based on first order methods, are particularly favored. }
Broadly speaking, for each regularization parameter $\lambda \in \Lambda\subset ]0,+\infty[$, an iterative optimization method is defined by a sequence 
\begin{equation}\label{E:abstract one-run algorithm}
x_{0,\lambda} \in X, \ x_{n+1,\lambda}=\text{Algorithm}(x_{n,\lambda};   \lambda;y),
\end{equation} 
in such a way that $x_{n, \lambda}$ tends to $x_\lambda$  as $n$ grows. It is then clear that,  as mentioned in the introduction, the need to select a regularization parameter  can have a dramatic effect from a numerical perspective. 
Indeed, in practice $\Lambda$ is a finite set $\Lambda_N\subset ]0,+\infty[$, and an optimization problem needs
to  be solved for each regularization parameter $\lambda \in \Lambda_N$. If $N$ is the cardinality of the set $\Lambda_N$, the numerical complexity of the iteration~\eqref{E:abstract one-run algorithm} is now multiplied by $N$. 

{The question of deriving alternative regularization techniques tackling directly non-linearity and large scale issues, and having better complexity, is then of both theoretical and practical relevance. 
As mentioned next, iterative regularization provides one such alternative. }

\paragraph{Iterative regularization.}
Iterative regularization is typically derived considering an iterative optimization procedure to solve directly problem~\eqref{e:P} (rather than~\eqref{e:Tik}), 
\begin{equation}\label{E:abstract one-run it reg}
x_{0} \in X, \ x_{n+1}=\text{Algorithm}(x_n; y). 
\end{equation} 
For instance, in the setting of Example~\ref{ex:MP}, a classical  iterative regularization method is the Landweber method \cite{EngHanNeu96} defined by the iteration
$$
x_{0} \in X, \ x_{n+1}=x_n - \tau A^* (Ax-y),$$
where $\tau \in\left]0, 2 \| A \|^{-2}\right[$ is a stepsize.
Note that for iterative regularization methods, the regularization parameter is the \textit{number of iterations}.
In this setting, the regularization property~\eqref{e:reg_prop} reduces  to the convergence 
of the iteration to  $x^\dagger$ when \eqref{E:abstract one-run it reg} is applied with $y=\bar y$. 
Stability, when iteration~\eqref{E:abstract one-run it reg} is applied to noisy data, is ensured by defining a regularization parameter choice,  which in this case is a stopping criterion. 

When compared to Tikhonov regularization, the advantage of iterative regularization is mostly numerical. Computing solutions corresponding to different regularization parameters
 is straightforward, since the latter is simply the number of iterations. In practice, this property often turns into dramatic computational speed-ups while performing regularization parameter tuning.

A main motivation for this work is the observation that, differently from Tikhonov regularization, how to design   iterative regularization for general regularizers and data-fit terms is not as clear. 
Iterative regularization method that allow to consider  more general regularizers are known in the literature, but are typically restricted to quadratic data-fit functions.
 In practice this latter choice can be limiting, since  considering different  data-fit functions is often crucial. 
However, we are not aware of studies considering iterative regularization for general classes of error functions. The results we describe next are  a step
towards   filling this gap.

\section{The Diagonal Dual Descent (3-D) method}\label{S:the 3D method}

In this section, we describe  the iterative algorithm we propose and analyze in the rest of the paper. We begin by an informal description 
introducing some basic ideas, before providing a more detailed discussion.
\subsection{Diagonal algorithms}
\label{sec:DTR}
{ We will consider  an iterative optimization method based on a \textit{diagonal principle}. The classic idea \cite{BahLem94}  is to combine an 
optimization algorithm, with a sequence of approximations of  the given problem~\eqref{e:P}, changing eventually the 
approximation at each step of the algorithm.  }
In our setting, this  corresponds to  an  algorithm as in~\eqref{E:abstract one-run algorithm}
where  the parameter $\lambda$ can be updated at each iteration,
\begin{equation}
\label{E:DiagonalAlgoGeneralForm}
x_0 \in X, \ x_{n+1}:=\text{Algorithm}(x_{n};\lambda_n;y), \ \lambda_n \to 0.
\end{equation}
Roughly speaking, we allow the algorithm to ``{\em switch}'' between penalized problems corresponding to  different values of $\lambda$.
 As briefly recalled previously, for iterative regularization methods, the number of iterations, and thus here the sequence $(\lambda_n)_{n\in\mathbb{N}}$,
controls the accuracy with which $x_{n}$ approaches $x^\dagger$. 

We will first show that the  basic regularization property holds, namely that in the noiseless case
$x_{n}\to x^\dagger$, as $n\to+\infty$, provided that $\lambda_n\to 0$.  Then, we will prove stability with respect to noise. Combining this latter property with the regularization one
will allow us to derive a suitable stopping rule and  to build a stable approximation of $x^\dagger$. In particular, in the presence of noise, the  stopping rule  will impose termination
of  the iterative procedure before $\lambda_n$ reaches $0$, preventing numerical instabilities.
We now illustrate the diagonal principle in the setting of Example~\ref{ex:MP}.

\begin{example}{(Diagonal Landweber algorithm)}
\label{Ex:Diagonal Landweber algo}
In the setting of Example~\ref{ex:MP},  a basic diagonal algorithm is the diagonal Landweber algorithm
\begin{equation}
\label{E:Diagonal Landweber algo}
\begin{array}{|l}
{x}_0 \in X,  \ \lambda_n \to 0, \ \tau >0 \text{ is a stepsize}, \\
{x}_{n+1} ={x}_n - \tau A^*(A{x}_n -{y}) - \tau \lambda_n {x}_n.
\end{array}
\end{equation}
The above iteration can be seen as  the gradient descent method applied to \eqref{e:Tik}, for 
$R=\frac{1}{2}\Vert \cdot \Vert^2$ and $D(\cdot;y)=\frac{1}{2}\Vert\cdot - {y} \Vert^2$, and especially  considering $\lambda$  to change at each iteration.
The above  iteration  has been mainly studied for nonlinear inverse problems, and is known under several names:
modified Landweber iteration \cite{Sch98}, iteratively regularized Landweber iteration \cite{KalNeuSch08}, iteratively regularized gradient method \cite{BakKok04},  
or Tikhonov-Gradient method \cite{Ram03}.  
In Figure \ref{F:Landweber vs diagonal landweber}, we illustrate the difference between the diagonal Landweber algorithm and the classic Tikhonov method.
\end{example}

\begin{figure}[t]
\begin{center}
\includegraphics[width=3cm]{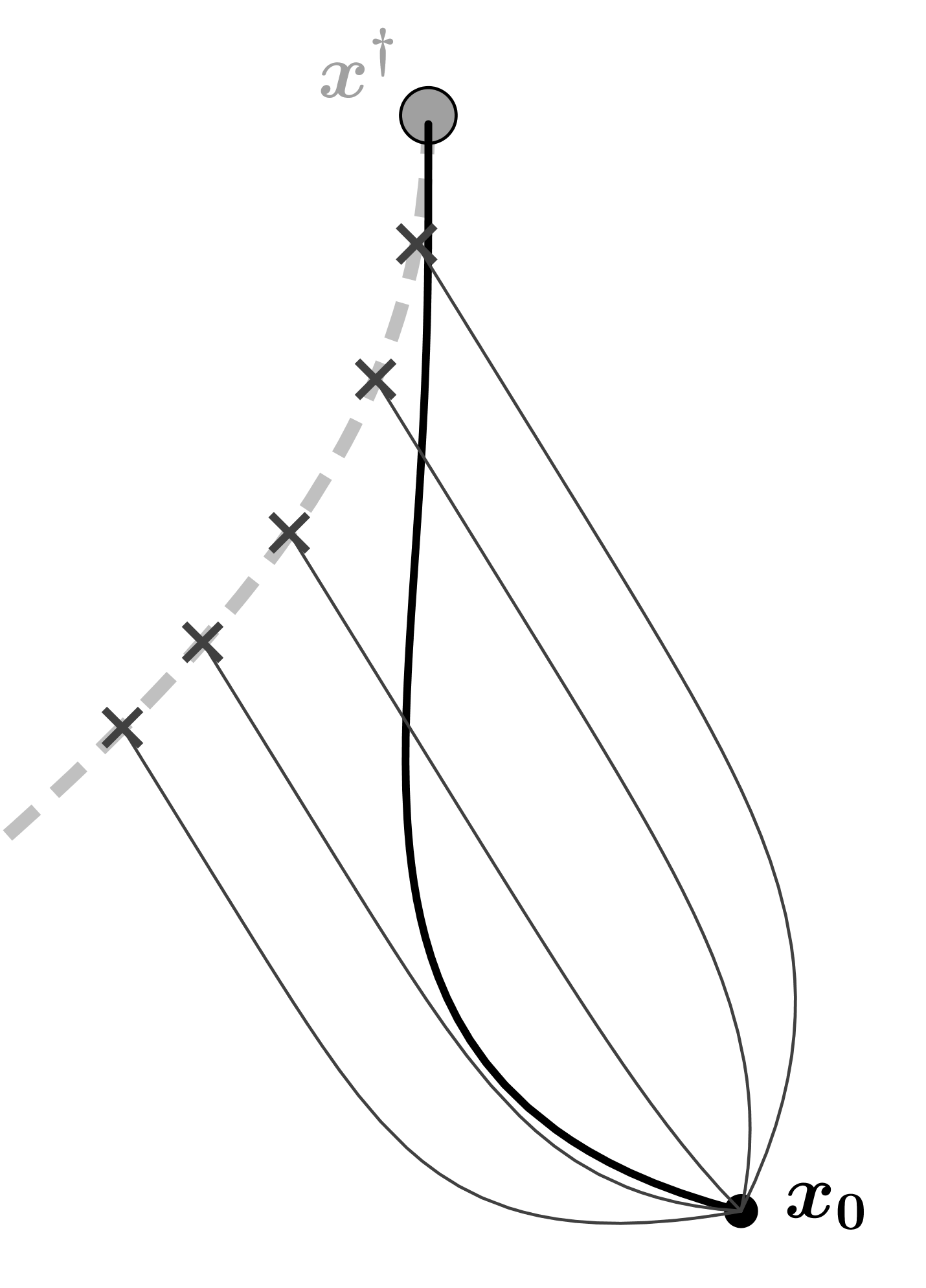}
\end{center}
\caption{
Thick dotted line: Tikhonov regularization path $\{ x_{\lambda} \}_{\lambda >0 }$. 
Thin  plain lines: Gradient Descent solving $(P_\lambda)$ for $\lambda\in \{1,0.75,0.5,0.25,0.1\}$, starting from $x_0$.
Thick plain line: Diagonal Landweber algorithm, with $\lambda_n=(n+1)^{-1}$.
Here $A=[(1,1)^T,(1,0)^T ]$ and $y=(2,1)^T$.
}
\label{F:Landweber vs diagonal landweber}
\end{figure}

It is interesting to relate diagonal methods to ``warm-restart'', a heuristic commonly used to speed up the computations of Tikhonov regularized 
solutions for different regularization parameter values \cite{BecBobCan11}. 

\begin{example}{(Warm restart)}
\label{Ex:warm restart}
Warm restart, or continuation method, is a popular heuristic used to approximately follow the {\em path} $\{ x_\lambda\,\colon\, \lambda\in\Lambda\}$ of 
solutions of problem \eqref{e:Tik}. The method is based on considering a sequence of problems $(P_{\lambda_i})_{i\in\N}$ 
for a decreasing family of parameters $(\lambda_i)_{i \in \N}$ in $\R_{++}$. 
Then, the  solutions corresponding to larger values of $\lambda_i$ are computed first and  used to initialize --  ``warm'' start -- the next problem.
The rationale behind the method, is the empirical observation that  
solving \eqref{e:Tik} with a first-order method as in \eqref{E:abstract one-run algorithm} is faster if $\lambda$ is large  \cite{HalYinZha08}. 
It is easy to see that this continuation  strategy generates a sequence $(x_n)_{n\in\N}$ which corresponds to  the diagonal algorithm \eqref{E:DiagonalAlgoGeneralForm}, for  
a piecewise constant decreasing sequence $(\lambda_n)_{n\in\N}$. 
The warm restart principle is illustrated in Figure \ref{F:Landweber vs warm restart}.
\end{example}

\begin{figure}[t]
\begin{center}
\includegraphics[width=3cm]{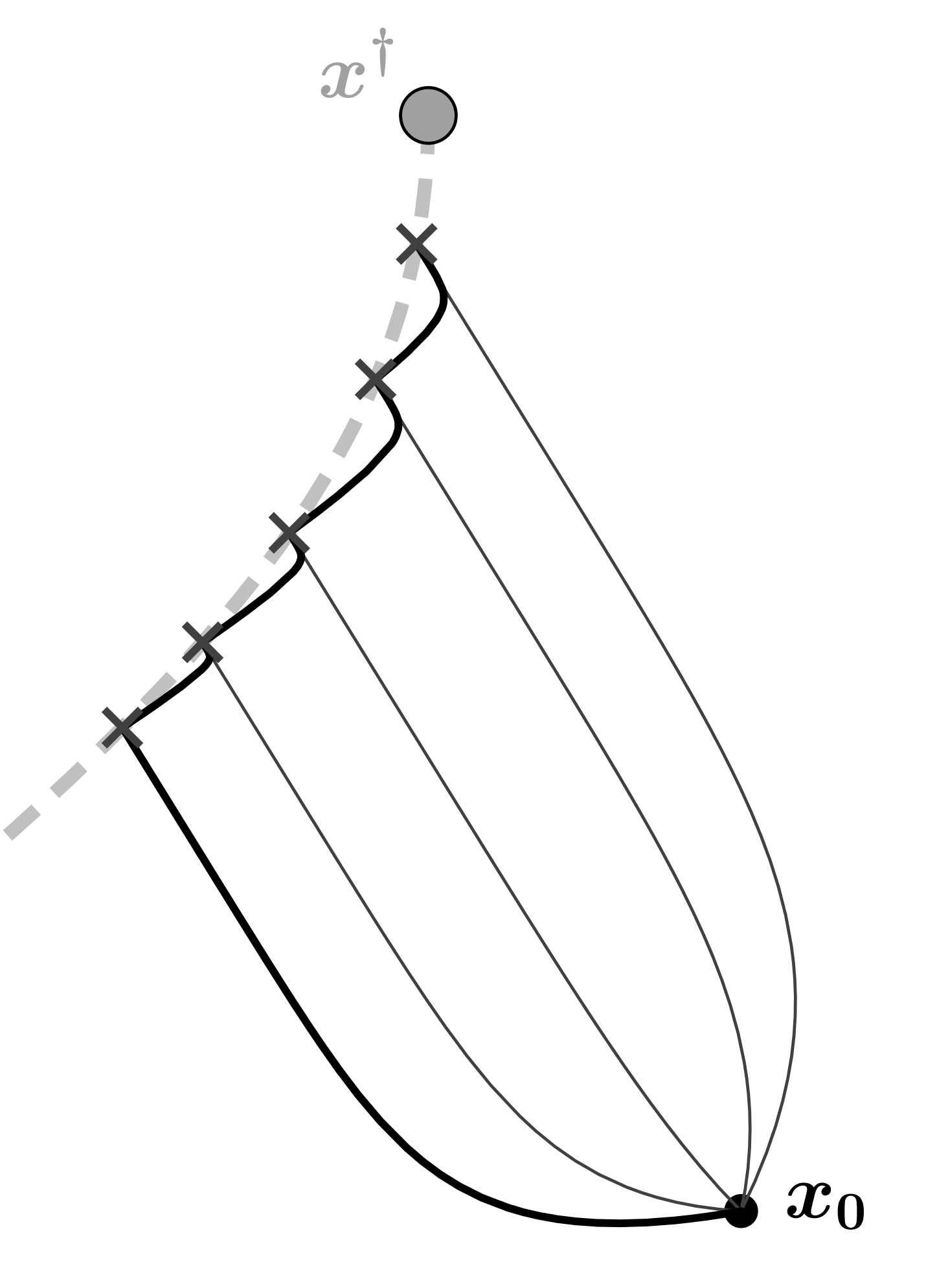}
\end{center}
\caption{
Exact same setting than for Figure \ref{F:Landweber vs diagonal landweber}, but here $\lambda_n$ is constant by parts, taking successively its values in  $\{1,0.75,0.5,0.25,0.1\}$.
}
\label{F:Landweber vs warm restart}
\end{figure}

In the optimization setting,  the literature on diagonal methods is vast. Diagonal procedures as in~\eqref{E:DiagonalAlgoGeneralForm} have been the object of various studies since the 70's \cite{Boy74,Kap75,Mar78,Kap79},  considering various algorithms coupled with a large class of penalization methods, such as Tikhonov penalization,  exponential barrier methods, interior methods, or more general principles.
More precisely, diagonal versions of the proximal algorithm have been considered in \cite{Kap75,Kap79,AusCroFed87,AlaLem91,Tos94,Com97,AlvCom02,Cab05},
the diagonal gradient method has been studied in \cite{Pey11}, and the diagonal projected gradient method in \cite{Boy74,Mar78}. 
More recently, a diagonal version of the forward-backward algorithm has been investigated in \cite{Lem88,Lem95,AttCzaPey11,CzaNouPey14}, see also \cite[Section 17.3.2]{YamYukYam11}. The above papers are concerned with convergence of the considered optimization criterion, corresponding to the regularization property for noiseless data in inverse problems. 
Stability and early stopping results are known only for the diagonal Landweber method 
 \cite{Sch98,Ram03,BakKok04,KalNeuSch08}.

A main novelty of our work is considering a dual diagonal approach, since the diagonal methods studied in the literature are  essentially primal\footnote{In \cite{AlvCom02}, 
the authors show that the proposed proximal method can be used to solve the dual problem, but the regularization 
method they consider is the exponential barrier, which is not of interest here.}. 
These latter approaches  are
well suited if the data-fit function $x \mapsto D(Ax;y)$ is ``simple", in the sense that either the proximity operator 
of $x \mapsto D(Ax;y)$ is easy to compute, or $D$ is smooth.  However, these properties might not be satisfied 
by the data-fit functions of interest, see Example~\ref{R:data-fit function list}. In particular,  when the data-fit function $D(\cdot ; y)$  is nonsmooth and $A$ is not orthogonal,  primal algorithms cannot be used.
{As we discuss next, a dual approach is necessary in this case, and requires the regularizer $R$ to be strongly convex. 
Note that, up to now, all studies on iterative regularization algorithms dealt with strongly convex regularizers
 and the least squares as the loss function, so the novelty in our approach is that it allows to extend the 
 iterative regularization principle to a large family of regularizers/loss functions. } 
 
 We point out that our analysis  builds on ideas and results recently developed to solve penalized problems \eqref{e:Tik}, see  \cite{BriCom11,ChaPoc11,ComDunVu10,ComPes12} and references therein.
 These latter works use  
 duality techniques to introduce  classes of algorithms that   decouple the contribution of $D$, $R$, and $A$. 
We have also been inspired by recent results concerning general diagonal dynamical systems \cite{AttCabCza16}.

\subsection{Main assumptions on the problem}
\label{SS:main assumptions}
Before describing the regularization method that we propose, we introduce the main assumptions on the constituents of the problem. 
Throughout the paper we make the simplifying assumption that there exists $\bar x\in X$ satisfying~\eqref{e:ideal}. 
The next assumption concerns the data-fit function $D$, and in particular its geometry, which is 
characterized by the notion of conditioning function, introduced in Definition \ref{D:conditioning}.

\smallskip

\noindent\textbf{Assumption (AD) on the data-fit function:}

\begin{myitemize} 
	\item[(AD1)] $D: Y\times Y \longrightarrow \left[0,+\infty\right]$, and
	\[	
	(\forall (u,y)\in Y^2)\quad D(u;y)=0 \iff u=y.
	\]	
	\item[(AD2)] for every $y\in Y$, $D_y:=D(\cdot;y)$ decomposes as 
	$$D_y= \psi_y \ \# \ \phi_y, $$
where $\phi_y \in \Gamma_0(Y)$, and $\psi_y=J_y+\frac{\sigma_{\psi}}{2}\|\cdot\|^2$ for some  $J_y\in\Gamma_0(Y)$ and $\sigma_{\psi}\in\R_{++}$.
	\item[(AD3)] $D(\cdot; \bar y)$ is coercive and
	$p$-well conditioned for some $p\in\left[1,+\infty\right[$, with conditioning modulus $\bar{m}$.
\end{myitemize}
Assumption (AD2) is equivalent to $D_y \in \Gamma_0(Y)$, but with the structural decomposition $D_y=\psi_y \ \# \ \phi_y$ 
we are able to detect its (possible) strongly convex component.  We will see later that  $\psi_y$ and $\phi_y$ 
 play a different role in our algorithm. 
This decomposition is not a restriction, since one can always take $\psi_y=\delta_{\{0\}}$, which is, for every $\sigma\in\R_{++}$,
$\sigma$-strongly convex, allowing the general form $D_y = \phi_y \in \Gamma_0(Y)$. 
For instance, all the data-fit functions listed 
in Example~\ref{R:data-fit function list} admit a trivial decomposition in which either $\psi_y$ or $\phi_y$
coincide with $\delta_{\{0\}}$, except for the Huber data-fit function, which can be equivalently written as 
$$
(\forall u \in \mathbb{R}^d)\quad H_\sigma(u-y)=\left( \Vert \cdot \Vert_1 \# \frac{\sigma}{2}\Vert \cdot - y \Vert^2 \right)(u).
$$
 Note that we assume the strong convexity constant $\sigma_\psi$ to be independent of $y$, which is always the 
case in the examples considered in Example~\ref{R:data-fit function list}.

Concerning (AD3), we remark that the coercivity of $D_{\bar y}$ is always satisfied in the finite dimensional setting, 
since, by (AD1), the zero level set of $D_y$ is nonempty and 
bounded~\cite[Proposition 11.12]{BauCom}.
The $p$-well conditioning assumption is satisfied for all the data fidelities considered in~Example~\ref{R:data-fit function list}, 
and their conditioning modulus can be easily computed, see Lemmas~\ref{l:cond_moduli}
and \ref{L:well conditionning of Kullback Leibler} in Appendix~\ref{S:Appendix2}.  All the mentioned losses are $2$-well conditioned, 
except for  the $L^1$ norm, which is $1$-well conditioned.

\smallskip

\noindent \textbf{Assumption (AR) on the regularizer:}

\begin{myitemize}
	\item[(AR1)] $R$ is $\sigma_R$-strongly convex, with $\sigma_R\in\R_{++}$,  
	\item[(AR2)] $\bar x \in \dom R$.
\end{myitemize}
Assumption $(AR1)$ plays a key role in our approach, which is based on the solution of the dual problem: indeed, strong convexity
is necessary for recovering primal solutions from the dual ones. 
Note that sparsity inducing regularizers are not strongly convex in general, since
they are usually the composition of the $L^1$ norm (or some mixed norm) with a linear operator. 
But we can enforce assumption (AR1), and thus apply our algorithm by adding a strongly convex 
quadratic term to the original regularizer, thus using a form of elastic-net penalty 
\cite{ZouHas05}. 
Assumption (AR2), combined with (AD1), implies that the 
ideal problem \eqref{e:P} with $y=\bar y$ has a solution and is equivalent to
\begin{equation*}
\label{E:primal ideal problem reduced}
\text{minimize} \ \left\{ R(x) \ | \ Ax=y \right\}.
\end{equation*}

\subsection{A primal-dual diagonal method}
\label{SS:primal dual for 3D}

As announced in  Section~\ref{sec:DTR}, our regularization method 
is a diagonal descent algorithm on the dual.
Given $\lambda\in\R_{++}$, we start by introducing the Fenchel-Rockafellar dual of  problem \eqref{e:Tik}
\cite[Definition~15.19]{BauCom}:
\begin{equation}
\label{E:dual penalized problem}
\tag{$D_\lambda$} \quad \minimize{u \in Y} \ R^*(-A^*u) + \frac{1}{\lambda} D_y^*(\lambda u).
\end{equation}
It is known that \eqref{e:Tik} converges, in an appropriate sense, to \eqref{e:P} as $\lambda$ goes to zero \cite{Att96}.
We will show in Proposition \ref{P:dissipativity} that, when $y=\bar y$, \eqref{E:dual penalized problem} converges to the dual problem
\begin{equation}
\label{e:D}
\tag{$D$} \quad \minimize{u \in Y} R^*(-A^*u) + \langle  \bar y , u \rangle.
\end{equation}
The decomposition we made explicit in (AD2) allows to express $D_{{y}}^*$ as the sum 
of a smooth and a $\Gamma_0(Y)$ component.
More precisely, by~\eqref{e:fenchelinfc}, we have 
$$
D_y^*=(\psi_y \ \# \phi_y)^*=\psi_y^*+\phi_y^*,
$$
so that \eqref{E:dual penalized problem} can be rewritten as:
\[
\begin{array}{ccc}
\minimize{u \in Y} \ & \underbrace{R^*(-A^*u) + \frac{1}{\lambda} \psi_y^*(\lambda u)} + 
& \underbrace{ \frac{1}{\lambda} \phi_y^*(\lambda u)} \\
 & \text{\footnotesize smooth} & \text{\footnotesize nonsmooth}
\end{array}
\]
We underline the fact that $R^*$ and $\psi_y^*$ are Fr\'echet differentiable, 
with their gradient being respectively $\sigma_R^{-1}$ and $\sigma_{\psi}^{-1}$-Lipschitz continuous \cite[Theorem 18.15(v)-(vii)]{BauCom}.
Then, it is natural to solve \eqref{E:dual penalized problem} using a forward-backward method \cite{ComWaj05}, which alternates between gradient 
steps with respect to the smooth part, and proximal steps with respect to the nonsmooth part.
This forward-backward splitting algorithm, coupled with the diagonal principle discussed in Section~\ref{sec:DTR}, takes the following form:
\[
\begin{array}{|l}
{u}_0 \in Y,  \ (\lambda_n)_{n\in\N}, \ \tau\in\R_{++}, \\
w_{n+1}= u_n + \tau A \nabla R^*(-A^* u_n) - \tau \nabla \psi_y^*(\lambda_n u_n), \\
u_{n+1} = \prox_{{\tau\lambda_n^{-1}}\phi^*_y(\lambda_n \cdot)}(w_{n+1}).
\end{array}
\]

\noindent By introducing an auxiliary primal variable, and making use of the Moreau decomposition theorem \cite[Theorem 14.3]{BauCom},
we obtain the final form of our algorithm:

\begin{center}
\begin{tabular}{|l|}
\hline\\[-0.4ex]
{\bf Diagonal Dual Descent (3-D) method}\\[1ex]
Let $(\lambda_n)_{\nin}$ be a sequence in $]0,+\infty[$ decreasing to $0$,\\ 
let ${L}={{\| A \|^2}/{\sigma_R} + {\lambda_0}/{\sigma_{\psi}}}$, and  $\tau\in\left]0, {1}/{{L}}\right]$. \\
Let $ u_0 \in Y$, and for all $\nin$, let\\[1.5ex]

$\begin{array}{l|l}
&  x_{n} = \nabla R^*(-A^*  u_{n}) \\
(\text{3-D})~&  w_{n+1}= u_n + \tau A  x_n - \tau \nabla \psi_y^*(\lambda_n  u_n) \\
&  u_{n+1} =  w_{n+1} - \tau \prox_{(\tau \lambda_n)^{-1} \phi_y} \left( \tau^{-1} w_{n+1} \right)
\end{array}$
\\
\ \\
\hline
\end{tabular}
\end{center}

The above method, dubbed (3-D), is a first order method, in which the main components of the problem ($R,A,\psi_y$, and $\phi_y$) are activated separately. (3-D) requires the computation of the proximity operator of ${\phi_y}$.
By definition, this is an implicit step, and the solution of the minimization problem in 
\eqref{D:proximal operator} is needed. 
However,  in many cases of interest, this proximity operator can be easily computed in closed form \cite{ComPes11}. 
Also, the computation of the gradients $\nabla R^*$ and $\nabla {\psi_y}^*$ is needed in (3-D), which also 
corresponds to  the computation of a proximity operator, as shown in the following.

Let us describe in detail what are the main steps of (3-D) when considering a pair of 
regularizer/data-fit function among the ones discussed in Examples~\ref{E:regularizers list} and~\ref{R:data-fit function list}.
We start with the first step, which involves the strongly convex regularizer $R$:
\begin{itemize}
	\item Let $R=\frac{1}{2}\Vert \cdot \Vert^2$, then $\nabla R^*(x)=x$ for every $x\in X$.
	\item Let $X=\R^d$, and, for every $x\in X$, let $R(x)=\Vert Wx \Vert_1 + \frac{\sigma}{2} \Vert x \Vert^2$, with $W \in \R^{d\times d}$. 
	If $W$ is an orthogonal matrix,  Lemma~\ref{L:gradient dual prox formula} yields
$$
\nabla R^*(x)= W^* \prox_{\sigma^{-1}\Vert \cdot \Vert_1}(\sigma^{-1}Wx),
$$
 where $\prox_{\sigma^{-1}\Vert \cdot \Vert_1}$ is the well-known soft-thresholding operator \cite{DauDefDem04,ComWaj05}. 
 If $W$ is not orthogonal, as it is the case for the total variation, we can only write
 $$
\nabla R^*(x)=  \prox_{\sigma^{-1}\Vert W \cdot \Vert_1}(\sigma^{-1}x),
$$
and the proximity operator have to be computed by a separate procedure.
\end{itemize}
Then, we consider the second step of (3-D) involving $\psi_y$, the strongly convex part of the data-fit function: 
\begin{itemize}
	\item If $\psi_y=\delta_{\{0\}}$, then $\nabla \psi_y^*=0$.
	\item If $\psi_y=\frac{1}{2}\Vert \cdot - y \Vert^2$, then, for every $u\in Y$, $\nabla \psi_y^*(u)=u+y$.
	\item If $X=\R^d$ and, for every $u\in \R^d$, $\psi_y(u)=\alpha_1\Vert u-y \Vert_1+\frac{\alpha_2}{2}\Vert u - y \Vert^2$, then 
	$$\nabla\psi_y^*(u)={y}+\prox_{\alpha_2^{-1}\alpha_1\|\cdot\|_1}(\alpha_2^{-1}u).$$
\end{itemize}
Finally, the third step of (3-D) involves $\phi_y$, the $\Gamma_0(Y)$ part of the data-fit function: 
\begin{itemize}
	\item If $\phi_y=\delta_{\{0\}}$, then $\prox_{\alpha \phi_y}=0$.
	\item If $X=\R^d$ and, for every $u\in\R^d$, $\phi_y(u)=\Vert u - y \Vert_1$, then $\prox_{\alpha \phi_y}(u)= y + \prox_{\alpha \Vert \cdot \Vert_1}(u-y)$.
	\item If $X=\R^d$ and $\phi_y=\KL(y,\cdot)$, the proximity operator of $\phi_y$ can be computed in closed form. Its expression
can be found  in \cite{ChaComPesWaj07} (see also \cite{DupFadSta12}).
\end{itemize}

\subsection{Relationship between (3-D) and other methods}
\label{relationships for 3D}

Before presenting our  main results, we relate (3-D)
 to  algorithms known in the literature.

\begin{remark}[Diagonal Lagrangian methods]
Assume that there exists $G\in \Gamma_0(Y)$ such that for all $(u,y) \in Y^2$, $D(u;y)=G(u-y)$.
Then,  problem~\eqref{e:Tik} can be rewritten
as
\begin{equation*}
 \minimize{Ax-z=y} \ R(x) + \frac{1}{\lambda} G(z).
\end{equation*}
Thanks to its structure, this problem is well suited for  Lagrangian methods.
Introduce then the Lagrangian $L\colon X\times Y^2\times\R_{++}\to \R\cup\{+\infty\}$ and,
for $\tau\in\R_{++}$, 
the augmented Lagrangian $L_\tau\colon X\times Y^2\times\R_{++}\to \R\cup\{+\infty\}$ of this problem, 
being respectively
$$
\begin{array}{lcl}
L(x,z,u;\lambda) & = & R(x) + \frac{1}{\lambda} G(z) + \langle u, Ax - z -y \rangle, \\
L_\tau(x,z,u;\lambda) & = & L(x,z,u;\lambda) + \frac{\tau}{2} \Vert Ax - z-y \Vert^2.
\end{array}
$$
Tseng's Alternating Minimization Algorithm \cite{Tse91}
is applicable and writes as,
 
$$
\begin{array}{|l}
(z_{-1},u_0) \in Y^2, \\
x_n= \underset{x \in X}{\argmin} \ L(x,z_{n-1},u_n;\lambda), \\
z_{n} = \underset{z \in Y}{\argmin} \ L_\tau(x_n,z,u_n;\lambda), \\
u_{n+1}=u_n + \tau (Ax_n - z_n -y) .
\end{array}
$$
The diagonal version of this Alternating Minimization Algorithm, 
where $\lambda$ is replaced by $\lambda_n$, is exactly (3-D) 
applied to $\psi_y=\delta_{\{0\}}$ and $\phi_y=G( \cdot - y)$.

If $G$ is strongly convex, it is not necessary to use the augmented Lagrangian to update  $z_n$.
Instead, we can use a simple Lagrangian method \cite{Uza58}:
$$
\begin{array}{|l}
(z_{-1},u_0) \in Y^2, \\
x_n=\underset{x \in X}{\argmin} \ L(x,z_{n-1},u_n;\lambda), \\
z_{n} = \underset{z \in Y}{\argmin} \ L(x_n,z,u_n;\lambda), \\
u_{n+1}=u_n + \tau (Ax_n - z_{n} -y).
\end{array}
$$
It can be verified that the diagonal version of 
this algorithm coincides with (3-D),  applied to $\psi_y=G( \cdot - y)$ and 
$\phi_y=\delta_{\{0\}}$.
Observe that, thanks to the decomposition we made explicit in (AD2), (3-D) 
unifies the two cases, and generalizes the analysis to a general
data-fit function, such as the Kullback-Leibler divergence, 
not necessarily of the form $G(\cdot-y)$. 
\end{remark}

\begin{remark}[Diagonal Mirror descent]
\label{Ex:diagonal mirror descent} 

\noindent Let $X=\R^d$, and suppose that $D(\cdot;y)=\psi_y=\frac{1}{2}\Vert \cdot - {y} \Vert^2$. 
Let $(x_n,w_n,u_n)_{n\in\N}$ be the sequence generated by (3-D), and define, for
every $\nin$, $x_n^*= -A^*u_n$. Then
$$
\begin{array}{|l}
{x}_0^* \in \im A^* \\
x_n=\nabla R^*(x_n^*)\\
{x}_{n+1}^* ={x}_n^* - \tau A^*(A{x}_n -{y}) - \tau \lambda_n {x}_n^*.
\end{array}
$$
Since $x_n^* \in \partial R(x_n)$, the latter can be seen as a diagonal version of the mirror descent method of  \cite{BecTeb03,BotHei12} applied to \eqref{e:Tik}, with $R$ as a mirror function.
In particular, when $R= \frac{1}{2}\Vert \cdot \Vert^2$, the (3-D) algorithm coincides with
the Diagonal Landweber algorithm of Example \ref{Ex:Diagonal Landweber algo}, 
with an initialization $x_0=x_0^* \in \im A^*$ (see more discussion on this in 
Remark \ref{R:diagonal Landweber non L1}).
\end{remark}

\section{Regularization properties of (3-D)}\label{S:main results}

In this section we present the two main results of this paper. The convergence  of (3-D) for exact data
is studied in Section~\ref{sec:reg} and its stability properties are considered
in Section~\ref{sec:stab}. The corresponding proofs are postponed to Sections~\ref{S:Convergence} and~\ref{S:regularization}, respectively.

\subsection{Regularization}\label{sec:reg}

We consider  the regularization properties of (3-D) in the noiseless case.
From an optimization perspective, this consists in studying the convergence of the algorithm.
To prove convergence of $(x_n)_{n\in\N}$, we need to impose a suitable decay condition on $(\lambda_n)_{n\in\N}$.
More precisely, we impose a summability condition on $(\lambda_n)_\nin$,
which is directly related to the $p$-well conditioning of $D_{\bar{y}}$ assumed in (AD3):
\begin{equation*}
(\lambda_n)_\nin \in \ell^{\frac{1}{p-1}}(\N).
\end{equation*}
Note that, when $p=1$, the notation $1/0$ will stand for $\infty$. In this case, the condition is automatically satisfied, since in the definition of (3-D) it is required that $\lambda_n \downarrow 0$.

\begin{theorem}[Convergence]
\label{T:mainconv}
Let $(x_n,w_n,u_n)_\nin$ be generated by {\em (3-D)} with $y=\bar y$.
Suppose that assumptions {\em (AR)} and {\em (AD)} hold, and suppose that $(\lambda_n)_\nin \in \ell^{\frac{1}{p-1}}(\N)$.
Let $x^\dagger$ be the solution of the problem~\eqref{e:P}.
Then the following three properties are equivalent:

\noindent (i) $\partial R(x^\dagger) \cap \im A^* \neq \emptyset$,

\noindent (ii) the dual problem \eqref{e:D} admits a solution,

\noindent (iii) $(u_n)_{n\in\N}$ is a bounded sequence.

\noindent If one of these properties is satisfied, then the following hold:
\begin{enumerate}
	\item $(u_n)_{n\in\N}$ weakly converges to a solution of \eqref{e:D}.
	\item $(x_n)_{n\in\N}$ strongly converges to $x^\dagger$, with 
\begin{equation}\label{e:smallo}
 \Vert x_n - x^\dagger \Vert=o\left(n^{-1/2} \right).
 \end{equation}
	\item  Let $u^\dagger$ be any solution of problem \eqref{e:D} and let $N\in \N$ 
	be such that $\Vert u^\dagger \Vert \lambda_N \in \inte \dom \bar m^*$. Then,
\begin{equation}\label{E:CV of 3-D precise rates}
\forall n\geq N, \ \Vert x_n - x^\dagger\Vert \leq \dfrac{C}{\sqrt{ n - N}}, \
\end{equation}
\noindent with 
$C^2=\dfrac{1}{\tau \sigma_R }\Vert u_{N} - u^\dagger \Vert^2 
+
\sum\limits_{n=N}^{+\infty} \frac{2}{\sigma_R\lambda_n}\bar{m}^*(\Vert u^\dagger \Vert \lambda_n) .$
\end{enumerate}
\end{theorem}
We next collect several observations on our convergence result. 

\begin{remark}[On the qualification condition]

\noindent The assumption $\partial R(x^\dagger) \cap \im A^* \neq \emptyset$ is used as a qualification\footnote{Here the term \textit{qualification condition} 
shall be understood as in the optimization literature. It is a sufficient condition ensuring, in our case, strong duality between the problems \eqref{e:P} and \eqref{e:D}.
It should not  be confused with the notion of qualification used in the inverse problem literature, which is a property for a regularization method \cite[Remark 4.6]{EngHanNeu96}. } condition for the optimization
problems \eqref{e:P} and \eqref{e:D}.
When $R$ is the squared norm, it is an instance of a common assumption in the inverse problem literature, known as 
\textit{source condition} (also source-wise representation, or smoothness assumption) \cite{EngHanNeu96}. 
In this more general form, it has been considered in the context of iterative regularization methods in a series of papers, see \cite{BurOsh13,BotHof13} and references therein. 
Observe that this qualification condition is verified as soon as $R$ is continuous at $x^\dagger$, and $\im A$ is closed.\footnote{To see this, it is enough to write the optimality condition of \eqref{e:P} and use the Moreau-Rockafellar Theorem~\cite[Theorem 3.30]{Pey}.}
Thus,  this assumption is always satisfied in a finite dimensional setting. 
\end{remark}

\begin{remark}[On the convergence and  rates] 

\noindent As we already mentioned, primal diagonal splitting methods for solving problem~\eqref{e:P} are considered in  \cite{Cab05,Pey11,AttCzaPey11,AttCzaPey11_2,CzaNouPey14}.
It is proved in these papers that, under a strong convexity assumption, the iterates converge strongly, but no rates of convergence are provided.
To the best of our knowledge, the  available results on convergence rates for a diagonal algorithm are limited to the 
diagonal Landweber algorithm \cite{Sch98,Ram03,BakKok04,KalNeuSch08}.
So, Theorem~\ref{T:mainconv} is the first result establishing convergence rates for the iterates  obtained with such a general diagonal scheme.
Even for primal algorithms, no convergence rates are known for the sequence $(D(Ax_n;\bar y))_{n\in\N}$. 
For our dual scheme,  the sequence of iterates $(x_n)_{\nin}$ is not necessarily contained in the domain of the loss function, therefore
convergence rates cannot be expected without additional assumptions.

The convergence rates obtained for the  iterates can also be compared with those of non-diagonal schemes.
Indeed, when dealing with the exact data $\bar y$, the problem \eqref{e:P} consists in minimizing $R$ over the 
set of solutions of the linear equation $Ax=\bar y$, so one could consider other algorithms that solve this problem. 
One possibility is to use a forward-backward splitting on the dual of \eqref{e:P}, a.k.a. mirror descent or linearized Bregman iteration.
In this case the rate of convergence $O(n^{-1/2})$ for the primal sequence have been obtained in \cite{BurResHe07} (see also references therein).
In this setting, it is possible to accelerate the rate of convergence by using an inertial method on the dual, and get $O(n^{-1})$ for 
the primal sequence \cite{ChaPoc15,RosVil16}. 
The convergence rate $O(n^{-1/2})$ for  the sequence $(D(Ax_n;\bar y))_{n\in\N}$ has been proved using  the cutting plane method in \cite{BecSab14}.
\end{remark}

\begin{remark}[On the decay of the parameters $(\lambda_n)_{\nin}$]

\noindent To get convergence we impose a summability assumption on $(\lambda_n)_{n\in\N}$.
This kind of hypothesis also appears in \cite{Pey11,AttCzaPey11,AttCzaPey11_2,CzaNouPey14} and is a key for obtaining convergence in the primal setting.
Here we would like to highlight that our assumption is easier to deal with than the one discussed in the above mentioned papers.
Indeed, in their primal setting, the authors make an assumption on $(\lambda_n)_\nin$ related to the well-conditioning 
of the data-fit  function $x \mapsto D(Ax;\bar y)$ (see \cite[Example 4.1]{AttCabCza16}).
But it can be difficult to compute the conditioning modulus for a general data-fit function \textit{coupled} with a 
linear operator, and in general such a modulus doesn't exist. For instance this happens when $A$ is ill-conditioned (take for 
instance $\frac{1}{2}\Vert Ax - \bar y \Vert^2$ when $A$ does not have  a closed range).
Our dual approach plays a crucial role here, since it allows us to make an assumption which involves only the data-fit function 
 $u\in Y \mapsto D(u;\bar y)$, and does not depend on the linear operator.
In addition, we point out that we do not need to consider a \textit{slow} decay for $(\lambda_n)_{\nin}$.
Indeed, it is a common assumption for primal diagonal methods to assume that 
$(\lambda_n)_{\nin} \notin \ell^1(\mathbb{N})$ (slow parametrization hypothesis).
See more discussion on this in Remark \ref{R:non L1 discussion}.
\end{remark}

\subsection{Stability}\label{sec:stab}

One of the main advantages of (3-D) is its capability to handle general data-fit functions, and therefore 
to be adaptive to the nature of the noise, see Remark~\ref{rem:bay}.

According to Theorem~\ref{T:mainconv}, the iterates of (3-D) converge to the unique solution of problem \eqref{e:P}
when we have access to the exact datum $\bar y$. Since we are interested in the situation where only a 
noisy version is available, in this section we investigate how the error  on $\bar y$
affects the sequence generated by (3-D). 
More precisely, let $\hat{y}$ be a noisy estimate of $\bar y$ (in a sense that will be made 
precise later). We consider the application of (3-D)  to the perturbed datum $\hat{y}$, that is 
\begin{equation}
\label{e:ndama}
\begin{array}{l|l}
 &  \hat{x}_{n} = \nabla R^*(-A^*  \hat{u}_{n}), \\
&  \hat{w}_{n+1}= \hat{u}_n + \tau A  \hat{x}_n - \tau \nabla \psi_{\hat y}^*(\lambda_n  \hat{u}_n), \\
 &  \hat{u}_{n+1} =  \hat{w}_{n+1} - \tau \prox_{(\tau \lambda_n)^{-1} \phi_{\hat{y}}} \left( \tau^{-1} \hat{w}_{n+1} \right),
\end{array}
\end{equation}
initialized with $\hat{u}_0 \in X$. We then consider the auxiliary sequence obtained applying (3-D) to the ideal
datum $\bar y$, with the same stepsizes, same sequence of parameters $(\lambda_n)_{n\in\N}$, and same initial 
point $u_0=\hat{u}_0$.
If we write 
\begin{equation}\label{E:optimization/error dichotomy}
\Vert \hat{x}_n-x^\dagger \Vert  \leq \Vert \hat{x}_n-x_n \Vert + \Vert x_n-x^\dagger \Vert,
\end{equation}
we immediately see that the analysis of (3-D) as  a regularization method is based on the decomposition of the error in two terms.
The first one is the error due to the noisy data, whose growth depends on the number of iterations, and the amplitude of the error between $\hat y$ and $\bar{y}$.  
The second is an approximation/regularization error, which coincides with the optimization error in the 
noise free case, that we bounded in Theorem~\ref{T:mainconv}. 
This decomposition suggests that when the stability error is of the same order of the optimization error, the iteration should
be stopped. Thus, iterative regularization properties of (3-D), as for  iterative regularization methods, depend on a reliable \textit{early stopping
rule} (see Proposition~\ref{T:stability 3-D}).   
This behavior is known as semiconvergence \cite{EngHanNeu96,BerBoc98}. 
As stressed above, to state this stability result we need to quantify the  error introduced in the problem by the noisy observation $\hat{y}$.
In the case of additive noise, a natural measure is  $\delta=\|\bar y-\hat{y}\|$. 
We next introduce a similar notion, which is tailored for more general noise models. 
It involves the data-fit function, and in particular the proximity operators of $\phi_{\bar y}$ and $J_{\bar y}$ (see (AD2))
and their noisy counterparts $\phi_{\hat{y}}$ and $J_{\hat{y}}$ required in the (3-D)'s steps. 

\begin{definition}
\label{def:deltanoise}
Let Assumption (AD) hold. Let $({y},\hat{y})\in Y^2$,
let $\delta\in\R_{++}$, and let $\theta\in\R_{+}$.
We say that  $\hat{y}$ is a $(\delta,\theta)$-perturbation of $y$ according to $D$, and write $\hat{y} \in S_{\delta,\theta}( y)$,
if the two following conditions are satisfied:
\begin{eqnarray}
\sup\limits_{u \in Y} \ \Vert\prox_{(J_{y/\sigma_{\psi})}}(u) - \prox_{(J_{\hat{y}/\sigma_{\psi})}}(u) \Vert & \leq & \delta,\label{E:stability psi} \\
(\forall \alpha >0) \ \sup\limits_{u \in Y} \ \Vert \prox_{\alpha\phi_{y}}(u) - \prox_{\alpha\phi_{\hat{y}}}(u) \Vert & \leq & \alpha^\theta \delta. \label{E:stability phi}
\end{eqnarray}
\end{definition}
Definition~\ref{def:deltanoise} identifies perturbations of the data that ensure stability of the data-fit function. 
Here, stability is measured in terms of sensitivity of the proximity operators of the components of 
$D(\cdot,\cdot)$ with respect to perturbations of the second variable. 
Since the definition is somewhat implicit,  before proving the stability result, we consider 
some specific data-fit functions, and give examples of $\hat{y}$ for 
which the conditions  \eqref{E:stability psi}  and \eqref{E:stability phi} are verified. 

\begin{example}\ 
\label{Ex:stability discrepancy functions}
We show that for the commonly used data-fit functions, 
the set of perturbed data $S_{\delta,\theta}(y)$ can be either characterized, or estimated.  
See Lemmas~\ref{L:prox additive form noise} and~\ref{L:prox KL form noise} in the Appendix for the proof of items (iii-iv).

\smallskip

\noindent (i) if $\psi_y$ (resp. $\phi_y$) is independent of $y$, then \eqref{E:stability psi}   (resp.  \eqref{E:stability phi}) is trivially satisfied for every $\hat{y} \in Y$. 
	This is the case for the function $\delta_{\{0\}}$, and for the $L^1$ norm term appearing in the Huber loss.

\smallskip	
	
\noindent (ii) if $\psi_y=\Vert \ \cdot \ - y \Vert^2/2$, then $J_{y}=\|y\|^2/2-\langle x,y\rangle$, $\sigma_{\psi_{y}}=1$, 
	 and $\Vert\prox_{(J_y/\sigma_{\psi})}(u) - \prox_{(J_{\hat{y}}/\sigma_{\psi})}(u) \Vert=\|y-\hat{y}\|$.
	 The previous computations imply in particular that for the quadratic and Huber data-fit functions (see Example~\ref{R:data-fit function list}), 
	 we recover the usual definition of perturbation and the classical notion of additive noise,  that is,  for every $\theta \geq 0$
\[
S_{\delta,\theta}(y)=\{\hat{y}\in Y\,|\,\|{y}- \hat{y}\|\leq\delta \}.
\]

\smallskip

\noindent (iii) Suppose that $\phi_y = G( \ \cdot \ - y )$, for some $G \in \Gamma_0(Y)$ such that $\argmin G =\{0\}$.
This covers most of the data-fit functions having an additive form: any norm (e.g. the $L^p$ norms for  $p \in [1,+\infty]$), the Huber loss, 
or sums of these functions.
Lemma~\ref{L:prox additive form noise} shows that
$$\sup\limits_{\alpha >0} \ \sup\limits_{u \in Y} \ \Vert \prox_{\alpha\phi_y}(u) - \prox_{\alpha\phi_{\hat{y}}}(u) \Vert = \Vert y - \hat y \Vert.$$
So, if we moreover assume that $\psi_y=\delta_{\{0\}}$, we obtain 
	 \[
S_{\delta,0}(y) = \{\hat{y}\in Y\,|\,\|\hat{y}-y\|\leq \delta \}.
\]
	 
\smallskip

\noindent  (iv) Let $Y=\R^d$, and  $D_y=\phi_y = \KL(y, \ \cdot \ )$. Let $y,\hat{y}\in\mathbb{R}^d_{++}$. Then, for all  $\alpha\in\R_{++}$, 
 \[ 
 \sup\limits_{u \in \R^d} \ \Vert \prox_{\alpha\phi_y}(u) - \prox_{\alpha\phi_{\hat{y}}}(u) \Vert = \sqrt{\alpha} \Vert \sqrt{\hat{y}}-\sqrt{y} \Vert,
 \]
so that
	\[
	S_{\delta,\theta}(y) = 
	\begin{cases}
			\{\hat{y}\in Y\,|\, \hat{y}> 0,\,\|\sqrt{\hat{y}}-\sqrt{y}\|\leq\delta \} & \text{ if } \theta=1/2, \\
			\{y \} & \text{ if } \theta \neq 1/2,
	\end{cases}
	\]
	where  the notation $\sqrt{y}$ shall be understood componentwise.
\end{example}
\begin{theorem}[Existence of early-stopping]
\label{thm:reg} 
Under the same assumptions as in Theorem~\ref{T:mainconv}, assume that the qualification condition
\[
\partial R(x^\dagger) \cap \im A^* \neq \emptyset
\] 
holds.
 Let $\hat{y}\in Y$, and let $(\hat{x}_n,\hat{w}_n,\hat{u}_n)_{\nin}$ be the sequence generated by {\em (3-D)} 
algorithm with $\hat{u}_0=u_0$ and $y=\hat{y}$. Moreover, suppose that
\begin{itemize}
	\item $\hat{y}\in Y$ is a $(\delta,\theta)$-perturbation of $\bar{y}$ according to $D$, with $\delta\in [0,+\infty [$ and  $\theta\in ]0,+\infty[$; 
	\item for every $\nin$, $\lambda_n =\lambda_0/(n+1)^\beta$, for some  $\beta\in]{p-1},+\infty[$.
\end{itemize}
Then there exists $t(\delta)\sim \delta^{-{2}/(3+2\beta\theta)}$ such that for all $c \geq 1$, the early stopping rule $n(\delta)=\lceil c t(\delta) \rceil$ verifies
\[
\Vert \hat{x}_{n(\delta)} - x^\dagger \Vert =O\left(\delta^{\frac{1}{3+2\beta\theta}}\right) \text{ when } \delta \to 0.
\]
\end{theorem}

As said before, the key for the proof of Proposition~\ref{T:stability 3-D} is the estimation \eqref{E:optimization/error dichotomy}, where we combine the regularization rates of Theorem \ref{T:mainconv}, and a stability estimate whose proof can be found in Section \ref{S:regularization}:
\begin{equation*}
\Vert x_n - \hat x_n \Vert =O \left( \delta n^{1+\beta \theta} \right).
\end{equation*}

\noindent As can be directly seen, the stability bound depends on the  chosen data-fit function. The dependence is  through the 
exponents $\beta$ and $\theta$, whose choice is restricted by the geometry (see (AD3)) and the stability properties of the data-fit function.  
 In particular, though the best rates are obtained for $\theta=0$, this choice is not always feasible, as 
Example~\ref{Ex:stability discrepancy functions} shows. 
We discuss more in detail the dependence of the iterative regularization 
rates from the two parameters $\beta$ and $\theta$ in the next remark.

\begin{remark}[{On the effect of ($\beta$,$\theta$) on the resulting rate}]\label{R:choice of beta for stability rates}
Our stability result shows that the convergence rates are faster when $\beta\theta$ is close to zero.
For most data-fit functions presented in Example~\ref{Ex:stability discrepancy functions}, in particular the ones having an additive form, we  can
consider errors with $\theta=0$.
In this case, the stopping rule $n(\delta) \sim \delta^{-2/3}$ leads to a rate of convergence $O(\delta^{1/3})$.
It is worth noting that in this setting, both estimates are independent from the choice of the parameter sequence $(\lambda_n)_{n\in\N}$. 
This is not the case if $\theta\neq 0$, e.g. for the Kullback-Leibler divergence, where we need to take $\theta=1/2$. 
For this function, for which the assumption $\lambda_n\in \ell^{1/(p-1)}(\N)$ is satisfied with $p=2$, 
it is possible to reach a convergence rate arbitrarily close to $O(\delta^{1/4})$ by considering $\beta$ arbitrarily close to $1$.  
Our method is, at the best, of  order $O(\delta^{1/3})$. 
When specialized to the square loss, this rate is not optimal. The optimal one, 
which is achieved e.g. for the diagonal Landweber algorithm, is of the same order of the Tikhonov regularization, 
and it is $O(\delta^{1/2})$ \cite[Theorem 6.5]{Ram03}.
It is an open question to know whether our rates can be improved (by a smarter choice of the early stopping rule, as
in \cite{RosVil16}), or if it is not possible to achieve optimal rates in such a general setting. 
\end{remark}

\begin{remark}[{On  early stopping in practice}] 
Theorem~\ref{thm:reg} states
the existence of a stopping time for which a stable reconstruction is achieved, and thus establish that the (3-D) method is
an iterative regularization procedure. 
In particular, this  explains the dependence on the noise of
the warm restart method, often used in practice to speed up Tikhonov regularization. Theorem~\ref{thm:reg} is mainly of theoretical interest, since  
the stopping iteration depends on constants that are not available, and on the noise level $\delta$, which as well is often not  accessible.
However, every parameter selection method used in practice (e.g. discrepancy principle or cross-validation) to choose the regularization parameter in
Tikhonov regularization can be used in this context  as well.
This will be illustrated in the numerical section \ref{S:numerical}.
\end{remark}

\section{Theoretical analysis: convergence result}\label{S:Convergence}

In this section we prove the convergence of (3-D).
 We  assume that $(x_n,w_n,u_n)_{\nin}$ is a sequence generated by the (3-D) algorithm, using the exact data $y=\bar y$.
We introduce the following notation which will be used in the subsequent proofs. For every $\nin$
and every $u\in Y$, 
\begin{itemize}
	\item  $d_n (u):= R^*(-A^*u) + \frac{1}{\lambda_n} D_{\bar y}^*(\lambda_n u)$,
	\item  $d_\infty (u):= R^*(-A^*u) + \langle \bar y,u \rangle$.
\end{itemize}
Here $d_n$ is the the objective function in $\eqref{E:dual penalized problem}$, the dual problem of $\eqref{e:Tik}$, for $\lambda=\lambda_n$, while $d_{\infty}$ is the one appearing in \eqref{e:D}.

Since (3-D) is a diagonal forward-backward applied to the family of dual functions $(d_n)_\nin$, we will use classic properties of the forward-backward 
method to obtain estimates on $(u_n)_{\nin}$.
These estimates combined with the convergence of $d_n$ towards $d_\infty$ yield the convergence of $(u_n)_{\nin}$ to a solution of \eqref{e:D}.
We highlight the fact that the proof of these results can be related to the arguments used in \cite[Section 3.1]{AttCabCza16}.
Finally, from the strong duality between \eqref{e:P} and \eqref{e:D}, we will derive estimates on the primal sequence  $(x_n)_{\nin}$, and its convergence to a solution of \eqref{e:P}.

\begin{proposition}{(Energy estimate)}
\label{P:energy estimate}
Assume {\em (AD)} and {\em (AR)} and let $(x_n,w_n,u_n)_{\nin}$ be
the sequence generated by the {\em (3-D)} method.  
Then, for every $u \in Y$ and every  $\nin$,
\begin{equation*}
\frac{1}{2\tau} \Vert u_{n+1}-u \Vert^2 - \frac{1}{2\tau} \Vert u_n -u \Vert^2 \leq d_n(u) - d_n(u_{n+1}).
\end{equation*}
\end{proposition}

\begin{proof}
Let us introduce the notation $\Psi_n:=\frac{1}{\lambda_n}\psi_y$, $\Phi_n:=\frac{1}{\lambda_n} \phi_y$ and $R^*_A:= R^* \circ (-A^* )$.
From the definition of (3-D), we have $u_{n+1} = \prox_{\tau \Phi_n^* }(w_{n+1})$, which implies
\begin{eqnarray}
0 &\in & \partial \Phi_n^*(u_{n+1}) + \dfrac{u_{n+1}-w_{n+1}}{\tau} \label{ee1} \\
 &=& \partial \Phi_n^*(u_{n+1}) + \dfrac{u_{n+1}-u_{n}}{\tau} + \nabla R_A^*(u_n) + \nabla \Psi_n^*(u_n). \nonumber
\end{eqnarray}
By developing the squares, we obtain
\begin{eqnarray}
\frac{1}{2\tau} \Vert u_{n+1}-u \Vert^2 - \frac{1}{2\tau} \Vert u_n -u \Vert^2
=  -\frac{1}{2\tau} \Vert u_n - u_{n+1} \Vert^2 + \left\langle \frac{u_n - u_{n+1}}{\tau}, u - u_{n+1} \right\rangle.
\label{ee2} 
\end{eqnarray}
Once combined with \eqref{ee1}, this gives, for some $u_{n+1}^* \in \partial \Phi_n^*(u_{n+1})$,
\begin{eqnarray}\label{ee3}
 & \dfrac{1}{2\tau} \Vert u_{n+1}-u \Vert^2 - & \frac{1}{2\tau} \Vert u_n -u \Vert^2 \\
  = & -\dfrac{1}{2\tau} \Vert u_n - u_{n+1} \Vert^2 
 & + \left\langle u_{n+1}^*, u - u_{n+1} \right\rangle  + \langle \nabla (R_A^* +\Psi_n^*)(u_n),u-u_{n+1} \rangle. \nonumber
\end{eqnarray}
Convexity of $\Phi_n^*$ yields
\begin{equation}\label{ee4}
\left\langle u_{n+1}^*, u - u_{n+1} \right\rangle \leq \Phi_n^*(u) - \Phi_n^*(u_{n+1}).
\end{equation}
We  also have
\begin{eqnarray*}\label{ee5}
  \langle   \nabla (R_A^* +\Psi_n^*)(u_n), u-u_{n+1} \rangle
 = \langle  \nabla(R_A^* +\Psi_n^*)(u_n),  u-u_{n} \rangle 
 - \langle \nabla(R_A^* +\Psi_n^*)(u_n),u_{n+1} - u_n \rangle,
\end{eqnarray*}
where the convexity of $R_A^* +\Psi_n^*$ gives
\begin{eqnarray}
\langle \nabla(R_A^* +\Psi_n^*)(u_n),u-u_{n} \rangle 
\leq (R_A^* +\Psi_n^*)(u) - (R_A^* +\Psi_n^*)(u_n), \label{ee6}
\end{eqnarray}
and the Descent Lemma \cite[Lem. 1.30]{Pey} applied to $R_A^* +\Psi_n^*$ (whose gradient is $L$-Lipschitz continuous) 
implies
\begin{eqnarray}
 - \langle \nabla(R_A^* +\Psi_n^*)(u_n),u_{n+1} - u_n \rangle
\leq  \frac{L}{2}\Vert u_{n+1} - u_n \Vert^2 - (R_A^* +\Psi_n^*)(u_{n+1}) + (R_A^* +\Psi_n^*)(u_n).  \label{ee7} 
\end{eqnarray}
By inserting \eqref{ee4}, \eqref{ee6} and \eqref{ee7} into \eqref{ee3}, we finally obtain
\begin{eqnarray}\label{ee8}
\frac{1}{2\tau} \Vert u_{n+1}-u \Vert^2 - \frac{1}{2\tau} \Vert u_n -u \Vert^2 
  \leq  \left( \frac{L}{2} - \frac{1}{2\tau} \right) \Vert u_n - u_{n+1} \Vert^2 
 + d_n(u) - d_n(u_{n+1}). 
\end{eqnarray}
The conclusion follows from the assumption $\tau \leq \frac{1}{L}$.
\end{proof}

\begin{proposition}{(Dissipativity)}
\label{P:dissipativity}
Assume {\em (AD)} and {\em (AR)} and let $(x_n,w_n,u_n)_{\nin}$ be
the sequence generated by the {\em (3-D)} method.  
Then,
\begin{itemize}
	\item[(i)] for all $u \in Y$, $d_n(u) \downarrow d_\infty(u)$ as $n \to + \infty$,
	\item[(ii)] $d_n(u_{n+1}) \downarrow \inf d_\infty  \in \R$ as $n \to + \infty$.
\end{itemize} 
\end{proposition}

\begin{proof}
(i): Let $u \in Y$.
It is enough to show that the real-valued function
\begin{equation}\label{d1}
\lambda \in ]0,+\infty[ \  \mapsto \frac{1}{\lambda} D_{\bar y}^*(\lambda u)
\end{equation}
is increasing in $\lambda$, and converges to $\langle \bar y,u \rangle$ when $\lambda \to 0$.
Since $D_{\bar y}^*(0)= - \inf D_{\bar y} =0$ by (AD1), the function in \eqref{d1} can be rewritten as
\begin{equation}\label{d2}
 \lambda \in ]0,+\infty[ \ \mapsto \dfrac{D_{\bar y}^*(0 + \lambda u) - D_{\bar y}^*(0)}{\lambda} .
\end{equation}
Convexity of $D_{\bar y}^*$ implies that the quotient in~\eqref{d2} is increasing in $\lambda$~\cite[Prop. 17.2]{BauCom}.
Moreover, the limit of this quotient when $\lambda \to 0$ is, by definition, $\frac{\d D_{\bar y}^*}{\d u}(0)$, the 
directional derivative of $D_{\bar y}^*$ at zero, in the direction $u$.
Assumption (AD3) and \cite[Prop. 14.16, Prop. 16.21 \& Thm. 17.19]{BauCom} implies that $\frac{\d D_{\bar y}^*}{\d u}(0)$
equals the support function of $\partial D_{\bar y}^*(0)$ evaluated at $u$.
But here $\partial D_{\bar y}^*(0) = \argmin D_{\bar y}=\{\bar y\}$ by (AD1), 
which means that the quotient in \eqref{d2} tends to $\langle \bar y,u \rangle$ when $\lambda \to 0$.

(ii): First, we observe that $\inf d_\infty > - \infty$.
To see this, use the Fenchel-Young inequality together with (AR2) to write, for all $u \in Y$:
\begin{eqnarray*}
 d_\infty(u) & = & R^*(-A^*u) + \langle \bar y  , u \rangle \\
  & = &   R^*(-A^*u) -  \langle \bar x  , -A^* u \rangle  \\
  & \geq & - R(\bar x ) >-\infty.
\end{eqnarray*}
Define now, for all $n \geq 1$, $r_n:= d_{n-1}(u_{n}) - \inf d_\infty$, and let us show that $r_n \downarrow 0$.
First, apply Proposition \ref{P:energy estimate} with $u=u_n$ to obtain
\begin{equation}\label{d3}
0 \leq \frac{1}{2\tau } \Vert u_{n+1} - u_n \Vert^2 \leq d_{n}(u_n) - d_n(u_{n+1}).
\end{equation}
Since we showed that the sequence $(d_n)_{\nin}$ is decreasing, $d_n(u_n) \leq d_{n-1}(u_n)$ for every $n\in\N$. 
We deduce then from \eqref{d3} that, for every $\nin$, $0 \leq r_n - r_{n+1}$,
meaning that $(r_n)_{\nin}$ is a decreasing sequence.
It follows from  (i)  that $d_n \downarrow d_\infty$, therefore $r_n \geq d_\infty(u_n) - \inf d_\infty \geq 0$.
So there exists some positive real $\ell \geq 0$ such that $r_n \downarrow \ell$. 

Let us finish the proof by showing that $\ell \leq 0$.
Let $u \in Y$. Proposition~\ref{P:energy estimate} yields, for every $\nin$,
\begin{equation*}\label{d5}
\frac{1}{2\tau} \Vert u_{n+1}-u \Vert^2 - \frac{1}{2\tau} \Vert u_n -u \Vert^2 \leq d_n(u) - d_n(u_{n+1}).
\end{equation*}
Do a telescopic sum on the above inequality and divide by $k \in \N^*$ to derive
\begin{equation}\label{d6}
\frac{1}{k}\sum\limits_{n=0}^{k} d_n(u_{n+1}) - d_n(u) \leq \frac{1}{2\tau k}\Vert u_0 - u \Vert^2.
\end{equation}
On the one hand, the right hand side of \eqref{d6} tends to zero when $k \to + \infty$.
On the other hand, we saw that $d_n(u_{n+1}) - d_n(u)$ tends to $\ell + \inf d_\infty - d_\infty(u)$ when $n \to +\infty$.
By Cesaro's lemma, we can pass to the limit in \eqref{d6} to obtain
\begin{equation*}
\ell + \inf d_\infty - d_\infty(u) \leq 0.
\end{equation*}
Since this inequality is true for any $u \in Y$, we deduce that $\ell \leq 0$.
\end{proof}

In the following, we will need an estimate of the rate of convergence of
$d_n$ to $d_\infty$, in particular once evaluated at some element of $\argmin d_\infty$.
For this, we will exploit the geometry of the data-fit function $D_{\bar y}$, through 
its conditioning modulus $\bar{m}$.

\begin{lemma}
\label{L:upper bound growth condition}
Let {\em (AD)} and {\em (AR)} hold and $(\lambda_n)_{\nin}$ be as in the {\em(3-D)} method.
For every $u^\dagger \in \argmin d_\infty$ and every $\nin$,
\begin{equation}\label{E:upper bound growth condition}
d_n(u^\dagger) - d_\infty(u^\dagger) \leq \frac{1}{\lambda_n} \bar m^*(\Vert u^\dagger \Vert \lambda_n ).
\end{equation}
\end{lemma}

\begin{proof}
From the definition of $d_n$ and $d_\infty$, we have
\begin{equation}\label{ubg1}
d_n(u^\dagger) - d_\infty(u^\dagger) = \frac{1}{\lambda_n}\left( D_{\bar y}^*(\lambda_n u^\dagger) - \langle {\bar y}, \lambda_n u^\dagger \rangle \right).
\end{equation}
It follows from the definition of conditioning modulus \eqref{E:growth condition}  that, for every $u \in Y$, $D_{\bar y}(u) \geq \bar m(\Vert u-{\bar y}\Vert)$.
This implies, for every $u \in Y$, $D_{\bar y}^*(u) \leq \left(\bar m(\Vert \cdot - {\bar  y} \Vert)\right)^*(u)$ (see \cite[Prop. 13.14]{BauCom}).
Since $\bar m : \R \rightarrow [0,+\infty]$ is an even function, we derive from \cite[Prop. 13.20 \& Ex. 13.7]{BauCom} that
\[
\left(\bar m(\Vert \cdot - {\bar y} \Vert)\right)^*(u) = \bar m^*(\Vert u \Vert) + \langle {\bar y},u \rangle. 
\]
The result follows by taking $u=\lambda_n u^\dagger$.
\end{proof}

\begin{lemma}
\label{L:summability lambda modulus}
If {\em (AD)} holds, then $0 \in \inte \dom \bar{m}^*$.
Moreover, if $(\lambda_n)_{\nin}$ is used in {\em (3-D)} and satisfies $(\lambda_n)_{\nin}\in \ell^{1/(p-1)}(\N)$,
then for every $r\in\R_{++}$, and for every $N \in \N$ such that 
$r \lambda_N \in \inte \dom m^*$, we have:
\[
\sum\limits_{n=N}^{+\infty} \frac{1}{\lambda_n}\bar{m}^*(r \lambda_n) \ < + \infty.
\]
\end{lemma}

\begin{proof}
Assumption (AD) implies  that $\argmin\bar{m} = \{0\}$. 
From  \cite[Prop. 11.12 \& 14.16]{BauCom}, it follows that   $0 \in \inte \dom \bar{m}^*$.
Let $r\in\R_{++}$ and let $N \in \N$ be such that $r \lambda_N \in \inte \dom \bar{m}^*$.
Note that, since $\lambda_n \downarrow 0$, we  have $r \lambda_n \in \inte\dom \bar{m}^*$ for every $n \geq N$.
(AD3) implies that $\bar{m} \geq f$ on $]-\eps,\eps[$, where here $f:={\gamma}\vert \cdot \vert^p/p$.
We derive from Lemma~\ref{L:Fenchel local inequality} (see Appendix) that there exists $\eps'\in\R_{++}$ 
such that $m^* \leq f^*$ on $]-\eps', \eps'[$.
This allows us to easily estimate $S_N(r):= \sum_{n=N}^{+\infty} \frac{1}{\lambda_n}\bar{m}^*(r \lambda_n)$, 
by considering the two cases $p=1$ and $p>1$.\\
If $p=1$, let $M \geq N$ be an integer such that $r \lambda_M < \min\{\gamma,\eps'\}$.
Since in that case $f^*(t)= \delta_{[-\gamma,\gamma]}(t)$, we directly see that
\[
S_N(r) \leq \sum\limits_{n=N}^{M} \frac{1}{\lambda_n}\bar{m}^*(r \lambda_n) <+ \infty.
\]
If $p>1$, let $M \geq N$ be an integer such that $r \lambda_M < \eps'$.
In this case, $f^*(t) = \frac{\gamma^{1-q}}{q} \vert t \vert^q$, where $q=p/(p-1)$.
Then we deduce that
\[
S_N(r) \leq \frac{\gamma^{1-q}}{q} \sum\limits_{n=N}^{M} \frac{1}{\lambda_n}\bar{m}^*(r \lambda_n) + \frac{\gamma^{1-q}r^q}{q} \sum\limits_{n=m}^{+\infty} \lambda_n^{\frac{1}{p-1}},
\]
where the last sum is finite since $\lambda_n\in \ell^{1/(p-1)}(\N)$.
\end{proof}

We are now ready to prove Theorem \ref{T:mainconv}, whose proof  using two main ingredients.
First, we will use the estimations we made on the sequence $u_n$ to prove its weak convergence, 
thanks to the celebrated Opial's lemma (see \cite[Lemma 5.2]{Pey} for a proof):

\begin{lemma}{(Opial)}\label{L:Opial}
Let $S$ be a subset of a Hilbert space $H$, and $(x_n)_{\nin}$ be a sequence in $H$.
Assume that
\begin{enumerate}
	\item for all $x \in S$, the real sequence $(\Vert x_n - x \Vert)_{\nin}$ admits a limit,
	\item every weak limit point of $(x_n)_{\nin}$ belong to $S$.
\end{enumerate}
Then $S \neq \emptyset$ if and only if $(x_n)_{\nin}$ is bounded.
In such a  case,  $(x_n)_{\nin}$ weakly converges to some element  belonging to $S$.
\end{lemma}
Second, we will exploit the strong duality between \eqref{e:P} and \eqref{e:D} to recover 
strong convergence for the primal sequence $(x_n)_{\nin}$ through estimations 
made on the dual one $(u_n)_{n\in\N}$. 
The key result is the following lemma (whose proof is in the Appendix):

\begin{lemma}{(Primal-dual values-iterates bound)}\label{L:primal-dual value-iterate bound}
Let $f \in \Gamma_0(X)$ be $\sigma$-strongly convex, $g \in \Gamma_0(Y)$ and $A\colon X\to Y$
be a bounded linear operator.
Let $x^\dagger$ be the unique minimizer of $p:=f+g\circ A$, and let $d:= f^*\circ(-A^*)+ g^*$.
Then
\begin{equation*}
\argmin d \neq \emptyset \ \Leftrightarrow \ 0 \in \partial f(x^\dagger) + A^* \partial g(A x^\dagger).
\end{equation*}
In that case, for every $u \in Y$ and every $x:=\nabla f^*(-A^*u)$, we have
\begin{equation}\label{E:primal-dual value-iterate bound}
\frac{\sigma}{2} \Vert x - x^\dagger \Vert^2 \leq d(u) - \inf\limits_{u \in Y} d.
\end{equation}
\end{lemma}

\begin{proof}[of Theorem \ref{T:mainconv}]
(i)$\iff$(ii): The equivalence  follows directly from Lemma \ref{L:primal-dual value-iterate bound} with $f=R$ and $g=\delta_{\{y\}}$.

(ii)$\iff$(iii): To prove this equivalence, together with the weak convergence of $(u_n)_{\nin}$ 
towards a minimizer of $d_\infty$, we will apply Opial's lemma with $f=d_\infty$ and $S=\argmin d_\infty$.
We thus only have to verify hypotheses $(1)$ and $(2)$ of Opial's lemma.
We start with hypothesis $(1)$ of Opial's lemma.
Without loss of generality, we can assume $S\neq\varnothing$. Let $u^\dagger \in S$, and let us 
show that the sequence $h_n:= \frac{1}{2\tau} \Vert u_n - u^\dagger \Vert^2$ admits a limit when $n \to +\infty$.
Using successively Propositions \ref{P:energy estimate}, \ref{P:dissipativity},
and Lemma \ref{L:upper bound growth condition}, we obtain
\begin{eqnarray}
 & & h_{n+1} - h_n \label{cod1}\\
  & \leq & d_n(u^\dagger) - d_\infty(u^\dagger) + d_\infty(u^\dagger) - d_n(u_{n+1}) \nonumber \\
& \leq & d_n(u^\dagger) -  d_\infty (u^\dagger) \nonumber \\
 & \leq &  \frac{1}{\lambda_n}\bar m^*(\lambda_n \Vert u^\dagger \Vert).\nonumber
\end{eqnarray}
Lemma~\ref{L:summability lambda modulus} implies that $(h_n)_{\nin}$ is a quasi-Fej\'er sequence, 
and therefore $(h_n)_{\nin}$ is convergent (see for instance \cite[Lem. 3.1]{Com01}).
We now turn to hypothesis $(2)$ of Opial's Lemma: assume that there exists a subsequence $(u_{n_k})_{k\in\N}$
converging weakly to some $u_\infty\in Y$.
By using the lower-semicontinuity of $d_\infty$, we obtain
\begin{equation}\label{cod2}
d_\infty(u_\infty) \leq \liminf_{k \to + \infty} d_\infty(u_{n_k}).
\end{equation}
Moreover, we know from Proposition \ref{P:dissipativity} that $d_n \downarrow d_\infty$, so 
$d_\infty(u_{n_k}) \leq d_{n_k -1}(u_{n_k})$. This, together with the fact that $d_{n-1}(u_n) \to \inf d_\infty$, 
implies that \eqref{cod2} is equivalent to $d_\infty(u_\infty) \leq \inf d_\infty$, meaning that $u_\infty \in S$.

Next, we  focus on the strong convergence of the primal sequence $(x_n)_{\nin}$.
Let $u^\dagger$ be any solution of \eqref{e:D}.
Use Lemma \ref{L:primal-dual value-iterate bound} with $f=R$ and $g=\delta_{\{y\}}$, together with Proposition \ref{P:dissipativity} to obtain
\begin{equation}\label{cod4}
\frac{\sigma_R}{2}\Vert x_n - x^\dagger \Vert^2 \leq d_{n-1}(u_n) - d_\infty(u^\dagger).
\end{equation}
The goal is to obtain an estimate on the rate of convergence to zero of $r_n:=d_{n-1}(u_n) - d_\infty(u^\dagger)$.
By using the same argument as  in \eqref{cod1}, we obtain, for every $\nin$,
\begin{equation}\label{cod3}
h_{n+1} - h_n \leq  \frac{1}{\lambda_n} \bar m^*(\lambda_n \Vert u^\dagger \Vert) - r_{n+1}.
\end{equation}
Lemma \ref{L:summability lambda modulus} ensures that there exists some  $N \in \N$  such that $\Vert u^\dagger \Vert \lambda_N \in \inte \dom \bar m^*$, and also that such integer verifies
\begin{equation}\label{cod5}
S_N(\Vert u^\dagger \Vert):= \sum_{n=N}^{+\infty} \frac{1}{\lambda_n}\bar{m}^*(\Vert u^\dagger \Vert \lambda_n) < + \infty .
\end{equation}
As a consequence, a telescopic sum on \eqref{cod3} gives
\begin{equation*}
\sum\limits_{n=N}^{+\infty} r_{n+1} \leq h_{N} + S_{N}(\Vert u^\dagger \Vert) <+\infty.
\end{equation*}
From Proposition \ref{P:dissipativity} it follows that $r_n$ is decreasing and positive, therefore
\begin{equation*}
0 \leq n r_{2n} \leq \sum\limits_{k=n}^{2n} r_k \overset{n\to +\infty}{\longrightarrow} 0,
\end{equation*}
which means that $r_n=o\left(n^{-1}\right)$.
This, together with \eqref{cod4}, implies that $\Vert x_n - x^\dagger \Vert = o\left( n^{-\frac{1}{2}} \right)$.

To obtain the rates \eqref{E:CV of 3-D precise rates}, we will do a  similar analysis.
Let $\varepsilon_n:=(n-N)r_n + h_n$.
Then, for every $n \geq N$, the inequality $r_{n+1} \leq r_{n}$ and \eqref{cod3} yield
\[
\begin{array}{rcl}
\varepsilon_{n+1} - \varepsilon_n& = & (n- N +1) r_{n+1} - (n- N) r_n + h_{n+1} - h_n \\
 & \leq & r_{n+1} + h_{n+1} - h_n , \\
 & \leq & \frac{1}{\lambda_n} \bar m^*(\lambda_n \Vert u^\dagger \Vert).
\end{array}
\]
Therefore,
\begin{eqnarray*}
(n-N)r_n \leq \varepsilon_n = \varepsilon_{N} + \sum\limits_{k=N}^{n-1} \varepsilon_{k+1} - \varepsilon_k 
\leq  h_{N} + S_{N}(\Vert u^\dagger \Vert).
\end{eqnarray*}
Dividing by $(n-N)$ and using \eqref{cod4}, we finally obtain \eqref{E:CV of 3-D precise rates}.
\end{proof}

\begin{remark}[On the non slow-decay assumption on $\lambda_n$]
\label{R:non L1 discussion}
\noindent A key point in our proof is the fact that we perform a diagonal descent 
method on a sequence of functions $(d_n)_{\nin}$ which is monotonically \textit{decreasing} to $d_\infty$.
This property ensures the Mosco convergence of $(d_n)_{\nin}$ to $d_\infty$, which is essential for viscosity methods \cite{Lem88,AttCabCza16}.
This decreasing property might explain the fact that we do not require $(\lambda_n)_{\nin} \notin \ell^1(\N)$, 
which is instead a standard assumption for diagonal primal methods \cite{Lem95,Pey11,AttCzaPey11,CzaNouPey14}.  
The rationale behind this might be that we do not need to make the link between \eqref{e:Tik} and 
\begin{equation}
\label{e:PC}
\tag{$\text{\u{P}}_\lambda$}  \min \lambda R(x) + D(Ax;y).
\end{equation}
In fact, while \eqref{e:Tik} and \eqref{e:PC} are trivially equivalent for fixed $\lambda$,  things change if $\lambda$  is allowed to move.
When $\lambda \downarrow 0$, the function $ R + \lambda^{-1} D(A \cdot ; y)$ is monotonically increasing to $R + \delta_{y}(Ax)$, while $ \lambda R + D(A \cdot ; y)$ is monotonically decreasing to $\delta_{\dom R} + D(A \cdot;y)$.
So \eqref{e:Tik} converges towards the problem we are interested in (the one in \eqref{e:P}), but it is not decreasing,
while  \eqref{e:PC}   has the desired decreasing property, but converges to the ``wrong'' problem.
To pass from one model to the other, it is necessary for primal diagonal schemes to perform an appropriate change of variable (see \cite[Section 1.2]{Cab04} or \cite[Section 4]{AttCza10}), 
which requires the assumption that $\lambda$ doesn't tend to zero too fast: whence the assumption $(\lambda_n)_{\nin} \notin \ell^1(\N)$ 
(see also \cite[Thm. 2]{ComPeySor08} and the following remark).
In our dual diagonal scheme, we  have the combination of the two desirable properties at the same time: indeed $(d_n)_{\nin}$ is decreasing and
 \eqref{E:dual penalized problem} converges towards the dual of \eqref{e:P}.
\end{remark}

\begin{remark}[On the Diagonal Landweber algorithm]\label{R:diagonal Landweber non L1}

\noindent In light of the previous remark, it is interesting to look again at the Diagonal Landweber algorithm.
As discussed in Remark \ref{Ex:diagonal mirror descent}, the Diagonal Landweber algorithm can be seen as a primal diagonal scheme, 
and  $(\lambda_n)_{\nin} \notin \ell^1(\N)$ is generally assumed to ensure the convergence of $(x_n)_{\nin}$ to $x^\dagger$, which is the minimal norm solution of $Ax=\bar y$.
Otherwise, it is known that without this assumption, the regularizer $R=(1/2)\Vert \cdot \Vert^2$ is ignored, and the sequence might converge to any other solution of $Ax=\bar y$ (see \cite[Proposition 1.2]{Cab04} or \cite[Theorem 2]{ComPeySor08}).
On the other hand, as we mentioned in Remark~\ref{Ex:diagonal mirror descent}, Diagonal Landweber can also be seen as a realization of (3-D), 
where we require $(\lambda_n)_{\nin}\in \ell^1(\N)$ to get convergence to $x^\dagger$.
This seems contradictory at a first sight with the above discussion, and one might wonder why the limit point of the sequence is indeed $x^\dagger$.
In fact, as observed in Remark \ref{Ex:diagonal mirror descent}, (3-D)  requires that we initialize the algorithm with $x_0\in\im A^*$.
This implies that the generated sequence $(x_n)_{\nin}$ will belong to $\im A^*$, which is orthogonal to the affine space of solutions $\{x\in X\,|\, Ax=\bar{y}\}$.
As a consequence, $(x_n)_{n\in\N}$ can only converge to a solution of $Ax=\bar{y}$ belonging also to $\overline{\im A^*}$, which is exactly $x^\dagger$. 
\end{remark}

\section{(3-D) as an iterative regularization procedure}\label{S:regularization}

In this section, $(x_n,w_n,u_n)_{\nin}$ and $(\hat x_n, \hat w_n,\hat u_n)_{\nin}$ are generated by the (3-D) algorithm, 
using the exact data $\bar y$ and the noisy ones $\hat y$, respectively.
We assume here that both sequences have the same initialization.

As suggested by~\eqref{E:optimization/error dichotomy}, showing that (3-D) acts as an iterative regularization procedure
requires a stability estimate, which controls the error propagation $(\|\hat{x}_n-x_n\|)_{\nin}$ in the presence of noisy data.
Analogously to what happen for the classical Landweber iteration \cite{EngHanNeu96}, this can be 
bounded in terms of  the number of iterations and an estimate of the noise. 

\begin{proposition}[Stability]
\label{T:stability 3-D}
Let assumptions (AR) and (AD) hold. Let $\delta\in\R_+$, let $\theta\in\R_{+}$, and let $\hat{y}\in Y$  be
a $(\delta,\theta)$-perturbation of $\bar{y}$ according to $D$.
Then, for all $n\in\N^*$,
\[
\Vert x_n - \hat x_n \Vert \leq  
 \dfrac{ \delta}{\| A \|} \left(  n+\tau^{\theta - 1}\sum\limits_{k=0}^{n-1} \lambda_k^{-\theta}\right).
\]
\end{proposition}

\begin{proof}
We introduce the notation $ \Psi_n:= \lambda_n^{-1} \psi_{\bar y}$, $ \Psi_n:= \lambda_n^{-1} \phi_{\bar y}$, together with their noisy counterpart $\hat \Psi_n:=\lambda_n^{-1} \psi_{\hat y}$ and $\hat \Phi_n:= \lambda_n^{-1} \phi_{\hat y}$.
By definition of (3-D), and using the triangle inequality:
\begin{eqnarray*}\label{sd1}
 \Vert u_{n+1} - \hat u_{n+1} \Vert
\leq 
\Vert \prox_{\tau    \Phi_n^* }(w_{n+1}) - \prox_{\tau  \Phi_n^* }(\hat w_{n+1}) \Vert
+ \Vert \prox_{\tau \Phi_n^* }(\hat w_{n+1}) - \prox_{\tau \hat \Phi_n^* }(\hat w_{n+1})\Vert. \nonumber 
\end{eqnarray*}
Nonexpansivity of the $\prox_{\tau   \Phi_n^* }$ \cite[Prop. 12.27]{BauCom}, together with the assumption on the noise \eqref{E:stability phi} and Lemma~\ref{L:gradient dual prox formula}, implies
\begin{equation}\label{sd2}
\Vert u_{n+1} - \hat u_{n+1} \Vert \leq \Vert w_{n+1} - \hat w_{n+1} \Vert + \lambda_n^{-\theta} \tau^\theta \delta.
\end{equation}
Let us introduce the notation $R^*_A:=R^* \circ (-A^*)$, and define 
$$T_n : Y \rightarrow Y, \quad u \mapsto T_n u:= u - \tau (\nabla R^*_A(u) + \nabla \psi^*_n(\lambda_n u)).$$
Then we have from the definition of (3-D):
\begin{equation*}\label{sd3}
w_{n+1} - \hat w_{n+1} = T_n u_n - T_n \hat u_n - \tau (\nabla \psi^*_n(\lambda_n \hat u_n) - \nabla \hat \psi^*_n(\lambda_n \hat u_n)).
\end{equation*}
Using the assumption on the noise \eqref{E:stability psi}, we can write
\begin{equation*}
\Vert w_{n+1} - \hat w_{n+1} \Vert \leq \Vert T_n u_n - T_n \hat u_n \Vert + \tau \delta.
\end{equation*}
Since $\nabla R^*_A + \nabla  \psi^*_n(\lambda_n \cdot)$ is $L$-Lipschitz continuous, and because it is assumed in (3-D) that $\tau \leq L^{-1}$, we deduce that $T_n$ is a non-expansive operator \cite[Theorem 18.15]{BauCom}, leading  to the estimation
\begin{equation}\label{sd4}
\Vert w_{n+1} - \hat w_{n+1} \Vert \leq \Vert u_n - \hat u_n \Vert + \tau \delta.
\end{equation}
By combining \eqref{sd2} and \eqref{sd4}, we obtain for all $n \geq 1$
\begin{equation*}
\Vert u_n - \hat u_n \Vert \leq \Vert u_0 - \hat u_0 \Vert + \delta \tau n + \delta \tau^\theta \sum\limits_{k=0}^{n-1} \lambda_k^{-\theta}
\end{equation*}
The fact that $x_n=\nabla R^*(-A^* u_n)$, where $\nabla R^* \circ (-A^*)$ is $\frac{\| A \|}{\sigma_R}$-Lipschitz continuous, implies
\begin{equation*}
 \frac{\sigma_R}{\Vert A \Vert}\Vert x_n - \hat x_n \Vert\leq \Vert u_0 - \hat u_0 \Vert + \delta \tau n + \delta \tau^\theta \sum\limits_{k=0}^{n-1} \lambda_k^{-\theta}
\end{equation*}
Since $ \tau \leq \frac{\sigma_R}{\| A \|^2}$, the  conclusion follows.
\end{proof}
We are now ready to prove our main stability result.
\begin{proof}[of Theorem~\ref{thm:reg}] 
Theorem \ref{T:mainconv} ensures the existence of a solution $u^\dagger \in Y $ for the dual problem \eqref{e:D}.
It follows from Lemma~\ref{L:summability lambda modulus} that there exists $N \in \N$ such that $\Vert u^\dagger \Vert \lambda_N \in \inte \dom \bar m^*$. Then 
we derive from Theorem~\ref{T:mainconv} that 
\[
(\forall n > N)\quad \ \Vert x_n - x^\dagger \Vert \leq \dfrac{b}{ \sqrt{n-N}},  
\]
with $b:=\Vert u_{N} - u^\dagger \Vert^2/(\tau \sigma_R) + S_{N}(\Vert u^\dagger \Vert)/\sigma_R$, where $ S_{N}(\Vert u^\dagger \Vert)$ is defined in \eqref{cod5}.
On the other hand, from Proposition \ref{T:stability 3-D} and the hypothesis on $(\lambda_n)_{n\in\N}$, we have
\[
(\forall \nin)\quad
\Vert x_n - \hat x_n \Vert \leq  
\delta a n^{1+\beta \theta},
\]
with $a=(1+\lambda_0^{-\theta}\tau^{\theta - 1})\Vert A \Vert^{-1}$.
Thus, for every $n > N$,
\begin{equation}\label{trg1}
\Vert \hat{x}_{n} -  x^\dagger \Vert \leq \delta a n^{1+\beta \theta} + \dfrac{b}{ \sqrt{n-N}}.
\end{equation}
The idea now is to derive an early stopping rule $n(\delta)$ by minimizing the right hand side of \eqref{trg1}.
We will achieve this by considering, for $\alpha := \beta\theta$, $T:=N$ and $\delta\in\R_{++}$, the real valued function
\[
f_\delta : t \in ]T , +\infty[ \mapsto f_\delta(t):= \delta a t^{1+\alpha} + \dfrac{b}{\sqrt{t-T}}.
\]
First observe that $f_\delta$ is convex, so we can characterize its minimizers with the Fermat's rule.
This function is differentiable on $]T,+\infty[$, and $f_\delta'(t)=0$ if and only if
\begin{equation}\label{fsr1}
t^\alpha(t-T)^{3/2} = C_1 \delta^{-1}, \ \text{with } C_1:=\frac{b}{2a(1+\alpha)}.
\end{equation}
The function $\eta(t):=t^\alpha(t-T)^{3/2}$ is strictly increasing on $]T,+\infty[$, and is a bijection between $]T,+\infty[$ and $]0,+\infty[$.
So we deduce from \eqref{fsr1} that there exists a unique minimizer for $f_\delta$, given by $t(\delta):=\eta^{-1}\left(C_1 \delta^{-1} \right)$.
Moreover, we also deduce from the relation $\eta(t(\delta))=C_1\delta^{-1}$  that $\delta \mapsto t(\delta)$ is decreasing, and that $t(\delta) \uparrow +\infty$ when $\delta \downarrow 0$.

Now we define an early stopping rule by taking $n(\delta):=\lceil ct(\delta) \rceil$, for some fixed $c \geq 1$, and we want to estimate $f_\delta(n(\delta))$.
For this, start by writing $n(\delta)=c(\delta) t(\delta)$, where $c(\delta) =\frac{\lceil ct(\delta) \rceil}{t(\delta)}$.
From now we assume that $t(\delta) \geq T+1$, which is achieved as soon as $\delta$ is small enough, since  $t(\delta) \uparrow +\infty$.
In particular, we deduce that $c(\delta) \in [c,c+1]$, and this implies that
\begin{eqnarray*}
f_\delta(n(\delta)) & = & \delta a (c(\delta) t(\delta))^{\alpha + 1 } + \frac{b}{\sqrt{c(\delta)t(\delta) - T}} \\
 & \leq & \delta a ((c+1)t(\delta))^{\alpha + 1 } + \frac{b}{\sqrt{t(\delta) - T}}.
\end{eqnarray*}
Now we can use \eqref{fsr1} to write
\begin{equation*}
\frac{1}{\sqrt{t(\delta)-T}}=C_1^{-1/3} \delta^{1/3} t(\delta)^{\frac{\alpha}{3}} ,
\end{equation*}
which gives in turn
\begin{equation}\label{fsr3}
f_\delta(n(\delta))  \leq \delta a ((c+1)t(\delta))^{\alpha + 1 } + bC_1^{-1/3}\delta^{1/3}t(\delta)^{\alpha/3} .
\end{equation}

Now we need to estimate $t(\delta)$.
Let $\gamma:= \frac{2}{3+2\alpha}$; by using \eqref{fsr1}, a simple computation shows that
\begin{equation*}
\frac{\delta^{-\gamma}}{t(\delta)}=C_1^{-\gamma}\left(1-\frac{T}{t(\delta)}\right)^\frac{3}{3+2\alpha}.
\end{equation*}
But we assumed that $t(\delta) \geq T+1$, so 
\begin{equation*}
\frac{\delta^{-\gamma}}{t(\delta)}\geq C_1^{-\gamma}\left(\frac{1}{T+1}\right)^\frac{3}{3+2\alpha},
\end{equation*}
which gives in turn
\begin{equation*}\label{fsr2}
t(\delta) \leq C_2 \delta^{\frac{-2}{3+2\alpha}}, \text{ with } C_2:=C_1^{\frac{2}{3+2\alpha}} (T+1)^{\frac{3}{3+2\alpha}}.
\end{equation*}
The above inequality together with \eqref{fsr3} finally gives
\begin{equation*}
f_\delta(n(\delta))  \leq C_3 \delta^{\frac{1}{3+2\alpha}}, 
\end{equation*}
with $C_3:= a(1+c)^{\alpha +1}C_2^{\alpha +1 } + b C_1^{-1/3}C_2^{\alpha/3}$.
\end{proof}

\section{Numerical results: Deblurring and denoising}\label{S:numerical}

\def\scalegraph{0.15}
\def\scaleimage{0.25}
\def\widthimage{0.19\linewidth}
\def\widthgraph{0.24\linewidth}

\def\widthimagethree{0.32\linewidth}
\def\widthimagefour{0.24\linewidth}
\def\widthimagefive{0.19\linewidth}

In this section we perform  several numerical experiments using  the (3-D) algorithm for image denoising and deblurring. 
We consider problems of the form \eqref{e:P}, involving a data-fit function selected according to the nature of the noise, and  a regularizer promoting some  desired property of the solution.
For all  the experiments, the linear operator $A$ is a blurring operator defined by a Gaussian kernel of size $9\times 9$ and variance $10$.
In  our experiments, we compare two \textit{versions} of (3-D), corresponding to two different choices of the sequence $(\lambda_n)_{n\in\N}$: an online choice, and an a priori choice.

For the online approach, we use the {warm restart} method described in Example~\ref{Ex:warm restart},  called \textit{warm 3D} in the following.
In this case, the sequence $(\lambda_n)_{n\in\N}$ is taken to be piecewise constant, and its decay is determined by a stopping rule.
In practice, we take $N_{wr}$ regularization parameters $\{\Lambda_1,\ldots, \Lambda_{N_{wr}}\}$ uniformly distributed on a logarithmic scale in an interval $[\lambda_{min},\lambda_{max}]$. 
Then, we start with $\lambda_1=\Lambda_1=\lambda_{max}$,
 and, for every $\lambda\in\{\Lambda_i\}_{i=1}^{N_{wr}}$,  we set $d_\lambda(u):=R^*(-A^*u)+D^*_{\hat y}(\lambda u)/\lambda$, and we keep $\lambda_n=\lambda$
 until the stopping rule 
\begin{equation}
\label{e:istopru}
\left\vert \frac{d_\lambda(u_n)-d_\lambda(u_{n-1}) }{d_\lambda(u_n)} \right\vert < \eps_{wr}
\end{equation}
is verified.
This warm 3D method can be considered as a benchmark, since it is one of the most efficient ways to approximate the regularization path corresponding to Tikhonov regularization~\cite{BecBobCan11}.

For the a priori choice, that we call hereafter  \textit{vanilla 3D} method, we consider a strictly decreasing sequence  $(\lambda_n)_{n\in\N}$.
In practice, we take $N_{v}$ regularization parameters $\{\Lambda_1,\ldots, \Lambda_{N_{v}}\}$ uniformly distributed on a logarithmic scale in an interval $[\lambda_{min},\lambda_{max}]$, and set 
for every $n\in\{1,\ldots,N_v\}$, $\lambda_n = \Lambda_n$.
Observe that this choice makes $\lambda_n$ an exponentially decreasing sequence:
\[
\lambda_n = \lambda_{max} \left( \frac{\lambda_{min}}{\lambda_{max}} \right)^{\frac{t-1}{N_v-1}}.
\]
This implies for instance that Theorem \ref{T:mainconv} applies for any choice of loss function. 
Concerning Proposition~\ref{T:stability 3-D}, we already discussed in Remark \ref{R:choice of beta for stability rates} the fact that for most loss functions, no assumptions are required on $\lambda_n$.
Only the Kullback-Leibler divergence requires a slow decreasing sequence to ensure the existence of an early sopping rule in polynomial time.
In practice,  no significant difference was observed with the use of the Kullback-Leibler loss.

\subsection{Introductory example}\label{SS:numerical intro}

\begin{example}\label{Ex:simple example}
We illustrate the behavior of the (3-D) method.
We take $\bar x$ as a $512\times 512$ grayscale image, which is blurred and corrupted by a salt and pepper noise of intensity $35\%$
(see Figure~\ref{F:simple example problem}).
\begin{figure}[h]
  \centering
\[
\begin{array}{ccccc}
\includegraphics[width=\widthimagethree]{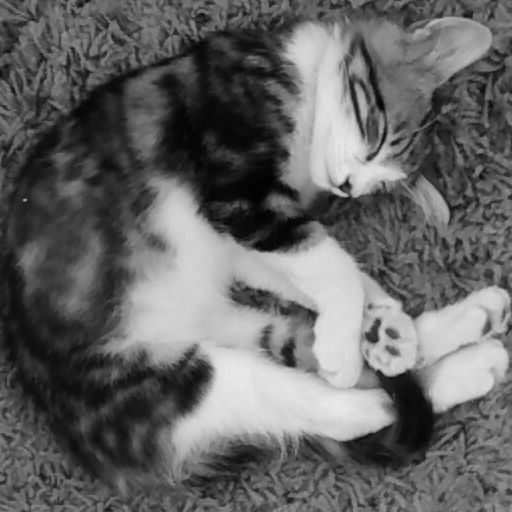}
&
\includegraphics[width=\widthimagethree]{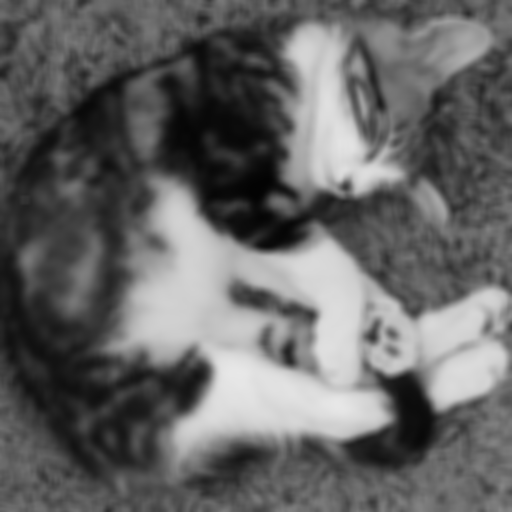}
&
\includegraphics[width=\widthimagethree]{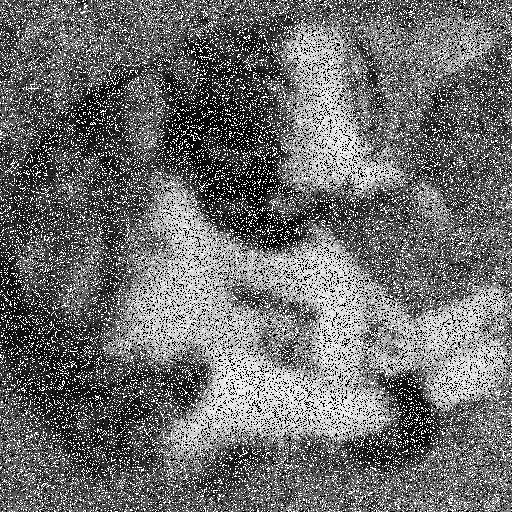}
\end{array}
\]
\caption{\small From left to right: original image, blurred image without noise, blurred image with noise.
}
\label{F:simple example problem}
\end{figure}

\noindent We reconstruct the image by using an $L^1$ data-fit function $D(u,y)=\Vert u-y \Vert_1$, and a regularizer enforcing  sparsity in 
a wavelet dictionary: 
\[
(\forall x\in X)\quad R(x) = \Vert Wx \Vert_1 + \frac{1}{2} \Vert x \Vert^2,
\]
where here $W$ is a Daubechies wavelet transform.
We run the (3-D) algorithm for $(\lambda_{max},\lambda_{min})=(10,10^{-2})$, and take $N_v=1000$, and $N_{wr}=30$, $\eps_{wr}=10^{-5}$.
In Figure \ref{F:simple example iterates}, some iterations of these two algorithms are displayed.

\begin{figure*}[h]
\centering
\[
\begin{array}{ccccc}
\includegraphics[width=\widthimagefive]{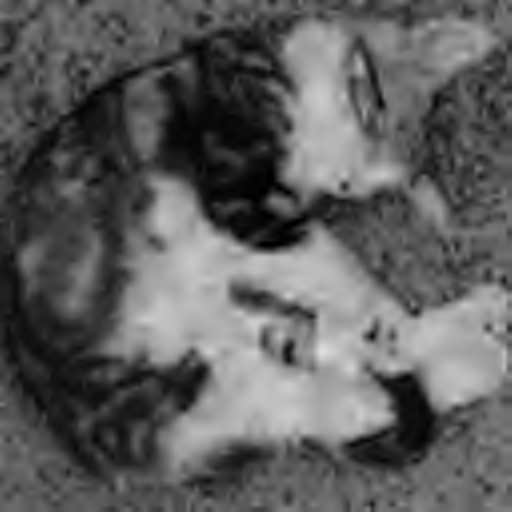}
&
\includegraphics[width=\widthimagefive]{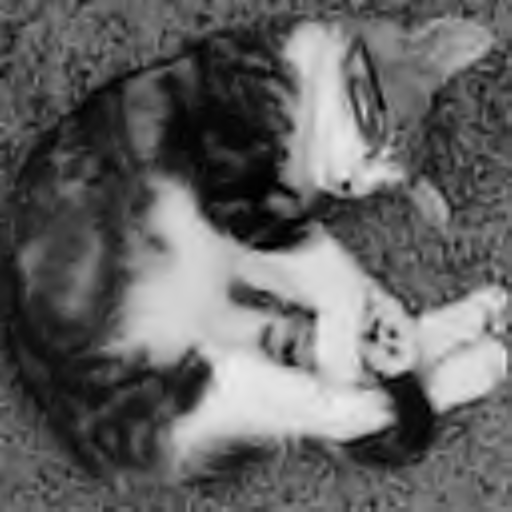}
&
\includegraphics[width=\widthimagefive]{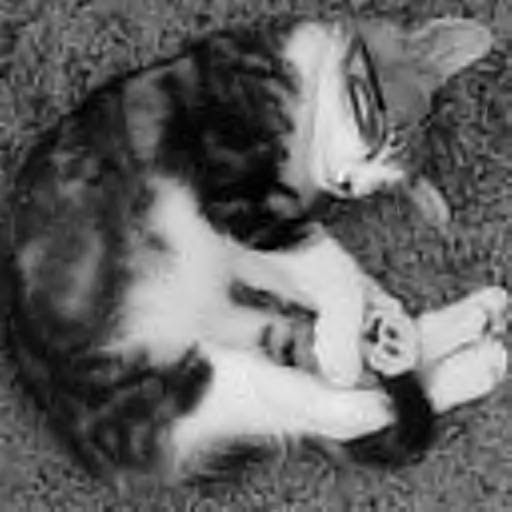}
&
\includegraphics[width=\widthimagefive]{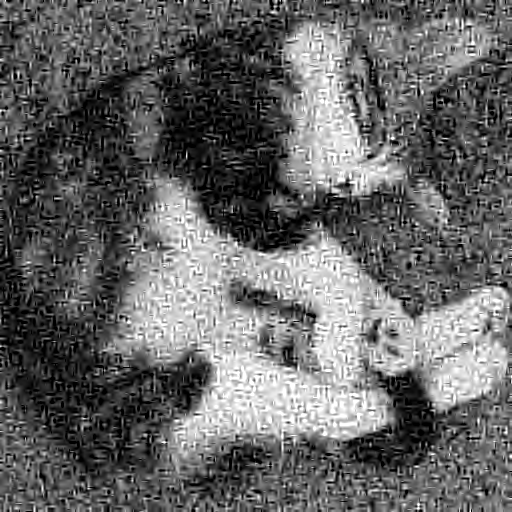}
&
\includegraphics[width=\widthimagefive]{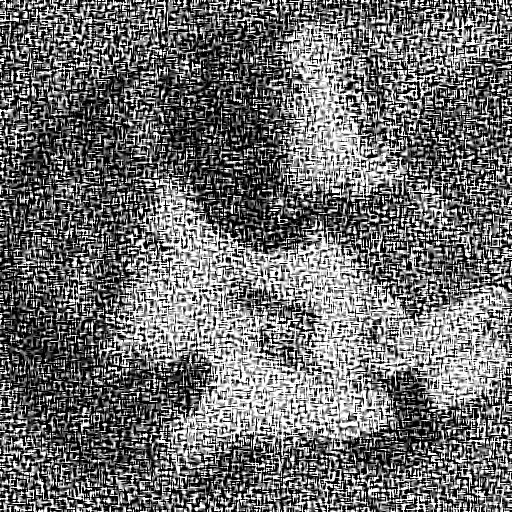} \\
\includegraphics[width=\widthimagefive]{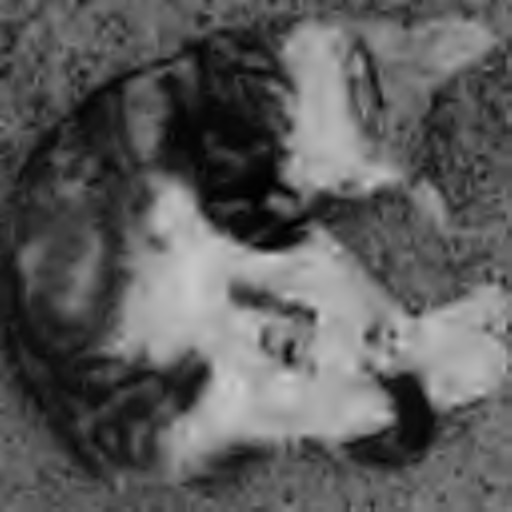}
&
\includegraphics[width=\widthimagefive]{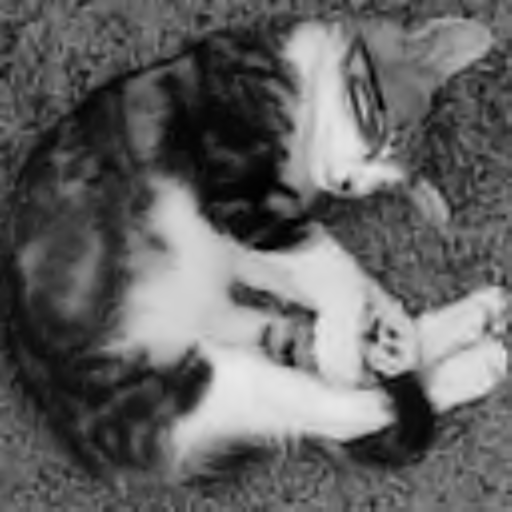}
&
\includegraphics[width=\widthimagefive]{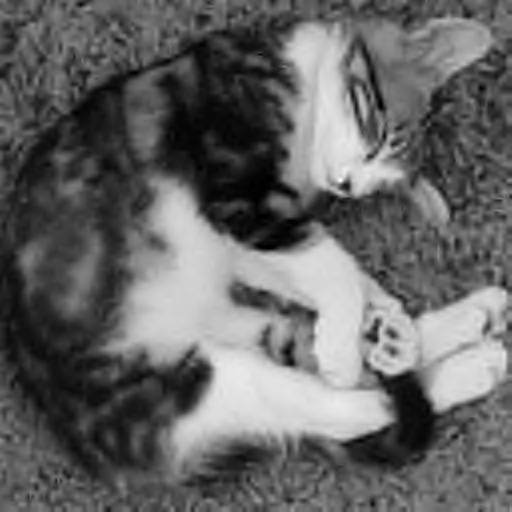}
&
\includegraphics[width=\widthimagefive]{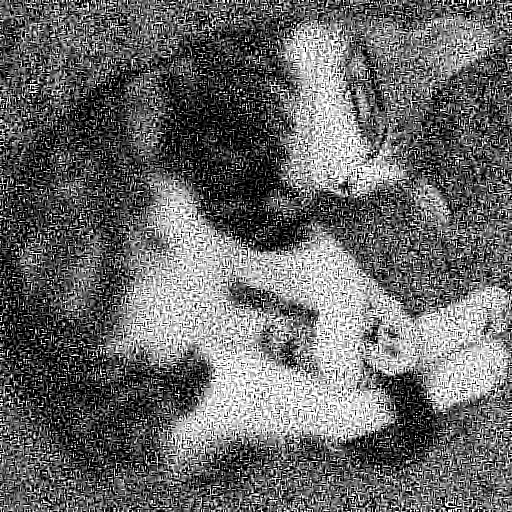}
&
\includegraphics[width=\widthimagefive]{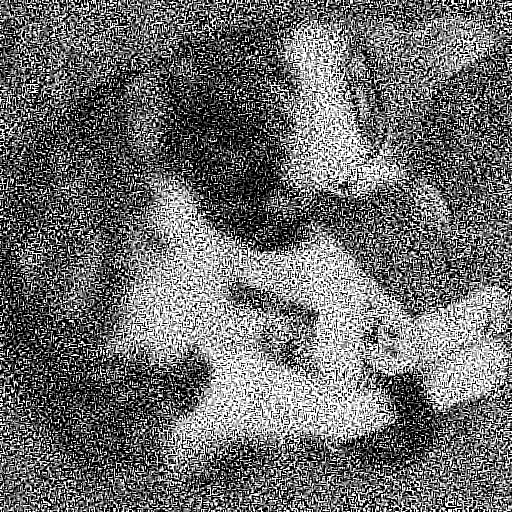} 
\end{array}
\]
\caption{\small First raw: Some iterates of vanilla 3D, with $(n;\lambda_n)=\{(350;0{,}89),(400;0{,}63),(500;0{,}32),(750;0{,}06),(900;0{,}02)\}$.
Second raw: Some iterates of warm 3D, with $(n;\lambda_n)=\{ (377;0{,}92),(535;0{,}57),(704;0{,}28),(2045;0{,}02),(2843;0{,}01)\}$.
}
\label{F:simple example iterates}
\end{figure*}

It can be seen that in the first iterations, the iterates go from an over-smoothed image towards an approximation of the original image, and then becomes contaminated by the noise.
This confirms our theoretical findings by showing that, in presence of noise, the number of iterations plays the role of a regularization parameter.
An \textit{early stopping} of the iterations leads then to better reconstruction results than the limit point.
By a simple visual inspection in Figure \ref{F:simple example iterates}, we would decide to stop the algorithms at the iterates corresponding to the middle column.
This transition between over-smoothing and noise contamination can also be measured, if one has access to the true image $\bar x$.
We can then measure what will be thereafter called the {Ground Truth Gap}: $GTG(x)=(1/d)\Vert x - \bar x\Vert,$
where $d$ is the number of pixels in $\bar x$ (see e.g. Figure \ref{F:simple example error gtg}).

\begin{center}

      \begin{minipage}{0.48\linewidth}
          \begin{figure}[H]
              \includegraphics[width=\linewidth]{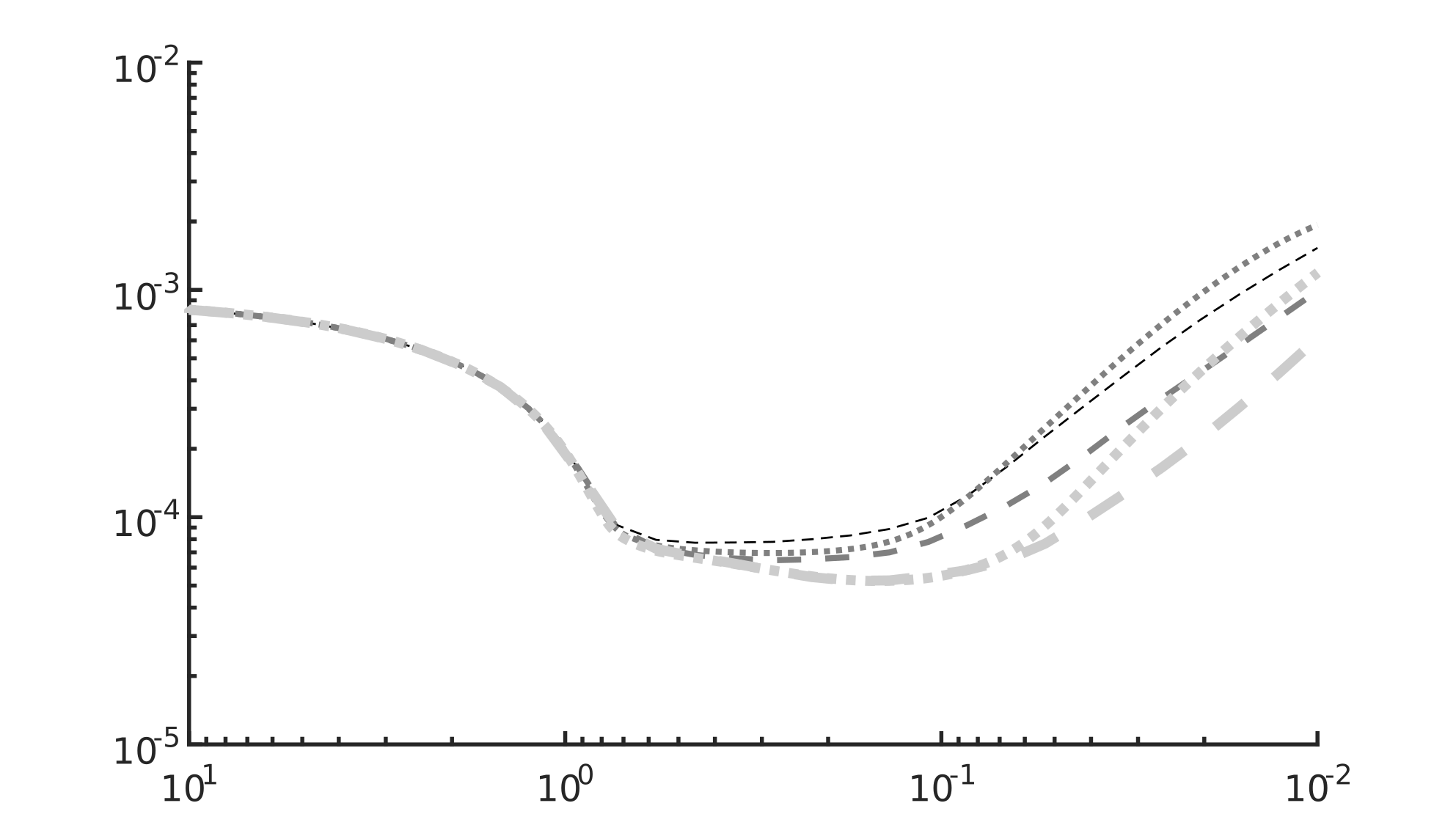}
              \caption{\small Plot of $GTG(x_n)$ with respect to $\lambda_n$, for various parameters. Dashed lines: warm 3D with $N_{wr}=30$ (from thin dark gray to thick light gray: $\eps_{wr}=\{10^{-4},10^{-5},10^{-6}\}$). Dotted lines: vanilla 3D (from thin dark gray to thick light gray: $N_v = \{10^3,10^4\}$).
}
			 \label{F:simple example error gtg}
          \end{figure}
      \end{minipage}
      \hspace{0.01\linewidth}
      \begin{minipage}{0.48\linewidth}
          \begin{figure}[H]
              \includegraphics[width=\linewidth]{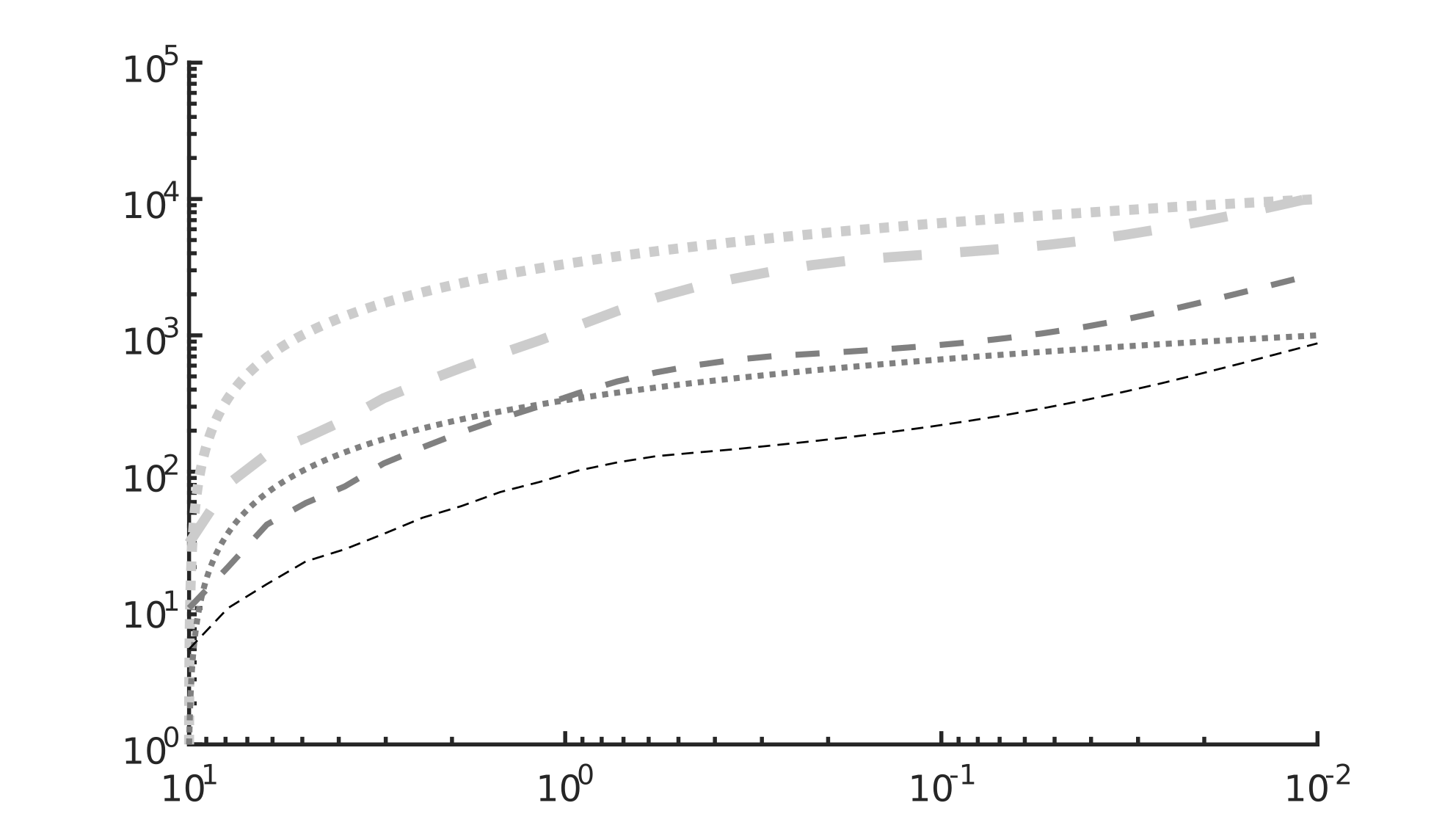}
              \caption{\small Cumulated number of iterations with respect to $\lambda_n$, for different parameters. Dashed lines: warm 3D with $N_{wr}=30$ (from thin dark gray to thick light gray: $\eps_{wr}=\{10^{-4},10^{-5},10^{-6}\}$). Dotted lines: vanilla 3D (from thin dark gray to thick light gray: $N_v = \{10^3,10^4\}$).
}
			\label{F:simple example nb iteration}
          \end{figure}
      \end{minipage}

\end{center}

We first observe that warm 3D and vanilla 3D  provide comparable reconstructions, but have a different complexity.
For vanilla 3D, the number $N_v$ controls directly the complexity and the accuracy of the method: the larger it is, the slower is the 
decay of $\lambda_n$, and the more parameters $\lambda$ are ``visited'' by the algorithm, improving the quality of the reconstructed image. 
For warm 3D, $N_{wr}$ plays a similar role, but here also the stopping rule parameter $\eps_{wr}$ has a strong indirect impact: the smaller it is, the 
slower is the decay of $\lambda_n$ because  more time is spent by the algorithm on each $\lambda$.
Also, the fact that problems $(P_\lambda)$ with a small $\lambda$ are harder to solve, heavily influence the number of iterations.
This trade-off between iteration complexity and reconstruction accuracy is illustrated in Figures \ref{F:simple example error gtg} and  \ref{F:simple example nb iteration}.
We observe  in the plots  the behavior predicted by Theorem \ref{T:stability 3-D}: the slower is the decay of $\lambda_n$, the better can be the solution, but also the larger is the number of iterations needed to reach this reconstruction.
To some extent, the parameters $N_v$ and $(N_{wr},\eps_{wr})$ play an analogue role to  $\beta$ in Theorem \ref{T:stability 3-D}.
One can also see that vanilla 3D and warm 3D can behave similarly: for the parameters $N_v=10^4$ and $(N_{wr},\eps_{wr})=(30,10^{-6})$, both methods reach a similar 
minimum value for the GTG, while requiring the same total amount of iterations.

We  note that  these two approaches outperform, in terms of computational time, the ``classic" Tikhonov regularization method.
To illustrate this, we compare in Figure \ref{F:simple example nb iteration classic} the complexity of classic Tikhonov and warm 3D methods, while the parameter $\eps_{wr}$ is fixed.
While achieving the same accuracy, classic Tikhonov requires between $2$ to $4$ times more iterations.

\begin{center}
\begin{minipage}{0.5\linewidth}
\begin{figure}[H]
\includegraphics[width=\linewidth]{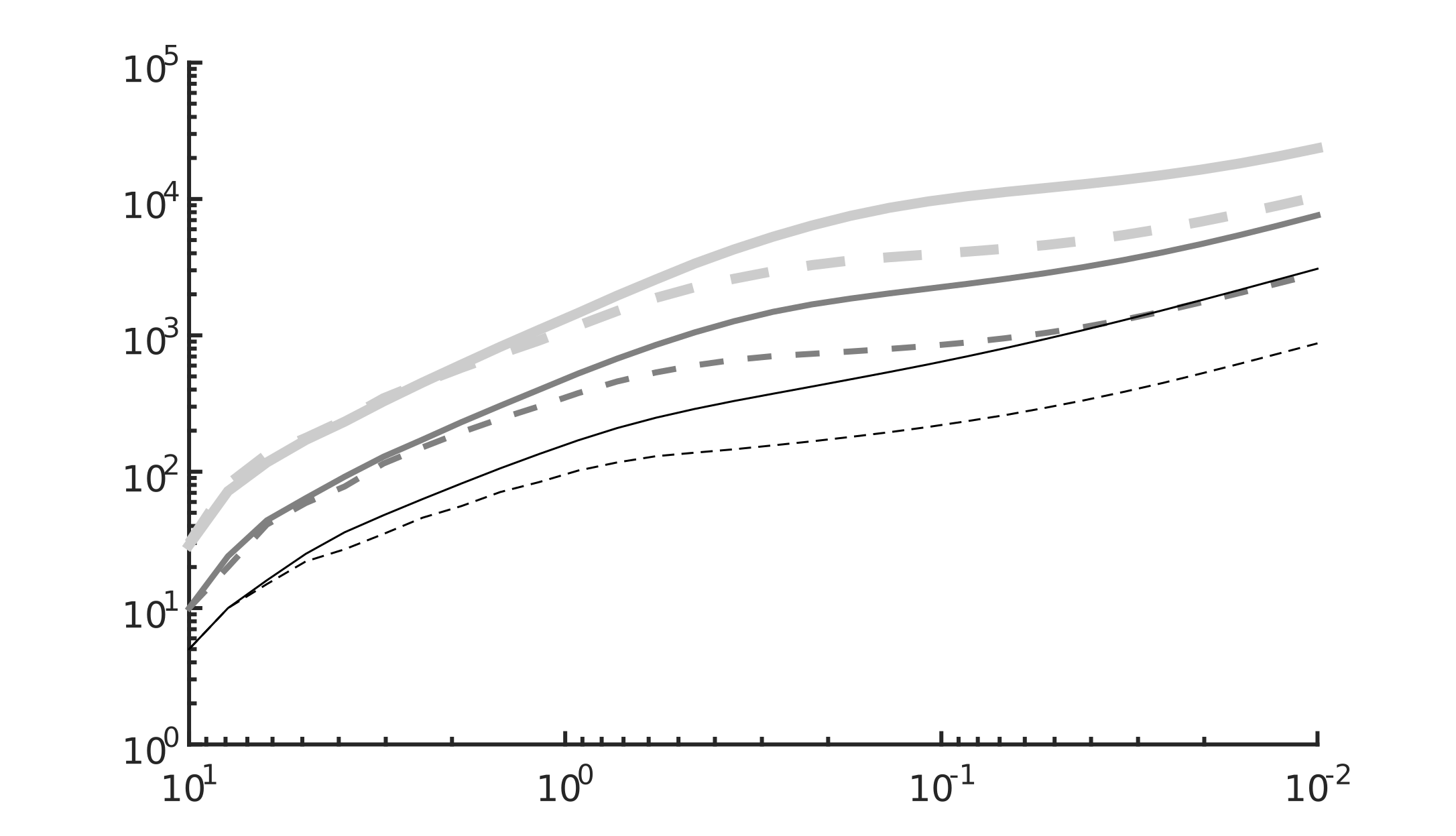}
\caption{\small Cumulated number of iterations with respect to $\lambda_n$, for different parameters.
Dashed lines: warm 3D with $N_{wr}=30$ (from thin dark gray to thick light gray: $\eps_{wr}=\{10^{-4},10^{-5},10^{-6}\}$).
Solid lines: classic Tikhonov (from thin dark gray to thick light gray: $\eps_{wr}=\{10^{-4},10^{-5},10^{-6}\}$).
}
\label{F:simple example nb iteration classic}
\end{figure}
\end{minipage}
\end{center}

\end{example}

\subsection{Parameter selection}
\label{SS:expreiments SURE discussion}

In this section, we discuss the problem of the regularization parameter selection.
We note again that iterative regularization provide a different way to explore different regularization level, and not a way to choose the right level. 
For the (3-D) method, the number of iterations $n$ is the regularization parameter, as shown in Proposition~\ref{T:stability 3-D}.
The problem of choosing the right regularization level-- i.e. the  right regularization parameter-- is of paramount important and still one of the biggest challenges in inverse problems.
For illustration purposes, in previous numerical experiments, we used the original image $\bar x$ to find the iterate $\bar n$ for which the ground truth gap $GTG(x_n)$ was minimized. This parameter's choice is clearly unrealistic in practical situations, where we only have access to a noisy data $\hat y$. Many automatic parameters choice are known, e.g. Morozov's discrepancy principle \cite{EngHanNeu96}, or SURE \cite{Ste81}). 
Next,  we comment on how they can be adapted to (3-D) and  in particular, consider the SURE parameter selection method, that we briefly present  below.

\medskip

Remember from \eqref{E:DiagonalAlgoGeneralForm} that our algorithm can be written as $ x_{n+1}:=\text{Algorithm}(x_{n};\lambda_n;\hat y)$. 
By recurrence, we can then express each iterate $x_n$ as a function of the starting point and the data
\[
x_n = \text{Algorithm}( .. \ \text{Algorithm}(x_0;\lambda_0;\hat y);.. \  ; \lambda_n;\hat y),
\]
or, more compactly, $x_n = \mathcal{A}_n(x_0;\hat y)$.
The SURE is an unbiased estimator for the \textit{Mean Squared Error}, 
$$MSE(x)=(1/d)\Vert A(x_n - \bar x) \Vert^2,$$ provided we have access to the noisy data $\hat y$ and the variance of the noise $\sigma^2$.
This estimator is defined by:
\begin{equation*}
SURE(x_n):=\frac{1}{d}\Vert Ax_n - \hat y \Vert^2 
+ \frac{2\sigma^2}{d} \langle A D_n, \xi \rangle,
\end{equation*}
where $\xi \sim \mathcal{N}(0,Id_{\R^d})$ and $D_n = \partial_{\hat y} \mathcal{A}_n(x_0; \cdot ) [\xi]$ is the weak directional derivative of $\mathcal{A}_n(x_0;\cdot)$ at $\hat y$ in the direction $\xi$.
For more details on this method, and how to compute it in practice, the reader might consult \cite[Section 4]{DelVaiFad14}).

The SURE estimator is depicted in Figure \ref{F:simple example sugar}, in the setting of Example \ref{Ex:simple example}.
One can see, and this behavior was
observed in all experiments, that the curve of $SURE(x_n)$ oscillates and is not convex.
This can be problematic when looking for a global minimum: the oscillations, together with the fact that the global minimum of $SURE$ often presents a sharp shape, do not allow to find a robust minimizer.
To circumvent these artifacts, we applied the following heuristic, which proved to be efficient in our experiments: smoothing the curve of $SURE(x_n)$, and defining the early stopping iterate $\hat n$ as the one minimizing the slope of this smoothed version of $SURE(x_n)$.
\begin{center}
\begin{minipage}{0.75\linewidth}
\begin{figure}[H]
\includegraphics[width=\linewidth]{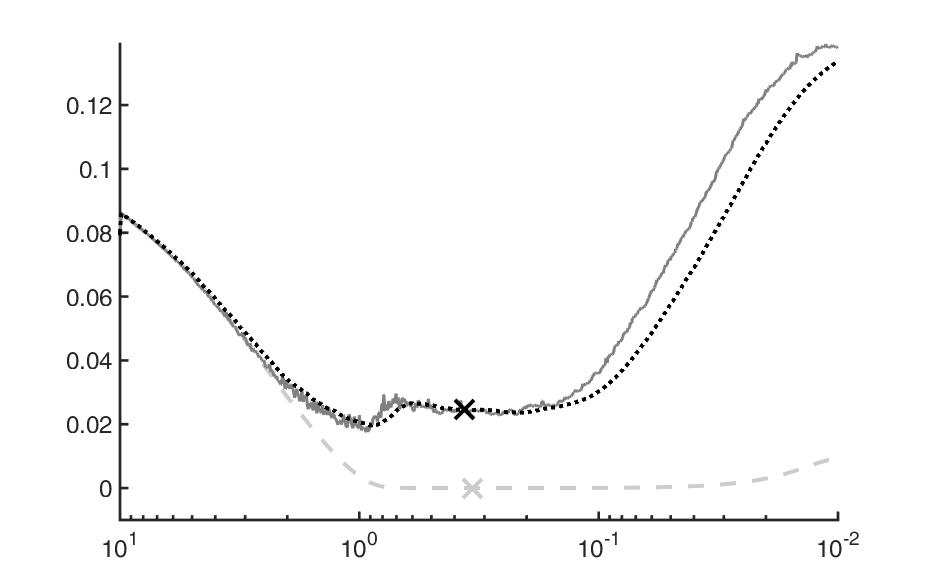}
\caption{\small 
Plot of various estimators with respect to $\lambda_n$.
Light gray dashed line: true MSE.
Light gray cross: minimum of the true MSE.
Gray plain line: SURE.
Black dotted line: smoothed SURE.
Black cross: minimum slope of the smoothed SURE.
For display purposes, SURE and its smoothed version are here corrected by an additive constant.
}
\label{F:simple example sugar}
\end{figure}
\end{minipage}
\end{center}

Note that the statistical properties of the SURE estimator rely on the assumption that the noise is Gaussian \cite{Ste81,DelVaiFad14}.
Nevertheless, as we will see below, it also provides surprisingly good results for the impulse noise, while being less efficient for the Poisson noise, or the mixed Gaussian-impulse noise.

\vspace{-0.6cm}

\subsection{Experiments for various noises and models}\label{SS:experiments dataset with true image}

In this section, we  run and compare vanilla 3D and warm 3D on a data-set, considering different noises and models for recovering the images.
This data-set, which is available online\footnote{\url{www.guillaume-garrigos.com/database/image_processing_512.zip}}, is made of 23 images, whose size range from $500 \times 375$ to $515 \times 512$ pixels.
For each experiment, the range of parameters $(\lambda_{max},\lambda_{min})$ will be chosen accordingly to the nature of the noise and its variance.
To fairly compare vanilla 3D and warm 3D, we will choose for each example  the parameters $(N_v,N_{wr},\eps_{wr})$ in such a way that the  number of iterations for both methods is of the same order ($\sim 10^3$).
For each experiment, the early stopping will be defined according to two different rules: keeping the notations of Section \ref{SS:expreiments SURE discussion}, $\bar n$ will denote the iterate minimizing the ideal ground truth gap $GTG(x_n)$, while $\hat n$ will be the iterate defined by means of the SURE estimator.

\begin{example}
\label{Ex:saltandpepper L1L1}
Each image of the data-set is blurred and corrupted by a salt and pepper noise of intensity $35\%$, and is reconstructed by using an $L^1$ data-fit function $D(u,y)=\Vert u-y \Vert_1$, and a regularizer enforcing  sparsity in a wavelet dictionary: 
\[
(\forall x\in X)\quad R(x) = \Vert Wx \Vert_1 + \frac{1}{2} \Vert x \Vert^2,
\]
where here $W$ is a Daubechies wavelet transform.
We run the (3-D) algorithm for $(\lambda_{max},\lambda_{min})=(10,10^{-1})$, and take $N_v=1000$ and $N_{wr}=20$, $\eps_{wr}=10^{-5}$ for vanilla 3D and warm 3D, respectively.
The results are summarized in Table~\ref{F:SPL1L1 numbers} and Figure~\ref{F:SPL1L1 pictures}.

\begin{center}
\begin{minipage}{0.6\linewidth}
\begin{table}[H]
\begin{tabular}{|c|c|c|}
\hline 
 & vanilla 3D & warm 3D \\ 
\hline 
Iterations & $1000$  & $886 \pm 160$ \\ 
\hline 
$GTG(x_{\bar n})$  & $1{,}14. 10^{-4} \pm 4{,}5 . 10^{-5}$ & $1{,}11. 10^{-4} \pm 4{,}3 . 10^{-5}$ \\ 
\hline 
$GTG(x_{\hat n})$  & $1{,}37. 10^{-4} \pm 5{,}8 . 10^{-5}$ & $1{,}40. 10^{-4} \pm 6{,}0 . 10^{-5}$ \\ 
\hline 
\end{tabular} 
\caption{\small Results of the experiments for Example \ref{Ex:saltandpepper L1L1}.
}
\label{F:SPL1L1 numbers}
\end{table}
\end{minipage}
\end{center}

\begin{center}
\begin{minipage}{0.6\linewidth}
\begin{figure}[H]
\centering
\[
\begin{array}{ccccc}
\includegraphics[width=\widthimagefive]{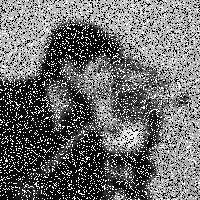}
&
\includegraphics[width=\widthimagefive]{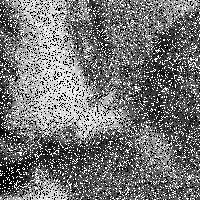}
&
\includegraphics[width=\widthimagefive]{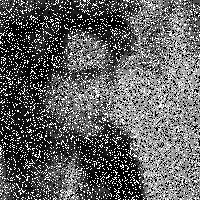}
&
\includegraphics[width=\widthimagefive]{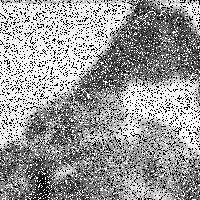}
&
\includegraphics[width=\widthimagefive]{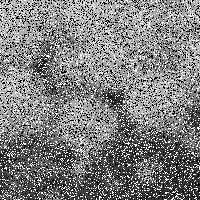} 
 \\
\includegraphics[width=\widthimagefive]{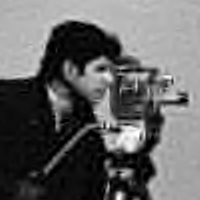}
&
\includegraphics[width=\widthimagefive]{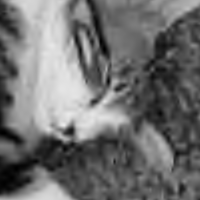}
&
\includegraphics[width=\widthimagefive]{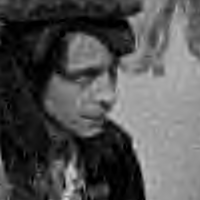}
&
\includegraphics[width=\widthimagefive]{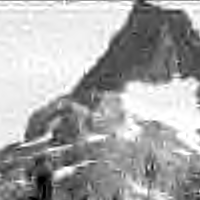}
&
\includegraphics[width=\widthimagefive]{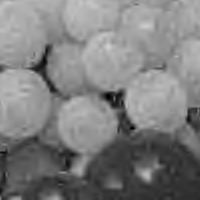} 
\\
\includegraphics[width=\widthimagefive]{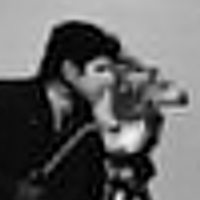}
&
\includegraphics[width=\widthimagefive]{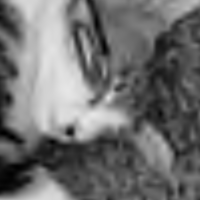}
&
\includegraphics[width=\widthimagefive]{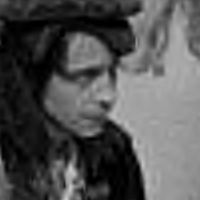}
&
\includegraphics[width=\widthimagefive]{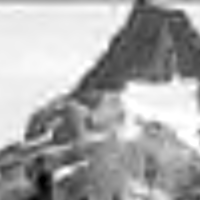}
&
\includegraphics[width=\widthimagefive]{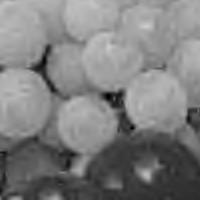} 
\\
\includegraphics[width=\widthimagefive]{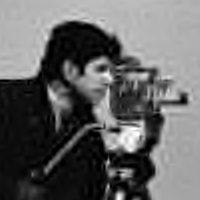}
&
\includegraphics[width=\widthimagefive]{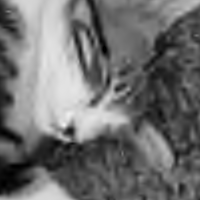}
&
\includegraphics[width=\widthimagefive]{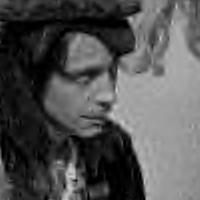}
&
\includegraphics[width=\widthimagefive]{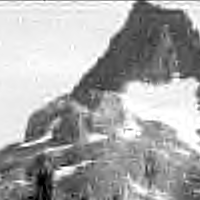}
&
\includegraphics[width=\widthimagefive]{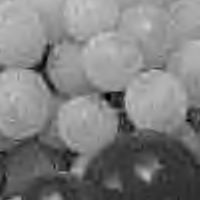} 
 \\
\includegraphics[width=\widthimagefive]{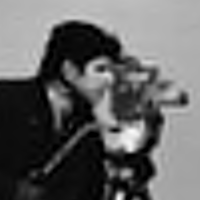}
&
\includegraphics[width=\widthimagefive]{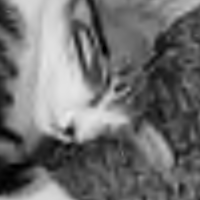}
&
\includegraphics[width=\widthimagefive]{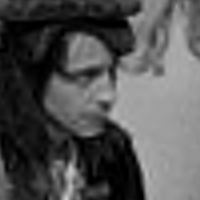}
&
\includegraphics[width=\widthimagefive]{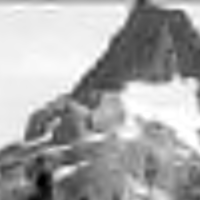}
&
\includegraphics[width=\widthimagefive]{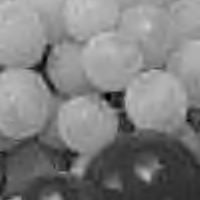} 
\end{array} 
\]
\caption{\small Samples from Example \ref{Ex:saltandpepper L1L1}.
From top to bottom: noisy image, reconstruction with vanilla 3D having access to the GTG (i.e.  $x_{\bar n}$), reconstruction with vanilla 3D using SURE (i.e.  $x_{\hat n}$), and reconstruction with warm 3D ($x_{\bar n}$ then $x_{\hat n}$).
}
\label{F:SPL1L1 pictures}
\end{figure}
\end{minipage}
\begin{minipage}{0.1\linewidth}
\hfill
\end{minipage}
\end{center}
\end{example}

\begin{example}
\label{Ex:SPL1TV}
Each image of the data-set is blurred and corrupted by a salt and pepper noise of intensity $35\%$, and is reconstructed by using an $L^1$ data-fit function $D(u,y)=\Vert u-y \Vert_1$, and a regularizer based on the total variation: 
\[
(\forall x\in X)\quad R(x) = \frac{1}{10}\Vert x \Vert_{TV} + \frac{1}{2} \Vert x \Vert^2.
\]
We run the (3-D) algorithm by taking $(\lambda_{max},\lambda_{min})=(10,10^{-1})$, with $N_v=1000$ and $N_{wr}=20$, $\eps_{wr}=10^{-5}$ for vanilla 3D and warm 3D, respectively.
The results are summarized in
Table \ref{F:SPL1TV numbers} and Figure \ref{F:SPL1TV pictures}.

\begin{center}
\begin{minipage}{0.6\linewidth}
\begin{table}[H]
\begin{tabular}{|c|c|c|}
\hline 
 & vanilla 3D & warm 3D \\ 
\hline 
Iterations & $1000$  & $618 \pm 106$ \\ 
\hline 
$GTG(x_{\bar n})$  & $1{,}02. 10^{-4} \pm 4{,}3 . 10^{-5}$ & $1{,}04. 10^{-4} \pm 4{,}0 . 10^{-5}$ \\ 
\hline 
$GTG(x_{\hat n})$  & $1{,}05. 10^{-4} \pm 4{,}4 . 10^{-5}$ & $1{,}13. 10^{-4} \pm 4{,}1 . 10^{-5}$ \\ 
\hline 
\end{tabular} 
\caption{Results of the experiments for Example \ref{Ex:SPL1TV}.
}
\label{F:SPL1TV numbers}
\end{table}
\end{minipage}
\end{center}

\begin{center}
\begin{minipage}{0.6\linewidth}
\begin{figure}[H]
\centering
\[
\begin{array}{ccccc}
\includegraphics[width=\widthimagefive]{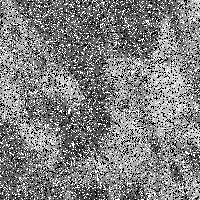}
&
\includegraphics[width=\widthimagefive]{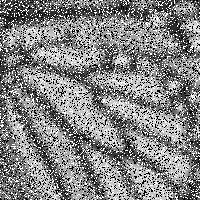}
&
\includegraphics[width=\widthimagefive]{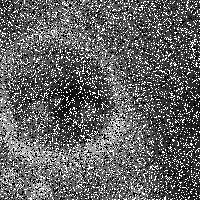}
&
\includegraphics[width=\widthimagefive]{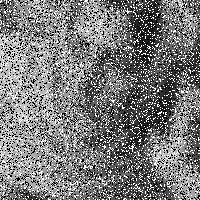} 
&
\includegraphics[width=\widthimagefive]{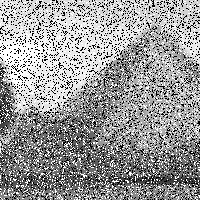} \\
\includegraphics[width=\widthimagefive]{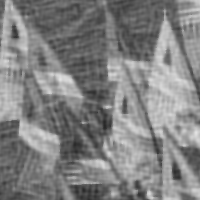}
&
\includegraphics[width=\widthimagefive]{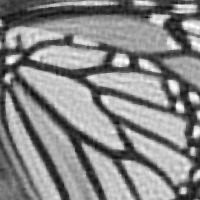}
&
\includegraphics[width=\widthimagefive]{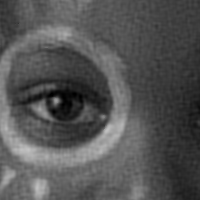}
&
\includegraphics[width=\widthimagefive]{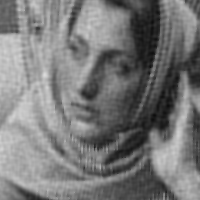} 
&
\includegraphics[width=\widthimagefive]{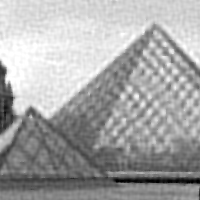}
\\
\includegraphics[width=\widthimagefive]{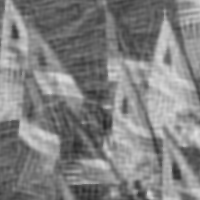}
&
\includegraphics[width=\widthimagefive]{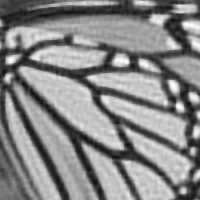}
&
\includegraphics[width=\widthimagefive]{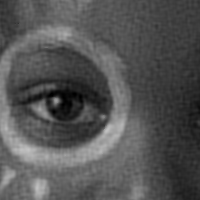}
&
\includegraphics[width=\widthimagefive]{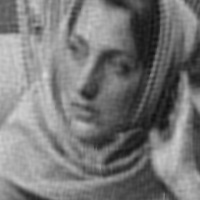} 
&
\includegraphics[width=\widthimagefive]{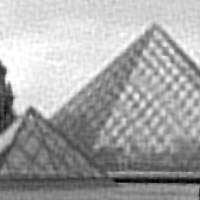}
\\
\includegraphics[width=\widthimagefive]{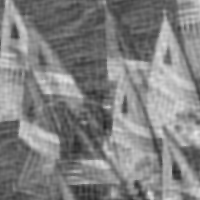}
&
\includegraphics[width=\widthimagefive]{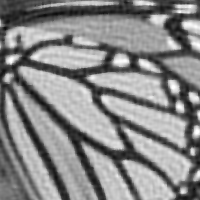}
&
\includegraphics[width=\widthimagefive]{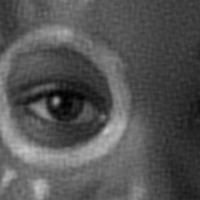}
&
\includegraphics[width=\widthimagefive]{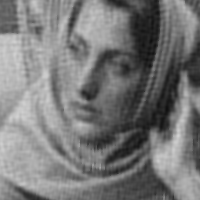} 
&
\includegraphics[width=\widthimagefive]{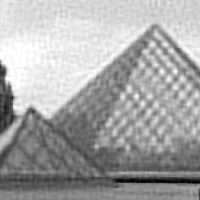}
\\
\includegraphics[width=\widthimagefive]{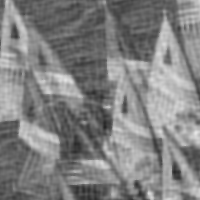}
&
\includegraphics[width=\widthimagefive]{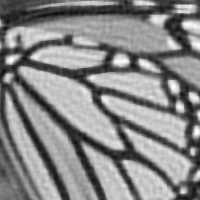}
&
\includegraphics[width=\widthimagefive]{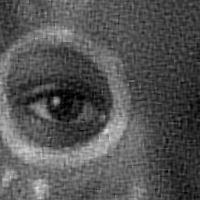}
&
\includegraphics[width=\widthimagefive]{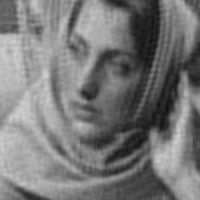} 
&
\includegraphics[width=\widthimagefive]{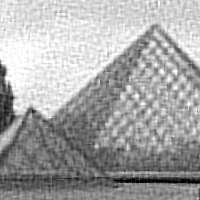}
\end{array} 
\]
\caption{Samples from Example \ref{Ex:SPL1TV}.
From top to bottom: noisy image, reconstruction with vanilla 3D having access to the GTG (i.e.  $x_{\bar n}$), reconstruction with vanilla 3D using SURE (i.e.  $x_{\hat n}$), and reconstruction with warm 3D ($x_{\bar n}$ then $x_{\hat n}$).
}
\label{F:SPL1TV pictures}
\end{figure}
\end{minipage}
\begin{minipage}{0.1\linewidth}
\hfill
\end{minipage}
\end{center}
\end{example}

\begin{example}
\label{Ex:GL2TV}
Each image of the data-set is blurred and corrupted by a Gaussian noise of {variance} $10^{-2}$, and is reconstructed by using an $L^2$ data-fit function $D(u,y)=(1/2)\Vert u-y \Vert^2$, and a regularizer based on the total variation: 
\[
(\forall x\in X)\quad R(x) = \Vert x \Vert_{TV} + \frac{1}{2} \Vert x \Vert^2.
\]
We run the (3-D) algorithm by taking $(\lambda_{max},\lambda_{min})=(1,10^{-2})$, with $N_v=1000$ and $N_{wr}=20$, $\eps_{wr}=10^{-4}$ for vanilla 3D and warm 3D, respectively.
The results are summarized in Table~\ref{F:GL2TV numbers} and Figure~\ref{F:GL2TV pictures}.

\begin{center}
\begin{minipage}{0.6\linewidth}
\begin{table}[H]
\begin{tabular}{|c|c|c|}
\hline 
 & vanilla 3D & warm 3D \\ 
\hline 
Iterations & $1000$  & $1096 \pm 50$ \\ 
\hline 
$GTG(x_{\bar n})$  & $1{,}41  . 10^{-4} \pm 4{,}4 . 10^{-5}$ & $1{,}42. 10^{-4} \pm 4{,}4 . 10^{-5}$ \\ 
\hline 
$GTG(x_{\hat n})$  & $1{,}48. 10^{-4} \pm 4{,}1 . 10^{-5}$ & $1{,}56. 10^{-4} \pm 3{,}9 . 10^{-5}$ \\ 
\hline 
\end{tabular} 
\caption{Results of the experiments for Example \ref{Ex:GL2TV}.
}
\label{F:GL2TV numbers}
\end{table}
\end{minipage}
\end{center}

\begin{center}
\begin{minipage}{0.6\linewidth}
\begin{figure}[H]
\centering
\[
\begin{array}{ccccc}
\includegraphics[width=\widthimagefive]{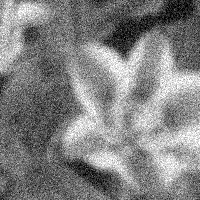}
&
\includegraphics[width=\widthimagefive]{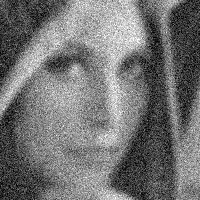}
&
\includegraphics[width=\widthimagefive]{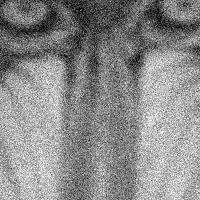}
&
\includegraphics[width=\widthimagefive]{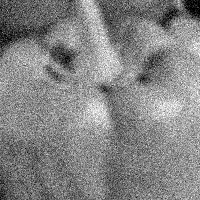}
&
\includegraphics[width=\widthimagefive]{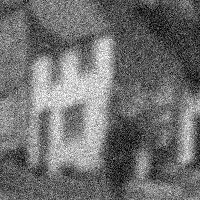}  
\\
\includegraphics[width=\widthimagefive]{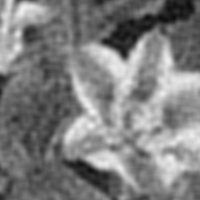}
&
\includegraphics[width=\widthimagefive]{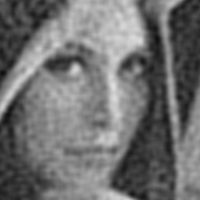}
&
\includegraphics[width=\widthimagefive]{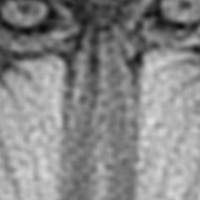}
&
\includegraphics[width=\widthimagefive]{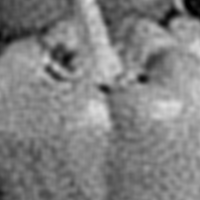}
&
\includegraphics[width=\widthimagefive]{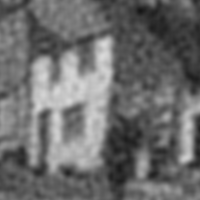}  
\\
\includegraphics[width=\widthimagefive]{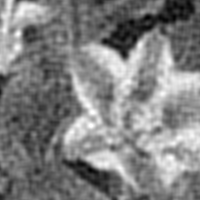}
&
\includegraphics[width=\widthimagefive]{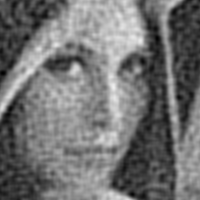}
&
\includegraphics[width=\widthimagefive]{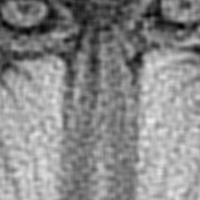}
&
\includegraphics[width=\widthimagefive]{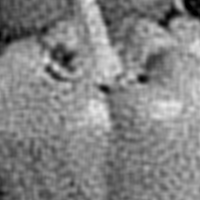}
&
\includegraphics[width=\widthimagefive]{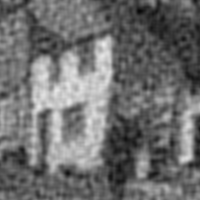}  
\\
\includegraphics[width=\widthimagefive]{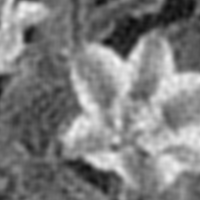}
&
\includegraphics[width=\widthimagefive]{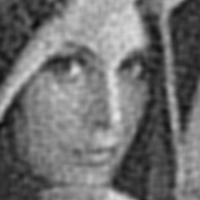}
&
\includegraphics[width=\widthimagefive]{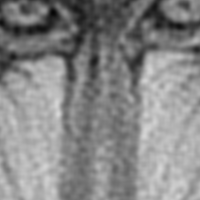}
&
\includegraphics[width=\widthimagefive]{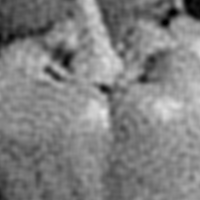}
&
\includegraphics[width=\widthimagefive]{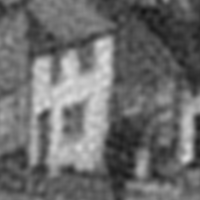} 
\\
\includegraphics[width=\widthimagefive]{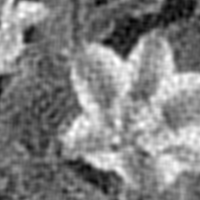}
&
\includegraphics[width=\widthimagefive]{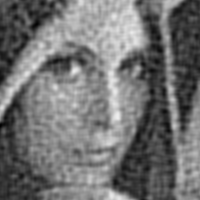}
&
\includegraphics[width=\widthimagefive]{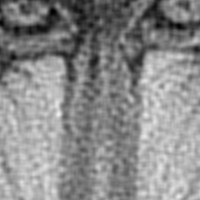}
&
\includegraphics[width=\widthimagefive]{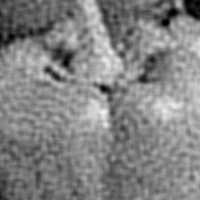}
&
\includegraphics[width=\widthimagefive]{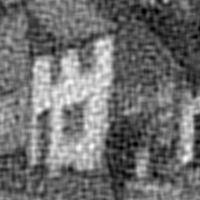} 
\end{array} 
\]
\caption{Samples from Example \ref{Ex:GL2TV}.
From top to bottom: noisy image, reconstruction with vanilla 3D having access to the GTG (i.e.  $x_{\bar n}$), reconstruction with vanilla 3D using SURE (i.e.  $x_{\hat n}$), and reconstruction with warm 3D ($x_{\bar n}$ then $x_{\hat n}$).
}
\label{F:GL2TV pictures}
\end{figure}
\end{minipage}
\begin{minipage}{0.1\linewidth}
\hfill
\end{minipage}
\end{center}
\end{example}

\begin{example}
\label{Ex:GSPL1L2TV}
Each image of the data-set is blurred and corrupted by a combination of a Gaussian noise of variance $5.10^{-3}$ and a salt and pepper noise of intensity $5\%$, and 
is reconstructed using an Huber data-fit function $D(u,y)=H_{\sigma}(x-y)$ with $\sigma=0.1$, and a regularizer based on the total variation: 
\[
(\forall x\in X)\quad R(x) = \Vert x \Vert_{TV} + \frac{1}{2} \Vert x \Vert^2.
\]
We run the (3-D) algorithm by taking $(\lambda_{max},\lambda_{min})=(10^{-1},10^{-3})$, with $N_v=1000$ and $N_{wr}=20$, $\eps_{wr}=10^{-4}$ for vanilla 3D and warm 3D, respectively.
The results are summarized in Table~\ref{F:GSPL1L2TV numbers} and Figure~\ref{F:GSPL1L2TV pictures}.

\begin{center}
\begin{minipage}{0.6\linewidth}
\begin{table}[H]
\begin{tabular}{|c|c|c|}
\hline 
 & vanilla 3D & warm 3D \\ 
\hline 
Iterations & $1000$  & $3760 \pm 131$ \\ 
\hline 
$GTG(x_{\bar n})$  & $1{,}56  . 10^{-4} \pm 4{,}3 . 10^{-5}$ & $1{,}58. 10^{-4} \pm 4{,}3 . 10^{-5}$ \\ 
\hline 
$GTG(x_{\hat n})$  & $2{,}16  . 10^{-4} \pm 6{,}8 . 10^{-5}$ & $1{,}99. 10^{-4} \pm 8{,}2 . 10^{-5}$ \\ 
\hline 
\end{tabular} 
\caption{Results of the experiments for Example \ref{Ex:GSPL1L2TV}.
}
\label{F:GSPL1L2TV numbers}
\end{table}
\end{minipage}
\end{center}

\begin{center}
\begin{minipage}{0.6\linewidth}
\begin{figure}[H]
\centering
\[
\begin{array}{ccccc}
\includegraphics[width=\widthimagefive]{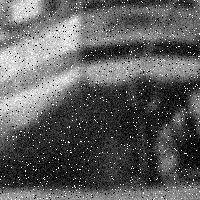}
&
\includegraphics[width=\widthimagefive]{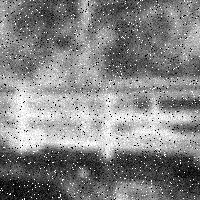}
&
\includegraphics[width=\widthimagefive]{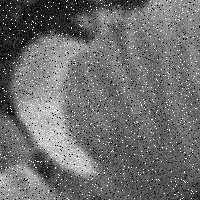}
&
\includegraphics[width=\widthimagefive]{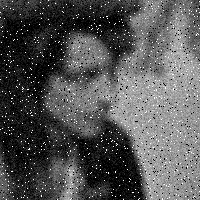} 
&
\includegraphics[width=\widthimagefive]{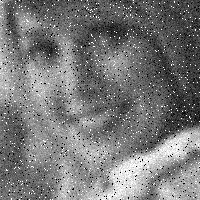} 
\\
\includegraphics[width=\widthimagefive]{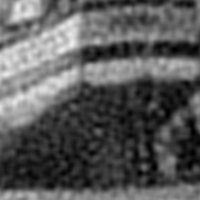}
&
\includegraphics[width=\widthimagefive]{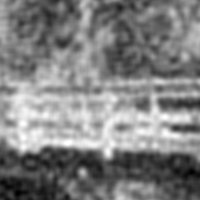}
&
\includegraphics[width=\widthimagefive]{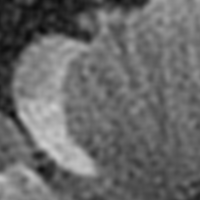}
&
\includegraphics[width=\widthimagefive]{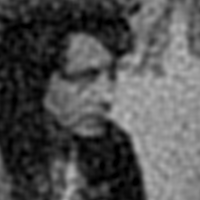} 
&
\includegraphics[width=\widthimagefive]{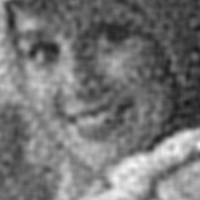}
\\
\includegraphics[width=\widthimagefive]{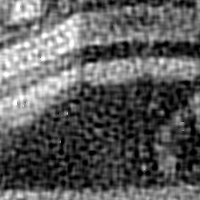}
&
\includegraphics[width=\widthimagefive]{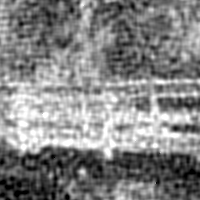}
&
\includegraphics[width=\widthimagefive]{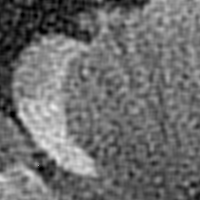}
&
\includegraphics[width=\widthimagefive]{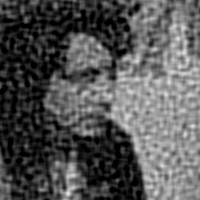} 
&
\includegraphics[width=\widthimagefive]{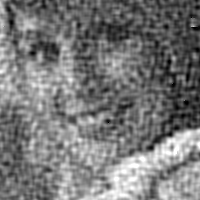}
\\
\includegraphics[width=\widthimagefive]{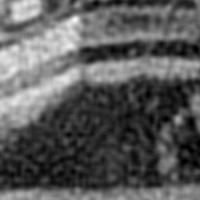}
&
\includegraphics[width=\widthimagefive]{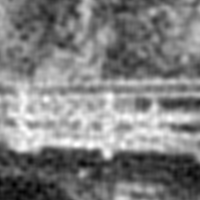}
&
\includegraphics[width=\widthimagefive]{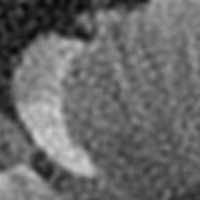}
&
\includegraphics[width=\widthimagefive]{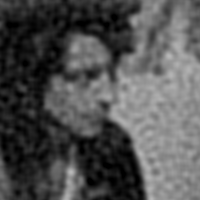} 
&
\includegraphics[width=\widthimagefive]{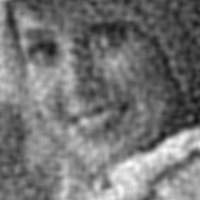}
\\
\includegraphics[width=\widthimagefive]{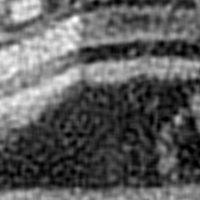}
&
\includegraphics[width=\widthimagefive]{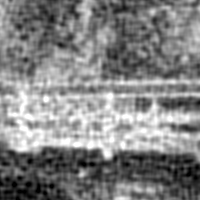}
&
\includegraphics[width=\widthimagefive]{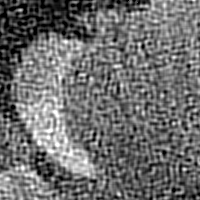}
&
\includegraphics[width=\widthimagefive]{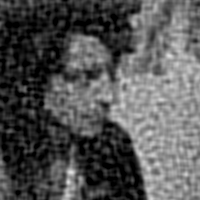} 
&
\includegraphics[width=\widthimagefive]{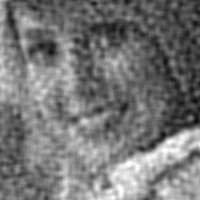}
\end{array} 
\]
\caption{Samples from example \ref{Ex:GSPL1L2TV}.
From top to bottom: noisy image, reconstruction with vanilla 3D having access to the GTG (i.e.  $x_{\bar n}$), reconstruction with vanilla 3D using SURE (i.e.  $x_{\hat n}$), and reconstruction with warm 3D ($x_{\bar n}$ then $x_{\hat n}$).
}
\label{F:GSPL1L2TV pictures}
\end{figure}
\end{minipage}
\begin{minipage}{0.1\linewidth}
\hfill
\end{minipage}
\end{center}
\end{example}

\begin{example}
\label{Ex:PKLTV}
Each image of the data-set is blurred and corrupted by a Poisson noise, and is reconstructed by using a Kullback-Leibler data-fit function $D(u,y)=\KL(y;u + b)$, where $b$ models a background noise of small intensity, and a regularizer based on the total variation: 
\[
(\forall x\in X)\quad R(x) = \frac{1}{10}\Vert x \Vert_{TV} + \frac{1}{2} \Vert x \Vert^2.
\]
We run the (3-D) algorithm by taking $(\lambda_{max},\lambda_{min})=(10^{-1},10^{-3})$, with $N_v=1000$ and $N_{wr}=20$, $\eps_{wr}=10^{-4}$ for vanilla 3D and warm 3D, respectively.
The results are summarized in Table \ref{F:PKLTV numbers} and Figure \ref{F:PKLTV pictures}.

\begin{center}
\begin{minipage}{0.6\linewidth}
\begin{table}[H]
\begin{tabular}{|c|c|c|}
\hline 
 & vanilla 3D & warm 3D \\ 
\hline 
Iterations & $1000$  & $3674 \pm 329$ \\ 
\hline 
$GTG(x_{\bar n})$  & $1{,}24  . 10^{-4} \pm 4{,}2 . 10^{-5}$ & $1{,}26. 10^{-4} \pm 4{,}2 . 10^{-5}$ \\ 
\hline 
$GTG(x_{\hat n})$  & $3{,}51  . 10^{-4} \pm 3{,}9 . 10^{-5}$ & $6{,}80. 10^{-4} \pm 1{,}93 . 10^{-4}$ \\ 
\hline 
\end{tabular} 
\caption{Results of the experiments for Example \ref{Ex:PKLTV}.
}
\label{F:PKLTV numbers}
\end{table}
\end{minipage}
\end{center}

\begin{center}
\begin{minipage}{0.6\linewidth}
\begin{figure}[H]
\centering
\[
\begin{array}{ccccc}
\includegraphics[width=\widthimagefive]{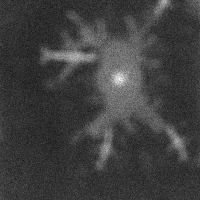}
&
\includegraphics[width=\widthimagefive]{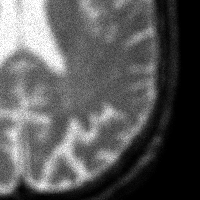}
&
\includegraphics[width=\widthimagefive]{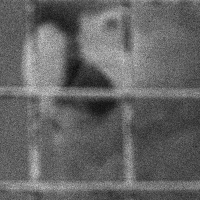}
&
\includegraphics[width=\widthimagefive]{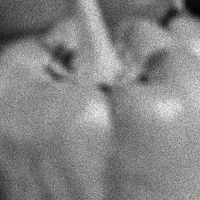} 
&
\includegraphics[width=\widthimagefive]{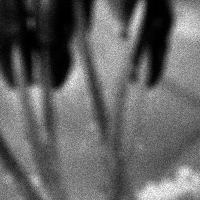} 
\\
\includegraphics[width=\widthimagefive]{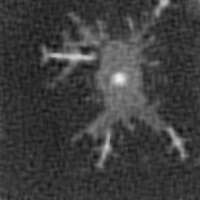}
&
\includegraphics[width=\widthimagefive]{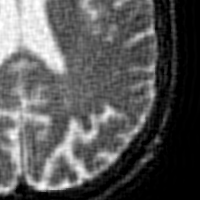}
&
\includegraphics[width=\widthimagefive]{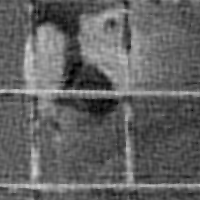}
&
\includegraphics[width=\widthimagefive]{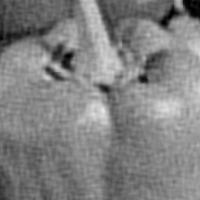} 
&
\includegraphics[width=\widthimagefive]{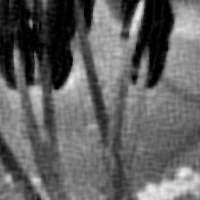}
\\
\includegraphics[width=\widthimagefive]{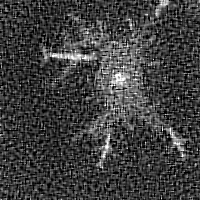}
&
\includegraphics[width=\widthimagefive]{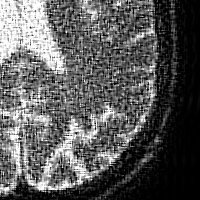}
&
\includegraphics[width=\widthimagefive]{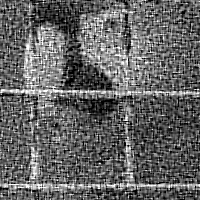}
&
\includegraphics[width=\widthimagefive]{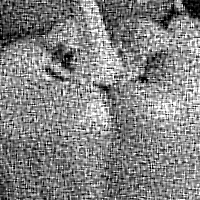} 
&
\includegraphics[width=\widthimagefive]{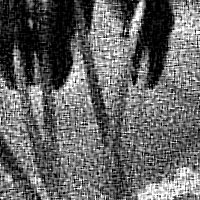}
\\
\includegraphics[width=\widthimagefive]{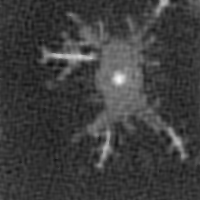}
&
\includegraphics[width=\widthimagefive]{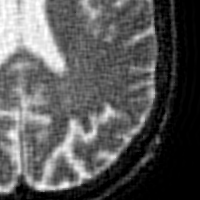}
&
\includegraphics[width=\widthimagefive]{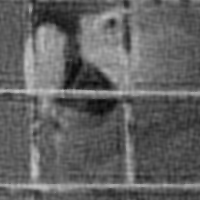}
&
\includegraphics[width=\widthimagefive]{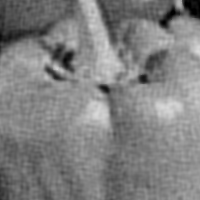} 
&
\includegraphics[width=\widthimagefive]{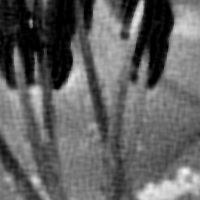}
\\
\includegraphics[width=\widthimagefive]{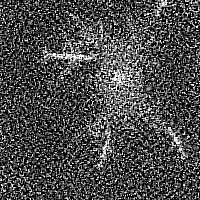}
&
\includegraphics[width=\widthimagefive]{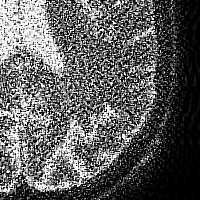}
&
\includegraphics[width=\widthimagefive]{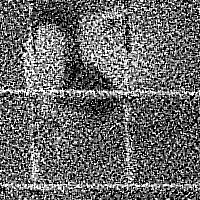}
&
\includegraphics[width=\widthimagefive]{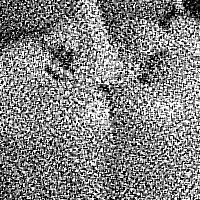} 
&
\includegraphics[width=\widthimagefive]{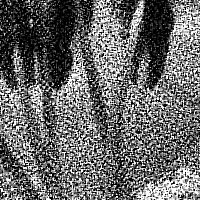}
\end{array} 
\]
\caption{Samples from Example \ref{Ex:PKLTV}.
From top to bottom: noisy image, reconstruction with vanilla 3D having access to the GTG (i.e.  $x_{\bar n}$), reconstruction with vanilla 3D using SURE (i.e.  $x_{\hat n}$), and reconstruction with warm 3D ($x_{\bar n}$ then $x_{\hat n}$).
}
\label{F:PKLTV pictures}
\end{figure}
\end{minipage}
\begin{minipage}{0.1\linewidth}
\hfill
\end{minipage}
\end{center}
\end{example}

As can be seen in the above experiments, the results achieved by the vanilla 3D method and the state-of-the-art 
warm 3D method are qualitatively comparable. 
Looking at the ideal early stopping rule ${\bar n}$, we see that they both perform very well in presence of impulse noise and Poisson noise (Examples \ref{Ex:saltandpepper L1L1}, \ref{Ex:SPL1TV} and \ref{Ex:PKLTV}), and quite well in presence of Gaussian noise and mixed Gaussian-impulse noise (Examples \ref{Ex:GL2TV} and \ref{Ex:GSPL1L2TV}).
Concerning the early stop $\hat n$ defined with the SURE estimator, it provided good reconstructions for Gaussian and impulse noise, but less satisfactory ones for the Poisson noise, and mixed Gaussian-impulse noise.
In the latter cases, the blur is removed but the image is contaminated by the noise, due to a late stopping.
This suggests that an appropriate stopping rule should be investigated for these specific noises.

We emphasize that the practical performance of warm 3D, its computational time, 
and the shape of the sequence $(\lambda_n)_{n\in\N}$, is crucially affected by the choice of $\eps_{wr}$.
For instance, in Examples \ref{Ex:saltandpepper L1L1} and \ref{Ex:SPL1TV}, $\eps_{wr}=10^{-5}$ and the number of iterations is around $700$, 
while in Examples \ref{Ex:GSPL1L2TV} and \ref{Ex:PKLTV},  $\eps_{wr}$ is taken as $10^{-4}$ but the number of iterations is larger (around $3700$).
We see then in the vanilla 3D method an advantage, which is its direct control on the complexity of the method, thanks to the a priori choice of $N_v$.
This can be of interest in practice, if one has a fixed computational budget.

\section{Conclusion}

In this paper we  propose and analyze the (3-D) algorithm, a new iterative regularization  method for solving ill-posed inverse problems,  and show very good performances  in practical imaging problems.  To the best of our knowledge this is the first iterative regularization scheme that allows to consider  general data fit terms and general regularizers, hence offering an alternative to standard Tikhonov approaches. 
The method is based on the forward-backward algorithm
applied to a dual problem in a diagonal fashion. 
The proposed framework encompasses in particular the warm restart technique often used  for  classical Tikhonov regularization. 

Our study opens many venues for future research.
For example,  our stability result appears to be suboptimal, since better result are known in special cases. It would  then be interesting to see if it can be improved.
Moreover, in our analysis we assume  a  solution of the linear inverse problem to exists, and it would be interesting to relax  this assumption. {Finally, considering convex, rather than strongly convex, regularization would be of interesting.

\section{Appendix}
\label{S:Annex}

\subsection{Proofs for Section \ref{S:Background and notation} }
\label{S:Annex1}
\begin{proof}[of Lemma \ref{L:gradient dual prox formula}]
By \cite[Theorem 18.15]{BauCom} $f^*$ is differentiable on $H_1$, and by \cite[Proposition 16.23]{BauCom},  $\nabla f^*=(\partial f)^{-1}$.
Let $x\in H_1$ and $u\in\partial f(x)$. Since $W$ is surjective and  ${\sigma} \Vert \cdot  - x'\Vert^2/2$ has full domain, 
it follows from~\cite[Proposition 16.42]{BauCom} that
\[
u\in W^* \partial J (Wx) + \sigma (x-x^\prime). 
\]
Orthogonality of $W$ implies
$\sigma^{-1}W u \in \sigma^{-1}\partial J(Wx) + W (x-x^\prime)$,
and hence
\[
 (\sigma^{-1}Wu+ Wx^\prime) \in (\sigma^{-1}\partial J + I) (Wx). 
 \]
The definition of proximity operator and  the orthogonality of $W$
yield
\[
x= W^*\prox_{\sigma^{-1} J} (\sigma^{-1} W u+Wx^\prime)=\nabla f^*(u).
\]
\end{proof}

\subsection{Proofs for Section \ref{S:the 3D method}}
\label{S:Appendix2}

We start computing the conditioning modulus of various data-fit function
terms and we establish the 
properties we use.

\begin{lemma}
\label{l:cond_moduli}
 Let  $d\in\N^*$, let $Y=\R^d$, and suppose that 
{\em(AD)} is satisfied. Then, the following hold. 
\begin{enumerate}
\item Suppose that $D_y=\|\cdot-y\|^p/p$ with $p\in\left]1,2\right]$.
Then its conditioning modulus satisfies
\[
(\forall t\in\R)\ \ m(t)=|t|^p/p, \ \text{ and } \ m^*(t)=|t|^{q}/q,
\]
where $q$ is the conjugate exponent of $p$.
\item Suppose $D_y=\|\cdot-y\|_1$. Then its conditioning modulus satisfies
\[
(\forall t\in\R)\ \ m(t)=|t|, \ \text{ and } \ m^*(t)=\delta_{[-1,1]}.
\]
\item Suppose $D_y=\alpha_1\|\cdot-y\|_1\# (\alpha_2/2)\|\cdot-y\|^2$ for some 
$(\alpha_1,\alpha_2)\in\R_{++}^2$.
Then its conditioning modulus satisfies, for every $t\in\R$,
\[
m(t)=\alpha_1 h_{\frac{\alpha_2}{\alpha_1}}(t) \quad h^*(t)=\alpha_1\delta_{[-\alpha_1,\alpha_1]}(t)+\frac{1}{2\alpha_2} t^2,
\]
where $h_{\frac{\alpha_2}{\alpha_1}}$ is the function defined in~\eqref{e:huber}.
\end{enumerate}
\end{lemma}
The proof of the above lemma is straightforward and it is omitted.
The computation of the conditioning modulus of the Kullback-Leibler divergence is more involved, 
and is done in the next lemma.

\begin{lemma}
\label{L:well conditionning of Kullback Leibler}
Let $y \in ]0,+\infty[^d$, and $D_y=\KL(y; \cdot)$.
\begin{enumerate}
	\item Let $c=d \Vert y \Vert_\infty$. 
	A conditioning modulus for $D_y$ is
	\[
	(\forall t\in \R)\ \  m(t):= \vert t \vert - c \ln \left( 1 + c^{-1} \vert t \vert \right).
	\]
	\item $m^*(t) =
	\begin{cases}
	 - c ( \vert t \vert + \ln ( 1 - \vert t \vert ) ) & \text{ if } t \in ]-1,1[, \\
	 +\infty & \text{ otherwise.}
	 \end{cases}$
	 \item $m(t) = \frac{1}{2c}t^2 + o(t^2)$ and $m^*(t) = \frac{c}{2} t^2+  o (t^2)$ for $t \to 0$.
\end{enumerate}
\end{lemma}

\begin{proof}
In this proof we use the notations of Example~\ref{R:data-fit function list}. 
We will consider, for all $\alpha >0$ and $t \in \R$,
\[
m_\alpha(t):=\kl(\alpha , \alpha + \vert t \vert ) =
 \alpha \ln \left(\frac{\alpha}{\alpha + \vert t \vert}\right) + \vert t \vert.
\]
According to \cite{BauBor01}, $m_\alpha \in \Gamma_0(\R)$.
Moreover, $\argmin m_\alpha=\{0\}$, so $t \mapsto m_\alpha(t)$ is an increasing function on $[0,+\infty[$.
For all $t \in \R$, $\alpha \mapsto m_\alpha(t)$ is decreasing, since
\[
\forall \alpha >0, \ \frac{\d}{\d \alpha} m_\alpha(t) = \ln \left( \frac{\alpha}{\alpha + \vert t \vert} \right) - \frac{\alpha}{\alpha + \vert t \vert} +1 \leq 0.
\]
Let us  start by showing a one dimensional analogue of \eqref{E:growth condition}:
\begin{equation}\label{wck1}
\forall \alpha >0, \forall \beta \in \R, \ m_\alpha(\vert \beta - \alpha \vert ) \leq \kl(\alpha,\beta).
\end{equation}
Note that when $\beta \leq0$, \eqref{wck1} is trivially satisfied because $\kl(\alpha,\beta)=+\infty$.
Moreover, if $\beta \geq \alpha $, we have by definition of $m_\alpha$ that $m_\alpha(\vert \beta - \alpha \vert) = \kl(\alpha,\beta)$.
So without loss of generality, we can assume that $\beta \in ]0,\alpha[$.
In this case, $m_\alpha(\vert \beta - \alpha \vert) = \KL(\alpha,\alpha + (\alpha -\beta))$.
Introduce now the function 
\[
\xi_\alpha : t \in [0,\alpha[ \  \longmapsto \kl(\alpha, \alpha - t) - \kl(\alpha, \alpha +t).
\]
It suffices to prove that $\xi_\alpha(t) \geq 0$ on $]0,\alpha[$ and then take  $t=\alpha - \beta \in ]0, \alpha[$ to obtain \eqref{wck1}.
To prove this, first observe that $\xi_\alpha(0) = 0$, and then observe that $\xi_\alpha$ is increasing on $]0,\alpha[$ by computing its derivative:
\[
\forall t \in ]0,\alpha[, \  \frac{\d}{\dt}\xi_\alpha(t) = \frac{2t^2}{\alpha^2 - t^2} \geq 0.
\]
Now that \eqref{wck1} is proved, let us prove item 1.
Start by considering $ y = (y_i)_{i\in\{1,...,d\}}$ and $x = (x_i)_{i\in\{1,...,d\}}$ in $ \R^d_{++}$, and let $y_\infty:=\max\{y_i \ | \ i \in \{1,...,d \} \}$.
Thanks to \eqref{wck1}, we can write
\begin{equation}
\KL(y,x)=\sum\limits_{i=1}^{d} \kl(y_i,x_i) \geq \sum\limits_{i=1}^{d} m_{y_i}(\vert x_i - y_i \vert).
\end{equation}
By using the fact that $\alpha \mapsto m_\alpha(t)$ is decreasing, we can bound the above estimate from below with
\begin{equation*}
\KL(y,x) \geq \sum\limits_{i=1}^{d} m_{y_\infty}(\vert x_i - y_i \vert).
\end{equation*}
Then, by using Jensen inequality applied to the convex function $m_{y_\infty}$, we deduce that
\begin{eqnarray*}
\frac{1}{d} \KL(y,x) & \geq & m_{y_\infty} \left( \sum\limits_{i=1}^{d} \frac{1}{d} \vert x_i - y_i \vert \right)
= m_{y_\infty}\left(\frac{1}{d} \Vert x_i - y_i \Vert_1\right).
\end{eqnarray*}
But an easy computation shows that $m_\alpha(t/d)=m_{d \alpha}(t)/d$, so that
\[
\frac{1}{d} \KL(y,x)  \geq \frac{1}{d} m_{dy_\infty} \left( \Vert x - y \Vert_1 \right).
\]
By observing the fact that $\Vert x- y  \Vert_1 \geq \Vert x-y \Vert$ and recalling that $m_{dy_\infty}$ is an increasing function, we finally 
obtain $\KL(y,x) \geq m_{dy_\infty}(\Vert x - y \Vert)$, which proves item 1.

We next prove 2, by computing the Fenchel conjugate of $m$.
Since $m(t)=m_{c}(t)=cm_1(t/c)$, we derive $m^*(t)=cm_1^*(t)$.
We then just have to compute
\[
m^*_1(t)=\sup\limits_{s \in \R} \ \eta_t(s), \text{ with } \eta_t(s) = \ st - \vert s \vert + \ln(1+ \vert s \vert),
\]
for every $s\in\R$.
If $t=0$, we see from $\eta_0 \leq 0$ and $\eta_0(0)=0$ that $\eta_0$ is maximized at $0$, whence $m_1^*(0)=0$.
If $t \in\left ]0,1\right[$, we have 
\[
\frac{\d}{\d s} \eta_t(s) = t - 1 + \frac{1}{1+s} \ \text{ on } \R_{++},
\]
which is zero at $s=\frac{t}{1-t} \in\R_{++}$.
Since $\eta_t$ is concave, this means that it is maximized there, whence $m_1^*(t)= -t - \ln(1-t)$.
If $t \in ]-1,0[$, the same argument shows that $\eta_t$ is maximized at $s=\frac{t}{1+t}\in\left]-\infty,0\right[$, leading in this case to $m_1^*(t) = t - \ln(1+t)$.
From all of this, we see that $m_1^*(t)=\vert t \vert - \ln(1 - \vert t \vert)$ on $]-1,1[$.
Moreover, $m_1^*(t)$ tends to $+\infty$ when $\vert t \vert \to 1$, so from the convexity of $m_1^*$ we deduce that $m_1^*(t) \equiv + \infty$ when $\vert t \vert \geq 1$.

Item 3 is a simple consequence of item 2 and the classic Taylor expansion
\[
\ln(1+t) = t -\frac{1}{2}t^2 + o(t^2) \ \text{ when } t \to 0.
\]
\end{proof}

\subsection{Proofs for Section  \ref{S:Convergence}}

\begin{lemma}\label{L:Fenchel local inequality}
Let $f,g \in \Gamma_0(\R)$ and let $a \in ]0,+\infty]$. Suppose that,   $f\leq g$ in $\left]-a,a\right[$.
If $\text{\rm argmin} \ g = \{0 \}$, then there exists $\eps \in\R_{++}$ such that 
\[
\forall t \in ]-\eps,\eps[, \ f^*(t) \geq g^*(t).
\]
\end{lemma}

\begin{proof}
First, note that $\partial g^*(0)=\text{\rm argmin} \ g = \{ 0 \}$ implies that $0 \in \inte \dom g^*$ (see \cite[Prop. 11.12 \& 14.16]{BauCom}).
Let $\eps_1 >0$ be such that $]-\eps_1,\eps_1[ \subset \inte \dom g^*$.
Then $\partial g^*$ is nonempty on  $\left]-\eps_1,\eps_1\right[$, and we can define a function $\eta\colon\left]-\eps_1,\eps_1\right[\to \R$, such that $\eta(t) \in \partial g^*(t)$ for all $t \in ]-\eps_1,\eps_1[$.
We derive from \cite[Th. 17.31]{BauCom} and \cite[Prop. 17.36]{BauCom}  that $\eta$ is continuous at zero.
Now we are ready to prove the desired inequality.
By using the Fenchel-Young inequality successively on $g$ and $f$, we write for all $t \in ]-\eps_1,\eps_1[$:
\[
g(\eta(t)) + g^*(t)  =  t \eta(t)  \leq  f( \eta(t)) + f^* (t).
\]
From the continuity of $\eta$ at zero, we infer the existence of some $\eps_2 \in ]0, \eps_1[$ such that $\vert t \vert < \eps_2 \Rightarrow \vert \eta(t) \vert < a$.
We deduce from our assumption that $f(\eta(t)) \leq g(\eta(t))$ holds for any $t \in ]-\eps_2,\eps_2[$, and the conclusion follows.
\end{proof}

\begin{proof}[of Lemma \ref{L:primal-dual value-iterate bound}]
Let us start by proving $\Rightarrow$.
Assume that there exists some $u^\dagger \in \argmin d$.
Define $\tilde x := \nabla f^*(-A^*u^\dagger)$, which is equivalent to say that $0 \in A^* u^\dagger + \partial f(\tilde x)$.
Using Fermat's rule on $d$ at $u^\dagger$, we obtain that $A\nabla f^*(-A^*u^\dagger) \in \partial g^*(u^\dagger)$, which is equivalent to $u^\dagger \in \partial g(A \tilde x)$.
So we have 
\begin{equation}\label{pde1}
0 \in \partial f(\tilde x) + A^* \partial g(A \tilde x),
\end{equation}
where a classic result \cite[Corollary 3.31]{Pey} shows that it implies $0 \in \partial(f + g \circ A)(\tilde x)$.
Because of the strong convexity of $R$, this is a sufficient condition for $\tilde x$ to be the unique solution of \eqref{e:P}, $x^\dagger$.
It follows from \eqref{pde1} that we have  $0 \in \partial f(x^\dagger) + A^* \partial g^*(Ax^\dagger)$.

Now we turn on proving $\Leftarrow$, and we assume that there exists some $v \in \partial g(Ax^\dagger)$ such that $-A^*v \in \partial f(x^\dagger)$.
Equivalently, $\nabla f^*(-A^*v)=x^\dagger$ holds, and this implies that $A \nabla f^*(-A^*v) =Ax^\dagger$.
Since $v \in \partial g(Ax^\dagger) \Leftrightarrow A x^\dagger \in \partial g^*(v)$, we deduce that $A\nabla f^*(-A^* v) \in \partial g(v)$, which is a sufficient condition for $v$ to be a minimizer of $d$.

\medskip

Now we end the proof by proving \eqref{E:primal-dual value-iterate bound}.
Let $u\in Y, x:=\nabla f^*(-A^*u)$, $z \in \partial g^*(u)$, and we also take $z^\dagger:=Ax^\dagger$.
Let $u^\dagger \in \argmin d$, so that by using a similar argument as above, we can write $x^\dagger = \nabla f^*(-A^*u^\dagger)$, and deduce that $z^\dagger \in \partial g^*(u^\dagger)$.
Define the Lagrangian $$L(x',z',u'):=f(x') + g(z') + \langle u',Ax'-z'\rangle,$$ and compute
\begin{eqnarray*}
 L(x^\dagger,z^\dagger,u)-L(x,z,u) 
  = f(x^\dagger) - f(x) - \langle -A^*u, x^\dagger - x \rangle  
 + g(z^\dagger) - g(z) - \langle u,z^\dagger - z \rangle .
\end{eqnarray*}
By using the fact that $f$ is $\sigma$-strongly convex with $-A^*u \in \partial f(x)$ and that $g$ is convex with $u \in \partial g(z)$, we deduce that
\begin{equation}\label{pde2}
L(x^\dagger,z^\dagger,u)-L(x,z,u)  \geq \frac{\sigma}{2}\Vert x- x^\dagger \Vert^2.
\end{equation}
On the one hand, using again $-A^*u \in \partial f(x)$, $u \in \partial g(z)$ together with the Fenchel-Young theorem gives us
\begin{equation}\label{pde3}
L(x,z,u) = -g^*(u) - f^*(-A^*u) = -d(u).
\end{equation}
On the other hand, using $z^\dagger=Ax^\dagger$, $-A^*u^\dagger \in \partial f(x^\dagger)$ and $u^\dagger \in \partial g(z^\dagger)$ together with the Fenchel-Young theorem leads to
\begin{eqnarray}
L(x^\dagger,z^\dagger,u)  & = & f(x^\dagger) + g(Ax^\dagger)\label{pde4} \\
&=& -f^*(-A^* u^\dagger) - g^*(u^\dagger) \nonumber \\
& =&  -d(u^\dagger)=-\inf d. \nonumber
\end{eqnarray}
The result follows then from \eqref{pde2}, \eqref{pde3} and \eqref{pde4}.
\end{proof}

\subsection{Proofs for Section \ref{S:regularization}}
Here we prove the estimations claimed in Example \ref{Ex:stability discrepancy functions}.
\begin{lemma}
\label{L:prox additive form noise}
Let $H$ be a Hilbert space, $G\in \Gamma_0(H)$ with $\text{\em argmin } G=\{0\}$.
Let $(y_1,y_2)\in H^2$, and $\phi_{i}:=G(\cdot - y_i)$ for $i\in\{1,2\}$.
Then
\begin{equation}
\sup\limits_{\alpha >0} \ \sup\limits_{u \in H} \ \Vert \prox_{\alpha\phi_{1}}(u) - \prox_{\alpha\phi_{2}}(u) \Vert = \Vert y_1 - y_2 \Vert.
\end{equation}
\end{lemma}

\begin{proof}
Let $\alpha \in\R_{++}$, and let $u \in H$. 
By using \cite[Table 1.i]{ComPes11}, we can write for $i\in\{1,2\}$ that 
\[
\prox_{\alpha \phi_i}(u)=y_i + \prox_{\alpha G}(u-y_i).
\]
Then it follows that
\begin{eqnarray*}
\Vert \prox_{\alpha\phi_{1}}(u) - \prox_{\alpha\phi_{2}}(u) \Vert 
 =  \Vert (\Id - \prox_{\alpha G})(u-y_1) - (\Id - \prox_{\alpha G})(u-y_2)\Vert.
\end{eqnarray*}
By using first the firm non expansiveness of the proximity operator \cite[Prop. 12.27]{BauCom}, one directly obtains
\begin{equation}
\sup\limits_{\alpha >0} \ \sup\limits_{u \in H} \ \Vert \prox_{\alpha\phi_{1}}(u) - \prox_{\alpha\phi_{2}}(u) \Vert \leq \Vert y_1 - y_2 \Vert.
\end{equation}
To achieve the equality in the inequality above, observe that $\prox_{\alpha G}(u-y_i)$ converges strongly to zero when $\alpha \to + \infty$, by using \cite[Lem. 1]{Bru74} and $\argmin G=\{0\}$.
This implies that 
$$\forall u \in H, \ \Vert \prox_{\alpha\phi_{1}}(u) - \prox_{\alpha\phi_{2}}(u) \Vert \overset{\alpha \to + \infty}{\longlongrightarrow} \Vert y_1 -  y_2 \Vert. $$ 
\end{proof}

\begin{lemma}
\label{L:prox KL form noise}
Let $y_1,y_2 \in \R^d_{++}$, and $\phi_i:=\KL(y_i, \cdot)$ for $i\in\{1,2\}$.
Let  $\alpha\in\R_{++}$. Then
 \[ 
\sup\limits_{u \in \R^d} \ \Vert \prox_{\alpha\phi_1}(u) - \prox_{\alpha\phi_2}(u) \Vert = \sqrt{\alpha} \Vert \sqrt{y_1}-\sqrt{y_2} \Vert,
 \]
 where $\sqrt{y_i}$ shall be understood component-wise.
\end{lemma}

\begin{proof}
Let $u=(u_j)_{j\in\{1,...,d\}} \in \R^d$, and let us denote $y_i=(y_{ij})_{j\in\{1,...,d\}}$ for $i\in\{1,2\}$.
The proximity operator of $\alpha \phi_i$ at $u$ is defined component-wise by (see~\cite{ChaComPesWaj07,DupFadSta12})
\[
\left( \prox_{\alpha \phi_i}(u) \right)_j=\frac{1}{2}\left( u_j - \alpha + \sqrt{(u_j - \alpha)^2 + 4 \alpha y_{ij}} \right).
\]
Then,
\begin{eqnarray*}
\Vert \prox_{\alpha \phi_1}(u) - \prox_{\alpha \phi_2}(u) \Vert^2 
 = \frac{1}{4} \sum\limits_{j=1}^{d} \left\vert \sqrt{(u_j - \alpha)^2 + 4 \alpha y_{1j}} - \sqrt{(u_j - \alpha)^2 + 4 \alpha y_{2j}} \right\vert^2.
\end{eqnarray*}
Let $(a,b) \in [0,+\infty]^2$, and define
 $$\xi : t \in ]0,+\infty[ \mapsto \left\vert \sqrt{t + a} - \sqrt{t + b } \right\vert^2.$$ 
Since $\xi$ is decreasing on $\R_{+}$,
by considering $u_j=\alpha$ for all $j\in\{1,\ldots,d\}$,
\begin{eqnarray*}
 \sup\limits_{u_j \in \R} \left\vert \sqrt{(u_j - \alpha)^2 + 4 \alpha y_{1j}} - \sqrt{(u_j - \alpha)^2 + 4 \alpha y_{2j}} \right\vert^2 
 = \left\vert  \sqrt{ 4 \alpha y_{1j}} - \sqrt{ 4 \alpha y_{2j}} \right\vert^2= 4 \alpha \vert \sqrt{y_{1j}} - \sqrt{y_{2j}} \vert^2.
\end{eqnarray*}
We  then conclude that
\begin{eqnarray*}
\sup\limits_{u \in \R^d} \ \Vert \prox_{\alpha \phi_1}(u) - \prox_{\alpha \phi_2}(u) \Vert^2 
 =   \sum\limits_{j=1}^{d} \alpha \vert \sqrt{y_{1j}} - \sqrt{y_{2j}} \vert^2 = \alpha \Vert \sqrt{y_1} - \sqrt{y_2} \Vert^2.
\end{eqnarray*}
\end{proof}




\end{document}